\documentclass[onefignum,onetabnum]{siamart220329}
\usepackage{cite}

\usepackage{lipsum}
\usepackage{amsfonts}
\usepackage{amsmath,booktabs,ctable,threeparttable}
\usepackage{amssymb,amsfonts,boxedminipage}
\usepackage{graphicx,epstopdf} 

\usepackage[caption=false]{subfig} 
\usepackage{algorithm}
\usepackage{algpseudocode}
\usepackage{qbordermatrix}
\usepackage{blkarray}
\usepackage{lineno}
\usepackage{xcolor}
\ifpdf
  \DeclareGraphicsExtensions{.eps,.pdf,.png,.jpg}
\else
  \DeclareGraphicsExtensions{.eps}
\fi

\newsiamremark{remark}{Remark}
\newsiamremark{hypothesis}{Hypothesis}

\crefname{hypothesis}{Hypothesis}{Hypotheses}
\newsiamthm{claim}{Claim}

\newtheorem{assumption}{Assumption}

\headers{JBD with inaccurate inner iteration}{H. Li}

\title{The joint bidiagonalization of a matrix pair with inaccurate inner iterations 
	\thanks{Submitted to the editors DATE.
\funding{This work was supported in part by the National Natural Science Foundation of China under Grant No. 3192270206.}
}}

\author{Haibo Li\thanks{Computing System Optimization Lab, Huawei Technologies, 100094 Beijing, China, and Institute of Computing Technology, Chinese Academy of Sciences, 100190 Beijing, China. 
  (\email{haibolee1729@gmail.com}).}
}

\usepackage{amsopn}

\ifpdf
\hypersetup{
  pdftitle={The joint bidiagonalization of a matrix pair with inaccurate inner iterations},
  pdfauthor={H. Li}
}
\fi

\begin{document}

\maketitle

\begin{abstract}
The joint bidiagonalization (JBD) process iteratively reduces a matrix pair $\{A,L\}$ to two bidiagonal forms simultaneously, which can be used for computing a partial generalized singular value decomposition (GSVD) of $\{A,L\}$. The process has a nested inner-outer iteration structure, where the inner iteration usually can not be computed exactly. In this paper, we study the inaccurately computed inner iterations of JBD by first investigating influence of computational error of the inner iteration on the outer iteration, and then proposing a reorthogonalized JBD (rJBD) process to keep orthogonality of a part of Lanczos vectors. An error analysis of the rJBD is carried out to build up connections with Lanczos bidiagonalizations. The results are then used to investigate convergence and accuracy of the rJBD based GSVD computation. It is shown that the accuracy of computed GSVD components depend on the computing accuracy of inner iterations and condition number of $(A^T,L^T)^T$ while the convergence rate is not affected very much. For practical JBD based GSVD computations, our results can provide a guideline for choosing a proper computing accuracy of inner iterations in order to obtain approximate GSVD components with a desired accuracy. Numerical experiments are made to confirm our theoretical results.
\end{abstract}



\section{Introduction}
For many matrix computation algorithms, a basic routine is reducing a matrix to a structured one using a series of orthogonal transformations. For example, the first step of QR algorithm for eigenvalue decomposition is reducing a matrix to a Hessenberg form \cite{Watkins2008}, while for SVD computation is reducing a matrix to a bidiagonal form \cite{Golub1965}. For large-scale matrices, the reduction via direct orthogonal transformations is very expensive, thus a Lanczos-type iterative process is a common choice, e.g. the symmetric Lanczos process for symmetric eigenvalue problem, the Lanczos bidiagonalization for singular value decomposition(SVD) and so on \cite{Lanczos1950,Paige1972, Larsen1998,Bjorck1988,Elden2004}. When it comes to large-scale matrix pairs, one method is implicitly transform it to a standard single matrix problem, such as the shift-and-invert Lanczos method for generalized symmetric eigenvalue decomposition \cite{Ericsson1986,Hetmaniuk2006}. Another popular choice is the Jacobi-Davidson method for generalized eigenvalue/singular value decomposition, where an inner correction equation need to be solved \cite{Sleijpen2000,Arbenz2004,Hochstenbach2009}. These algorithms often have the structure of nested inner-outer iterations, and the inner iteration contains a large-scale matrix computation problem that should be computed iteratively. 

In this paper we focus on another inner-outer iterative algorithm that can reduce a matrix pair to two bidiagonal forms simultaneously, which is called the joint bidiagonalization process. This algorithm was first proposed by Zha \cite{Zha1996} for computing a partial GSVD of a large-scale matrix pair $\{A,L\}$ with $A\in\mathbb{R}^{m\times n}$ and $L\in\mathbb{R}^{p\times n}$. It was later adapted in \cite{Kilmer2007} to solve large-scale linear ill-posed problems with general-form regularization \cite{Hansen1998,Hansen2010}. Let $C=(A^T, L^T)^T$. Consider the following compact QR factorization:
\begin{equation}\label{qr}
	C=\begin{pmatrix}
		A \\
		L
	\end{pmatrix} = QR =
	\begin{pmatrix}
		Q_{A} \\
		Q_{L}
	\end{pmatrix}R ,
\end{equation}
where $Q \in \mathbb{R}^{(m+p)\times n}$ is column orthonormal with $Q_{A}\in\mathbb{R}^{m\times n}$, $Q_{L}\in\mathbb{R}^{p\times n}$, and $R\in \mathbb{R}^{n\times n}$ is upper triangular. Zha's method generates two upper bidiagonal matrices by implicitly applying the upper Lanczos bidiagonalization to both $Q_A$ and $Q_L$. In contrast, the method proposed in \cite{Kilmer2007} implicitly uses the lower and upper Lanczos bidiagonalizations to reduce $Q_A$ and $Q_L$ to lower and upper bidiagonal matrices, respectively. To keep the presentation short, we focus on Zha's joint bidiagonalization from now on, since it is more convenient for GSVD computation using upper bidiagonal matrices.

By choosing the same initial vectors $v_1=\hat{v}_1$, the two Lanczos bidiagonalization can reduce $Q_A$ and $Q_L$ to the following two bidiagonal matrices:
\begin{equation*}
	B_{k}=\begin{pmatrix}
		\alpha_{1} &\beta_{1} & & \\
		&\alpha_{2} &\ddots & \\
		& &\ddots &\beta_{k-1} \\
		& & &\alpha_{k}
	\end{pmatrix}\in  \mathbb{R}^{k\times k}, \ \
	\widehat{B}_{k}=\begin{pmatrix}
		\hat{\alpha}_{1} &\hat{\beta}_{1} & & \\
		&\hat{\alpha}_{2} &\ddots & \\
		& &\ddots &\hat{\beta}_{k-1} \\
		& & &\hat{\alpha}_{k}
	\end{pmatrix}\in  \mathbb{R}^{k\times k}.
\end{equation*}
Meanwhile they generate two groups of orthonormal vectors $\{u_1,\dots,u_k\}$, $\{v_1,\dots,v_k\}$ corresponding to $Q_A$ and two groups of orthonormal vectors $\{\hat{u}_1,\dots,\hat{u}_k\}$, $\{\hat{v}_1,\dots,\hat{v}_k\}$ corresponding to $Q_L$. It is shown in \cite{Zha1996} the basic relation 
\begin{equation*}
	\hat{v}_{i} = (-1)^{i-1}v_{i} , \ \ 
	\hat{\alpha}_{i}\hat{\beta}_{i} = \alpha_{i}\beta_{i} 
\end{equation*}
holds. Based on this property, the two Lanczos bidiagonalization can be jointed without an explicit QR factorization of $C$. Denote by $\mathcal{P}_Q$ the projection operator onto the subspace spanned by columns of $Q$, which can be written in matrix form as $\mathcal{P}_Q=QQ^{T}$ since $Q$ has orthonormal columns. Then we have the following joint bidiagonalization algorithm:
\begin{align*}
	&\mbox{\textbf{Zha's joint bidiagonalization of}\ $\{A,L\}$ } \\
	& \mbox{Choose nonzero} \ s\in\mathbb{R}^{n}, 
	\mbox{set} \ \tilde{v}_1 = Cs/\|Cs\| \\
	& \mbox{for} \ \ i = 1,2, \dots, k \\
	& \ \ \ \ \ \ \alpha_{i}u_{i}=\tilde{v}_{i}(1:m)-\beta_{i-1}u_{i-1} \\
	& \ \ \ \ \ \ \beta_{i}\tilde{v}_{i+1}=
	\mathcal{P}_Q\begin{pmatrix}
		u_{i} \\
		0_{p}
	\end{pmatrix}-\alpha_{i}\tilde{v}_{i}  \\
	& \ \ \ \ \ \ \hat{\alpha}_{i}\hat{u}_{i}=
		(-1)^{i-1}\tilde{v}_{i}(m+1:m+p)-\hat{\beta}_{i-1}\hat{u}_{i-1} \\
	& \ \ \ \ \ \ \hat{\beta}_{i}=(\alpha_{i}\beta_{i})/\hat{\alpha}_{i}  \\
	& \mbox{end}
\end{align*}
where $\beta_0=\hat{\beta}_0=0$ and $\|u_i\|=\|\hat{u}_i\|=\|\tilde{v}_i\|=1$. In this paper, $\|\cdot\|$ always means 2-norm of a matrix or vector. Let $\tilde{u}_i=(u_{i}^{T}, 0_{p}^T)^T$, where $0_{p}\in\mathbb{R}^{p}$ denotes the $p$-dimensional zero vector. Note that $\mathcal{P}_Q\tilde{u}_i$ can be computed iteratively by the relation $\mathcal{P}_Q\tilde{u}_i=C\tilde{z}_i$ with 
\begin{equation}\label{ls}
	\tilde{z}_i=\arg\min_{\tilde{z}\in \mathbb{R}^n}
	\left\|C\tilde{z}-\tilde{u}_i\right\|,
\end{equation}
where this least squares problem can be solved by an iterative method. The Lanczos vectors $v_i$ and $\hat{v}_i$ can be implicitly obtained from $\tilde{v}_i$. 

The joint bidiagonalization of a matrix pair is a generalization of Lanczos bidiagonalization of a single matrix. In exact arithmetic, the $k$-step JBD reduces $A$ and $L$ to small-scale upper bidiagonal matrices $B_{k}$ and $\widehat{B}_{k}$, and the reduction processes of $A$ and $L$ are equivalent to the upper Lanczos bidiagonalizations of $Q_A$ and $Q_L$. Therefore, $B_{k}$ and $\widehat{B}_{k}$ are the Ritz-Galerkin projections of $Q_{A}$ and $Q_L$ on proper Krylov subspaces. This makes the JBD process useful for designing efficient algorithms for large sparse matrix pair problems. For example, some extreme generalized singular values and vectors of $\{A,L\}$ can be approximated by using the SVD of $B_{k}$ or $\widehat{B}_{k}$ \cite{Zha1996,JiaLi2019}; the linear ill-posed problems with general-form Tikhonov regularization $\min_x\left\{\|Ax-b\|^2+\lambda^2\|Lx\|^2\right\}$ can be solved iteratively by solving small-scale problems containing $B_{k}$ and $\widehat{B}_{k}$ at each iteration \cite{Kilmer2007,JiaYang2020}.

For a practical implementation of JBD, there are some issues must to be addressed. For example, as was pointed out by Zha at the end of \cite{Zha1996}, the computed Lanczos vectors $u_i$, $\hat{u}_i$ and $\tilde{v}_i$ in finite precision arithmetic quickly lose orthogonality, which will cause a delay of convergence for approximating GSVD components and the appearance of spurious copies of approximations \cite{JiaLi2019}. Therefore, a proper reorthogonalization strategy should be included in JBD to maintain some level of orthogonality to preserve regular convergence for GSVD computations. This issue has been studied by Jia and Li \cite{JiaLi2021,JiaLi2019}, where they investigate the semiorthogonalization strategy for JBD and propose an efficient partial reorthogonalization technique that can keep regular convergence behavior of computed quantities. Another issue is the increasing computation and storage cost due to the gradually expanding Krylov subspaces, especially when reorthogonalization is exploited, since all Lanczos vectors must be kept throughout the computation. Recently, Alvarruiz \textit{et al.} \cite{Alvar2022} develop a thick restart technique for JBD to compute a partial GSVD, which can keep the size of Krylov basis bounded, thus the storage and computation cost can be further saved. 

However, it should be pointed out that the above researchers do not take into consideration the inaccurate computation of $QQ^T\tilde{u}_i$. The JBD process has the structure of nested inner-outer iterations where the Lanczos bidiagonalization is the outer iteration while an iterative solver for \cref{ls} plays the role of inner iteration. The overall computational costs of the algorithm is proportional to the overall number of inner iterations, and thus a bottleneck is that iteratively solving a large-scale least squares problem at each outer iteration may be very costly, especially when the solution accuracy is high. Numerical experiments have shown that the inaccuracy in forming $\mathcal{P}_Q\tilde{u}_i$ does limit the final accuracy of computed GSVD components \cite{Zha1996}, thus this issue is important for the JBD based GSVD computation. On the other hand, for the JBD based regularization algorithms for discrete ill-posed problems with general-form regularization, it is numerically shown that the inner least squares problem need not to be solved with very high accuracy \cite{Kilmer2007,JiaYang2020}, e.g., for the LSQR solver for \cref{ls}, the default stopping tolerance $\mathtt{tol}=10^{-6}$ for iteration is often enough to obtain a final regularized solution without loss of accuracy. This means for ill-posed problems the inner least squares problems may be solved with considerably relaxed accuracy, thus the overall efficiency can be improved.

In this paper, we study the influence of inaccuracy of inner iterations on the behavior of the JBD algorithm, with an emphasis on the effect on the convergence and accuracy of computed GSVD components. The main contributions are:
\begin{itemize}
	\item For a commonly used stopping criterion for iteratively solving \cref{ls}, we investigate the influence of computational error of the inner iteration on the outer iteration. It reveals that when the inner iteration is inaccurately computed, the outer iteration is not equivalent to the Lanczos bidiagonalization of $Q_A$ and $Q_L$ any longer but with a perturbation of order $\mathcal{O}(\kappa(C)\tau)$, where $\kappa(C)$ is the condition number of $C$ and $\tau$ describes the solution accuracy of \cref{ls}. A couple of recursive relations that describes the loss of orthogonality of the computed vectors is also established.
	\item We propose a reorthogonalized JBD (rJBD) process which maintains the orthogonality of $\tilde{v}_i$. We perform an error analysis on the $k$-step rJBD to establish connections between rJBD and the two Lanczos bidiagonalizations of $Q_A$, where $\kappa(C)\tau$ plays a crucial role for obtaining some useful upper bounds.
	\item The results of the above error analysis are used to investigate the convergence and accuracy of the computed GSVD components by rJBD. We show that the approximate generalized singular values can only reach an accuracy of order $\mathcal{O}(\kappa(C)\tau)$ by the SVD of $B_k$ or $\widehat{B}_k$, and the accuracy of approximate right generalized singular vectors depends not only on the value of $\kappa(C)\tau$ but also on the gap between generalized singular values. Besides, it can be shown that regular convergence rate of approximate generalized singular values and right vectors can be kept for the rJBD based GSVD computation. 
\end{itemize}

Our results can theoretically explain some numerically observed phenomena of JBD in \cite{Zha1996}. For example, we theoretically demonstrate that when $\mathcal{P}_Q\tilde{u}_i$ is computed with lower accuracy and $C$ is more ill-conditioned, the orthogonality is lost at an earlier stage. We also give a theoretical explanation for the numerically conclusion in \cite{Zha1996} that convergence rates for the approximate GSVD components are not affected very much while the final accuracy does depend on $\kappa(C)$ and the computing accuracy of inner iterations. For practical JBD based GSVD computations, our results can provide a guideline for choosing the computing accuracy of inner iterations in order to obtain approximate GSVD components with a desired accuracy.

The paper is organized as follows. In \Cref{sec2}, we review some basic properties of the joint bidiagonalization and GSVD of a matrix pair. In \Cref{sec3}, we investigate the inaccurately computed inner iterations and propose a reorthogonalized JBD (rJBD) process to keep orthogonality of $\widetilde{V}_k$. An error analysis is carried out in \Cref{sec4} to build up connections between the rJBD process and Lanczos bidiagonal reductions of $Q_A$ and $Q_L$. The convergence and accuracy of approximate GSVD components computed by rJBD is investigated in \Cref{sec5}. Numerical experiment results for indicating our theory are in \Cref{sec6}, and some concluding remarks follow in \Cref{sec7}.

Throughout the paper, we denote by $0_{k}$ the $k$-dimensional zero column vector, and by $I_{k}$ and $0_{k\times l}$ the identity matrix of order $k$ and zero matrix of order $k\times l$, respectively. The subscripts are omitted when there is no confusion. We use $e_{i}^{(k)}$ to denote the $i$-th column of $I_k$ and $\mathcal{R}(M)$ to denote the range space of matrix $M$.

\section{Joint bidiagonalization and GSVD of a matrix pair}\label{sec2}
Although not computing $Q_A$ or $Q_L$ explicitly, the joint bidiagonalization process proposed in \cite{Zha1996} is based on the application of upper Lanczos bidiagonalization processes to $Q_A$ and $Q_L$, respectively. In exact arithmetic, the $k$-step joint bidiagonalization reduces $A$ and $L$ to two upper bidiagonal matrices $B_{k}$ and $\widehat{B}_{k}$, and it generates three groups of column orthonormal matrices $U_{k}=(u_1,\dots,u_{k})$, $\widehat{U}_{k}=(\hat{u}_1,\dots,\hat{u}_{k})$ and $\widetilde{V}_{k+1}=(\tilde{v}_{1},\cdots,\tilde{v}_{k+1})$. Meanwhile, it follows that $\tilde{v}_i$ is in $\mathcal{R}(Q)$, and thus we can write it as $\tilde{v}_{i}=Qv_{i}$ with $v_i\in\mathbb{R}^{n}$.
The $k$-step JBD process can be written in matrix form:
\begin{align}
	& (I_{m},0_{m\times p})\widetilde{V}_{k}=U_{k}B_{k} \label{jbd1} , \\
	& QQ^{T}
	\begin{pmatrix}
		U_{k} \\
		0_{p\times k}
	\end{pmatrix}
	=\widetilde{V}_{k}B_{k}^{T}+\beta_{k}\tilde{v}_{k+1}(e_{k}^{(k)})^T \label{jbd2} ,  \\
	& (0_{p\times m},I_{p})\widetilde{V}_{k}P=\widehat{U}_{k}\widehat{B}_{k} \label{jbd3} ,
\end{align}
where $P=\mbox{diag}(1,-1,\dots ,(-1)^{k-1})\in\mathbb{R}^{k\times k}$. Let $\hat{v}_i=(-1)^{i-1}v_i$. Using the relation 
\begin{equation}\label{B_Bh}
	B_{k}^{T}B_{k}+\bar{B}_{k}^{T}\bar{B}_k=I_{k}
\end{equation}
proved in \cite{Zha1996}, where $\bar{B}_k=\widehat{B}_{k}P$, one can deduce from \cref{jbd1,jbd2,jbd3} the following matrix form relations:
\begin{align}
	&Q_AV_k=U_{k}B_k,\ \ Q_A^TU_{k}=V_kB_k^T+\beta_{k}v_{k+1}(e_{k}^{(k)})^T,\label{2.3}\\
	&Q_L\widehat{V}_k=\widehat{U}_k\widehat{B}_k,\ \ Q_L^T\widehat{U}_k
	=\widehat{V}_k\widehat{B}_k^T+\hat{\beta}_k\hat{v}_{k+1}(e_{k}^{(k)})^T,\label{2.4}
\end{align}
where $V_k=(v_1,\ldots,v_k)$ and $\widehat{V}_k=(\hat{v}_1,\ldots,\hat{v}_k)$ are column orthonormal. Therefore, the process of computing $U_{k}$, $V_{k+1}$ and $B_{k}$ is actually the upper Lanczos bidiagonalization of $Q_{A}$, while the process of computing $\widehat{U}_{k}$, $\widehat{V}_{k+1}$ and $\widehat{B}_{k}$ is the upper Lanczos bidiagonalization of $Q_{L}$.

The generalized singular value decomposition (GSVD) of a matrix pair was introduced by Van Loan \cite{Van1976}, with subsequent additional developments by Paige and Saunders \cite{Paige1981}. The following description of GSVD is based on the CS decomposition \cite[§2.5.4]{Golub2013}, where the compact QR factorization of $C$ is defined as \cref{qr}. 
\begin{theorem}[CS decomposition]\label{thm_cs}
  Suppose $\mathrm{Rank}(C)=r$. For the column orthonormal matrix $Q$,  the CS decomposition of $\{Q_{A}, Q_{L} \}$ is
   \begin{equation*}
   	\begin{pmatrix}
   		Q_A \\Q_L
   	\end{pmatrix} =
   \begin{pmatrix}
   	P_A & \\
   	&  P_L
   \end{pmatrix}
   \begin{pmatrix}
   	C_A \\ S_L
   \end{pmatrix}
	W^T ,
   \end{equation*}
	where
	\begin{displaymath}
		C_A =
		\bordermatrix*[()]{%
			\Sigma_{A}, & 0 & m \cr
			r &   n-r   \cr
		} \ , \ \ 
		S_L = \bordermatrix*[()]{%
			\Sigma_{L}, & 0 & p \cr
			r &   n-r   \cr
		} 
	\end{displaymath}
	with
	\begin{displaymath}
		\Sigma_{A} =
		\bordermatrix*[()]{%
			I_{q}  &  &  & q \cr
			&  C_{l}  &  & l \cr
			&  & O  & m-q-l \cr
			q & l & r-q-l
		} , \ \
		\Sigma_{B} =
		\bordermatrix*[()]{%
			O  &  &  & p-r+q \cr
			&  S_{l}  &  & l \cr
			&  & I_{t}  & r-q-l \cr
			q & l & r-q-l
		} 
	\end{displaymath}
	satisfying $C_{A}^{T}C_A+S_{L}^{T}S_L=I_{n}$, and $P_{A}\in \mathbb{R}^{m\times m}$, $P_{L}\in \mathbb{R}^{p\times p}$ and $W\in\mathbb{R}^{n\times n}$ are orthogonal matrices.
 \end{theorem}

If we write $C_{l}=\mbox{diag}(c_{q+1}, \dots, c_{q+l})$ with $c_{q+1}\geq \cdots
\geq c_{q+l}>0$ and $S_{l}=\mbox{diag}(s_{q+1}, \dots, s_{q+l})$ with $0<s_{q+1}\leq \cdots \leq s_{q+l}$, then $c_{i}^{2}+s_{i}^{2}=1, \ i=q+1, \dots, q+l$, and the generalized singular values of $\{A, L\}$ are
\[\underbrace{\infty, \dots, \infty}_{q}, \ \
\underbrace{c_{q+1}/s_{q+1}, \dots, c_{q+l}/s_{q+l}}_{l}, \ \ 
\underbrace{0, \dots, 0}_{t},\]
where $t=r-q-l$. To ease the presentation, we always assume that $\{A,L\}$ is regular, i.e. $\mathrm{Rank(C)}=n$, which is called a Grassmann matrix pair \cite{Sun1983a,LiRC1993}. Discussions about GSVD of a nonregular matrix pair can be found in \cite{Paige1981,Paige1984,Sun1983b}. For the regular $\{A, L\}$, it follows that $R$ is nonsingular and the GSVD is
\begin{equation}\label{2.2}
	A = P_{A}C_AX^{-1} , \ \  L = P_{L}S_LX^{-1}
\end{equation}
with $X=R^{-1}W\in\mathbb{R}^{n\times n}$. The columns of $P_A$, $P_L$ and $X$ are called generalized singular vectors.

By \cref{jbd1} and \cref{jbd3}, the $k$-step JBD process satisfies
\begin{equation}\label{jbd4}
AZ_{k} = U_{k}B_{k} , \ \ LZ_{k}=\widehat{U}_{k}\bar{B}_{k} ,
\end{equation}
where $Z_{k}=R^{-1}V_{k}=(z_{1},\dots,z_{k})$. Therefore, $B_{k}$ is the Ritz-Galerkin projection of $A$ on Krylov subspaces $\mathrm{span}(U_k)$ and $\mathrm{span}(Z_k)$, while $\bar{B}_{k}$ is the Ritz-Galerkin projection of $L$ on Krylov subspaces $\mathrm{span}(\widehat{U}_k)$ and $\mathrm{span}(Z_k)$. This makes it convenient to approximate some extreme (largest or smallest) generalized singular values and corresponding vectors of $\{A, L\}$ by the SVD of $B_{k}$ and $\bar{B}_{k}$, where the singular values of $B_{k}$ and $\bar{B}_k$ can be used to approximate $c_i$ and $s_i$, respectively. A more detailed investigation on the JBD method for GSVD computation is in \Cref{sec5}.

\section{Inaccurate inner iterations and the reorthogonalized JBD process}\label{sec3}
The JBD process has the structure of nested inner-outer iterations. The overall computational costs of the algorithm is proportional to the overall number of inner iterations. In some cases, using a sparse QR factorization for the inner least squares problem is a good choice. For most large-scale sparse problems, however, the LSQR solver is often much faster, thus should be exploited for computing inner iterations of the JBD process. 

Suppose that \cref{ls} is solved iteratively with the following stopping criterion:
\begin{equation}\label{stop}
\dfrac{\lVert C^{T}\bar{r}_{i} \lVert}{\lVert C \lVert\lVert\bar{r}_{i}\lVert} \leq \tau ,
\end{equation}
where $\bar{r}_{i} = \tilde{u}_{i}-C \bar{z}_{i}$ is the residual with $\bar{z}_{i}$ the approximate solution to \cref{ls}. In \cref{stop}, $\tau$ is called the stopping tolerance, which describes the accuracy of the computed approximation. This stopping criterion is commonly used in iterative methods for solving least squares problems such as the LSQR; see e.g. \cite{Paige1982}. In our analysis, rounding errors in finite precision arithmetic are not taken into account. Denote by $\kappa(M)=\lVert M\lVert\lVert M^{\dag}\lVert$ the condition number of a matrix $M$. The following result describes the influence of computational error of inner iteration on the outer iteration. 

\begin{theorem}\label{thm3.1}
For the $k$-step JBD, suppose the inner least squares problem \cref{ls} is solved iteratively with stopping criterion \cref{stop}. Then there exist vectors $\tilde{g}_{i}\in \mathcal{R}(Q)$ such that
\begin{equation}\label{g_i}
	\beta_{i}\tilde{v}_{i+1} = QQ^{T}\begin{pmatrix}
		u_{i} \\
		0_{p}
	\end{pmatrix} - \alpha_{i}\tilde{v}_{i} - \tilde{g}_{i}
\end{equation}
for $i=1,\dots,k$. If $\kappa(C)\tau<1$, then $\tilde{g}_{i}$ satisfies
\begin{equation}\label{bound_g}
	\lVert \tilde{g}_{i}\lVert \leq 3\kappa(C)\tau + \mathcal{O}(\kappa(C)^2\tau^2) .
\end{equation}
\end{theorem}
\begin{proof}
Suppose that the exact solution of \cref{ls} is $\tilde{z}_{i}$ with residual $\tilde{r}_{i} = \tilde{u}_{i}-C\tilde{z}_{i}$. Then $QQ^{T}\tilde{u}_i=C\tilde{z}_i$. Since the computed approximation to $QQ^{T}\tilde{u}_{i}$ is $C\bar{z}_{i}$, at the $i$-th iteration we have
\begin{align*}
	\beta_{i}\tilde{v}_{i+1} 
	&= C\bar{z}_{i} - \alpha_{i}\tilde{v}_{i} \\
	&= C\tilde{z}_{i} - \alpha_{i}\tilde{v}_{i} - (C\tilde{z}_{i}-C\bar{z}_{i}) \\
	&= QQ^{T}\tilde{u}_i-\alpha_{i}\tilde{v}_{i}-\tilde{g}_{i} ,
\end{align*}
where $\tilde{g}_{i}:=C\tilde{z}_{i}-C\bar{z}_{i}\in\mathcal{R}(Q)$.

Now we give the upper bound on $\tilde{g}_i$. It is known from \cite{Paige1982} that $\bar{z}_{i}$ is the exact solution to the perturbed problem
\begin{equation}\label{3.2}
\min_{\tilde{z}} \lVert \tilde{u}_{i} - (C+\widetilde{E}_{i})\tilde{z} \lVert
\end{equation}
with
\[\widetilde{E}_{i} = -\dfrac{\bar{r}_{i}^{T}\bar{r}_{i}C}{\lVert \bar{r}_{i} \lVert^{2}}, \ \ \ 
\frac{\lVert \widetilde{E}_{i} \lVert}{\lVert C \lVert} =
\dfrac{\lVert C^{T}\bar{r}_{i} \lVert}{\lVert C \lVert\lVert\bar{r}_{i}\lVert} \leq \tau .\]
Suppose the residual of \cref{3.2} is $\bar{r}_{i} = \tilde{u}_{i}-(C+\widetilde{E}_{i})\bar{z}_{i}$. By the perturbation theory of least squares problems \cite[Theorem 20.1]{Higham2002}, we have
\begin{align}
& \frac{\lVert \tilde{z}_{i}-\bar{z}_{i}\lVert}{\lVert\tilde{z}_{i}\lVert} \leq 
\frac{\kappa(C)\tau}{1-\kappa(C)\tau} \Big(1+\frac{\kappa(C)\lVert \tilde{r}_{i}\lVert}{\lVert C\lVert\lVert\tilde{z}_{i}\lVert}\Big) , \label{3.3}  \\
&\lVert\tilde{r}_{i}-\bar{r}_{i} \lVert\leq 2\kappa(C)\lVert \tilde{u}_{i} \lVert \tau
= 2\kappa(C)\tau . \label{3.4}
\end{align}
We mention that the right-hand terms of \cref{3.3} and \cref{3.4} are slightly different from that in \cite[Theorem 20.1]{Higham2002}, since $\tilde{u}_{i}$ in \cref{3.2} is not perturbed \footnote{One can check the proof of \cite[Theorem 20.1]{Higham2002} in \cite[\S 20.10]{Higham2002} to verify the correctness of \cref{3.3,3.4}.}. By the expressions of $\tilde{r}_{i}$ and $\bar{r}_{i}$, we have
\begin{align*}
\| C\tilde{z}_{i}-C\bar{z}_{i} \|
&= \|\widetilde{E}_{i}\bar{z}_{i}-(\tilde{r}_{i}-\bar{r}_{i}) \| 
\leq \| \widetilde{E}_{i} \| \| \bar{z}_{i}\| + \| \tilde{r}_{i}-\bar{r}_{i}\| \\
&\leq (\| C \| \| \bar{z}_{i} \| + 2\kappa(C))\tau  .
\end{align*}
Note that $\|\tilde{z}_{i}\|=\|C^{\dag}\tilde{u}_{i}\|\leq\|C^{\dag}\|$ and $\|\tilde{r}_i\|\leq\|\tilde{u}_i\|=1$. By \cref{3.3} we have
\begin{align*}
	\| C \| \| \bar{z}_{i} \|\tau 
	&\leq \| C \| \tau (\| \tilde{z}_{i} \|+\|\tilde{z}_{i}-\bar{z}_{i}\|) \\
	&\leq \|C\|\|\tilde{z}_i\|\tau + \frac{\|C\|\|\tilde{z}_i\|\tau\kappa(C)\tau}{1-\kappa(C)\tau} + 
	\frac{\kappa(C)^2\tau^2\|\tilde{r}_{i}\|}{1-\kappa(C)\tau} \\
	&\leq \kappa(C)\tau + \frac{\kappa(C)^{2}\tau^2}{1-\kappa(C)\tau} + \frac{\kappa(C)^2\tau^2}{1-\kappa(C)\tau} \\
	&= \kappa(C)\tau +\mathcal{O}(\kappa(C)^2\tau^2), 
\end{align*}
which leads to
\begin{equation*}
	\| C\tilde{z}_{i}-C\bar{z}_{i}\| \leq 3\kappa(C)\tau + \mathcal{O}(\kappa(C)^2\tau^2) .
\end{equation*}
 Thus the upper bound on $\|\tilde{g}_{i}\|$ is obtained.
\end{proof}

Since $\tilde{v}_1$ and $\tilde{g}_{i}$ are in $\mathcal{R}(Q)$, by \cref{g_i} we get $\tilde{v}_{i}\in \mathcal{R}(Q)$. However, due to the appearance of nonzero $\tilde{g}_{i}$, the computed matrices $\widetilde{V}_{k+1}$, $U_{k}$ and $\widehat{U}_{k}$ do not have orthonormal columns any longer. For example, by letting $\tilde{v}_i=Qv_i$ we have from \cref{g_i}
\begin{align*}
	\beta_{1}\tilde{v}_{1}^{T}\tilde{v}_2
	&= \tilde{v}_{1}^{T}(QQ^T\tilde{u}_1-\alpha_1\tilde{v}_{1}-\tilde{g}_1) \\
	&= v_{1}^{T}Q_{A}^{T}u_1-\alpha_1-\tilde{v}_{1}^{T}\tilde{g}_1 \\
	&= [(I_m,0_{m\times p})\tilde{v}_1]^Tu_1-\alpha_1-\tilde{v}_{1}^{T}\tilde{g}_1  \\
	&= -\tilde{v}_{1}^{T}\tilde{g}_1 ,
\end{align*}
where we use $\alpha_{1}u_1=(I_m,0_{m\times p})\tilde{v}_1$. Therefore, $\tilde{v}_1$ and $\tilde{v}_2$ are not orthogonal to each other. The following theorem describes the loss of orthogonality of $u_{i}$ and $\tilde{v}_{i}$.

\begin{theorem}\label{thm3.2}
	Define $\mu_{ji}:=u_{j}^{T}u_i$ and $\nu_{ji}:=\tilde{v}_{j}^{T}\tilde{v}_{i}$. Let $\beta_{0}\mu_{0i}=0$. Then $\mu_{ji}$ and $\nu_{ji}$ satisfy the following coupled recursive relations:
	\begin{align}
		\alpha_{i}\mu_{ji} &= \beta_{j}\nu_{j+1,i}+\alpha_{j}\nu_{ji}-\beta_{i-1}\mu_{j,i-1}+\tilde{v}_{i}^{T}\tilde{g}_{j}, \ \ \ 
		1\leq j \leq i-1, \label{loss1} \\
		\beta_{i}\nu_{j,i+1} &= \alpha_{i}\mu_{ji}+\beta_{j-1}\mu_{j-1,i}-\alpha_{i}\nu_{ji}-\tilde{v}_{j}^{T}\tilde{g}_{i}, \ \ \
		1\leq j \leq i . \label{loss2}
	\end{align}
\end{theorem}
\begin{proof}
Using relation \cref{g_i} and $\alpha_{i}u_{i}=\tilde{v}_{i}(1:m)-\beta_{i-1}u_{i-1}$, we have
\begin{align}
	\alpha_{i}u_{i} &= Q_{A}v_{i} - \beta_{i-1}u_{i-1}, \label{recur1} \\
	\beta_{i}v_{i+1} &= Q_{A}^{T}u_{i} - \alpha_{i}v_{i} - Q^{T}\tilde{g}_{i}, \label{recur2} 
\end{align}
where we have used $\tilde{v}_{i}=Qv_{i}$. Premultiply \cref{recur1} by $u_{j}^{T}$ we have
\begin{align*}
	\alpha_{i}\mu_{ji} 
	&= u_{j}^{T}Q_{A}v_{i}-\beta_{i-1}\mu_{j,i-1} \\
	&= v_{i}^{T}(\alpha_{j}v_{j}+\beta_{j}v_{j+1}+Q^{T}\tilde{g}_{j})-\beta_{i-1}\mu_{j,i-1}  \\
	&= \alpha_{j}\nu_{ji}+\beta_{j}\nu_{j+1,i}-\beta_{i-1}\mu_{j,i-1}+\tilde{v}_{i}^{T}\tilde{g}_{j},
\end{align*}
where we have used $v_{j}^{T}v_{i}=\tilde{v}_{j}^{T}\tilde{v}_{i}$. The relation \cref{loss2} can be proved similarly.
\end{proof}

This result is a corresponding version to \cite[Theorem 6]{Larsen1998} that describes the loss of orthogonality of computed Lanczos vectors when rounding errors are considered. Here, the loss of orthogonality occurs because a perturbation term $Q^{T}\tilde{g}_{i}$ is added to the exact recursive relations of the Lanczos bidiagonalization at each iteration, which is caused by the inaccurate inner iteration; this can be observed from the coupled recursive relations \cref{recur1,recur2}. A similar recursive relation about the loss of orthogonality of $\hat{u}_i$ could also be established, but the expression is more complicated; we omit it since it is not the main aim of this paper. \Cref{thm3.2} indicates that as the inner iteration becomes more inaccurate and $C$ becomes more ill-conditioned, the loss of orthogonality of $u_i$ and $\tilde{v}_i$ occurs more rapidly. This phenomenon has been numerically observed and pointed out when the JBD was first proposed in \cite{Zha1996}.

For Lanczos-type methods, the loss of orthogonality leads to a delay of convergence of Ritz values, making the JBD method for GSVD have an irregular convergence behavior. To avoid this problem, we propose the following reorthogonalized JBD (rJBD) process. At each step, we use the Gram-Schmidt orthogonalization to reorthogonalize $\tilde{v}_{i}$ such that $\widetilde{V}_{k+1}$ is column orthonormal, while vectors $u_i$ and $\hat{u}_i$ do not need to be reorthogonalized. This can save storage and computation costs compared to full reorthogonalizion of all $\tilde{v}_{i}$, $u_i$ and $\hat{u}_i$. This modified algorithm is described in \Cref{alg2}.

\begin{algorithm}[htb]
	\caption{The reorthogonalized JBD (rJBD) process}
	\begin{algorithmic}[1]\label{alg2}
		\Require $A\in\mathbb{R}^{m\times n}$, $L\in\mathbb{R}^{p\times n}$, nonzero $s\in\mathbb{R}^{n}$
		\State Let $\tilde{v}_1 = Cs/\|Cs\|$  \Comment{$C=(A^T, L^T)^T$}
		\For{$i=1,2,\ldots,k$}
		\State $\alpha_{i}u_{i}=\tilde{v}_{i}(1:m)-\beta_{i-1}u_{i-1}$
		\State Solve $\min_{\tilde{z}\in \mathbb{R}^n}
			\|C\tilde{z}-\tilde{u}_i\|$ by LSQR with the stopping criterion \cref{stop}, the 
		\Statex	\quad \ approximated solution is denoted by $\bar{z}_{i}$  \Comment{$\tilde{u}_i=(u_{i}^T, 0_{p}^T)^T$}
		\State $s_{i} = C\bar{z}_{i}-\alpha_{i}\tilde{v}_{i}$  
		\State $\beta_{i}\tilde{v}_{i+1}= s_{i}- \sum\limits_{j=1}^{i}(s_{i}^{T}\tilde{v}_{j})\tilde{v}_{j}$  \Comment{Reorthogonalize $\tilde{v}_{i+1}$}
		\State $\hat{\alpha}_{i}\hat{u}_{i}=
		(-1)^{i-1}\tilde{v}_{i}(m+1:m+p)-\hat{\beta}_{i-1}\hat{u}_{i-1} $
		\State $\hat{\beta}_{i}=(\alpha_{i}\beta_{i})/\hat{\alpha}_{i} $
		\EndFor
		\Ensure $\{u_i, \hat{u}_i\}_{i=1}^{k}$, $\{\tilde{v}_i\}_{i=1}^{k+1}$, $\{\alpha_i, \beta_i, \hat{\alpha}_i, \hat{\beta}_i\}_{i=1}^{k}$
	\end{algorithmic}
\end{algorithm}

We mention that for the rJBD process, the computed quantities $u_{i}$, $\tilde{v}_{i}$, $\hat{u}_{i}$, $\alpha_{i}$, $\hat{\alpha}_{i}$, etc. are different from those obtained by the JBD process. Here we use the same notations to avoid introducing too many tedious notations, and from now on, these notations always denote quantities computed by rJBD. For the reorthogonalization of $\tilde{v}_{i+1}$, by Step 5 and 6, we can write it in a general form:
\[\beta_{i}\tilde{v}_{i+1}=C\bar{z}_{i} - \alpha_{i}\tilde{v}_{i} -\sum\limits_{j=1}^{i}\xi_{ji}\tilde{v}_{j}, \]
where $\xi_{ji}=s_{i}^{T}\tilde{v}_{j}$ for the classical Gram-Schmidt reorthogonalization as is presented in Step 6. In practical computations, using the modified Gram-Schmidt reorthogonalization is usually a better choice. By the above relation, we have
\begin{align*}
	\beta_{i}\tilde{v}_{i+1} 
	&= C\tilde{z}_{i}-\alpha_{i}\tilde{v}_{i} -\sum\limits_{j=1}^{i}\xi_{ji}\tilde{v}_{j} - (C\tilde{z}_i-C\bar{z}_i) \\
	&= QQ^{T}\tilde{u}_{i} -\alpha_{i}\tilde{v}_{i} -\sum\limits_{j=1}^{i}\xi_{ji}\tilde{v}_{j} - (C\tilde{z}_i-C\bar{z}_i)
\end{align*}
Similar to \Cref{thm3.1} and its proof, if we define $\tilde{g}_{i}:=C\tilde{z}_i-C\bar{z}_i$, then we have
\begin{equation}\label{3.6}
	\beta_{i}\tilde{v}_{i+1} = QQ^{T}\begin{pmatrix}
		u_{i} \\
		0_{p}
	\end{pmatrix} - \alpha_{i}\tilde{v}_{i}-\sum\limits_{j=1}^{i}\xi_{ji}\tilde{v}_{j}-\tilde{g}_{i}.
\end{equation}
In particular, the property $\tilde{g}_{i}\in\mathcal{R}(Q)$ and the upper bound in \cref{bound_g} still hold. Note that \cref{3.6} is applied to rJBD, and $\tilde{g}_i$ is different from that in \cref{g_i} for JBD.

Let $\widetilde{G}_{k}=(\tilde{g}_{1},\dots,\tilde{g}_{k})$. The $k$-step rJBD can be written in matrix form:
\begin{align}
& (I_{m},0_{m\times p})\widetilde{V}_{k}=U_{k}B_{k} , \label{3.7} \\
& QQ^{T}
\begin{pmatrix}
U_{k} \\
0_{p\times k}
\end{pmatrix}
=\widetilde{V}_{k}(B_{k}^{T}+D_{k})+\beta_{k}\tilde{v}_{k+1}(e_{k}^{(k)})^{T}+\widetilde{G}_{k}, \label{3.8} \\
& (0_{p\times m},I_{p})\widetilde{V}_{k}P=\widehat{U}_{k}\widehat{B}_{k} , \label{3.9}
\end{align}
where 
\[D_{k}=\begin{pmatrix}
\xi_{11} &\cdots &\cdots  &\xi_{1k} \\
&\xi_{22} &\cdots &\xi_{2k} \\
&  &\ddots &\vdots \\
& &  &\xi_{kk} \\
\end{pmatrix} \in \mathbb{R}^{k \times k} .\]
For the rJBD, the matrix $\widetilde{V}_{k+1}$ is column orthonormal, while $U_{k}$ and $\widehat{U}_{k}$ are not column orthonormal.

\section{Error analysis of the rJBD process}\label{sec4}
For the $k$-step rJBD, if one of $\alpha_i$, $\beta_i$, $\hat{\alpha}_i$ and $\hat{\beta}_i$ becomes zero, then the procedure terminates. It is usually called a ``lucky terminate" \cite{Golub1965}, since the procedure have found an invariant singular subspace. In our analysis, we assume that $\alpha_i$, $\beta_i$, $\hat{\alpha}_i$ and $\hat{\beta}_i$ never become zero or numerical negligible after $k$-steps.

Since $\tilde{v}_1$ and $\tilde{g}_{i}$ are in $\mathcal{R}(Q)$, by \cref{3.6} we have $\tilde{v}_{i}\in \mathcal{R}(Q)$ for $i=1,2,\dots$. Suppose $\tilde{v}_{i}=Qv_{i}$ and let $\hat{v}_i=(-1)^{i-1}v_i$. Then we obtain from \cref{3.7,3.8,3.9} that
\begin{align}
& Q_{A}V_{k} = U_{k}B_{k}, \label{3.10} \\
& Q_{A}^{T}U_{k}=V_{k}(B_{k}^{T}+D_{k})+\beta_{k}v_{k+1}(e_{k}^{(k)})^{T}+G_{k} , \label{3.11} \\
& Q_L\widehat{V}_k = \widehat{U}_k\widehat{B}_k, \label{3.12}
\end{align}
where $V_k=(v_1,\dots,v_k)$, $\widehat{V}_k=(\hat{v}_1,\dots,\hat{v}_k)$ and $G_{k}=Q^{T}\widetilde{G}_{k}=(g_1,\dots,g_k)$ with $g_{i}=Q^{T}\tilde{g}_i$. By $\tilde{g}_{i}\in\mathcal{R}(Q)$ we have $\|g_i\|=\|\tilde{g}_i\|$. Based on these matrix form relations, we make an error analysis of the rJBD process, which builds up connections with the bidiagonal reductions of $Q_A$ and $Q_L$.

\subsection{Bidiagonal reduction of $Q_A$}
Note that \cref{3.10,3.11} imply that the process of generating $B_k$ is closely related to the Lanczos bidiagonalization of $Q_A$, where the differences include the reorthogonalizations of $v_i$ and perturbation errors $g_i$. We develop methods inspired by \cite{Barlow2013} to establish a backward error bound about the $k$-step bidiagonal reduction of $Q_A$. For simplicity, we only discuss the case for $m\geq n$.

First we give a relation describing the generation of each column of $B_k$. 
\begin{theorem}\label{lem4.1}
	For the $k$-step rJBD process, define
	\begin{equation}
		\widehat{P}_{l+1}=P_{1}\cdots P_{l+1} , \ \
		P_{i}=I_{m+n}-p_{i}p_{i}^{T}  , \ \
		p_{i}=\begin{pmatrix}
			-e_{i}^{(n)} \\
			u_{i}
		\end{pmatrix} \in \mathbb{R}^{m+n} 
	\end{equation}
	for $1 \leq l \leq k-1$. There exist $f_{l+1}\in \mathbb{R}^{m+n}$ such that
	\begin{equation}
	\widehat{P}_{l+1}\begin{pmatrix}
	\beta_{l}e_{l}^{(l)} \\
	\alpha_{l+1}e_{1}^{(s)}
	\end{pmatrix} =
	\begin{pmatrix}
	0_{n} \\
	Q_Av_{l+1}
	\end{pmatrix} + f_{l+1}, \ \  \|f_{l+1}\|=\mathcal{O}(l\kappa(C)\tau).
	\end{equation}
	with $s=m+n-l$. 
\end{theorem}
\begin{proof}	
	By \cref{3.10} and the expression of $P_i$, we have
	\begin{align*}
	P_{l+1}\begin{pmatrix}
	\beta_{l}e_{l}^{(l)} \\
	\alpha_{l+1}e_{1}^{(s)}
	\end{pmatrix}
	&= \begin{pmatrix}
	\beta_{l}e_{l}^{(l)} \\
	\alpha_{l+1}e_{1}^{(s)}
	\end{pmatrix} - p_{l+1}^{T}\begin{pmatrix}
	\beta_{l}e_{l}^{(l)} \\
	\alpha_{l+1}e_{1}^{(s)}
	\end{pmatrix}p_{l+1} \\
	&=\begin{pmatrix}
	\beta_{l}e_{l}^{(l)} \\
	\alpha_{l+1}e_{1}^{(s)}
	\end{pmatrix} + \alpha_{l+1}\begin{pmatrix}
	-e_{l+1}^{(n)} \\
	u_{l+1}
	\end{pmatrix} \\
	&= \begin{pmatrix}
	\beta_{l}e_{l}^{(n)} \\
	\alpha_{l+1}u_{l+1}
	\end{pmatrix} 
	= \begin{pmatrix}
	\beta_{l}e_{l}^{(n)} \\
	Q_{A}v_{l+1}-\beta_{l}u_{l}
	\end{pmatrix} ,
	\end{align*}
	and
	\begin{align*}
	P_{l}\begin{pmatrix}
	\beta_{l}e_{l}^{(n)} \\
	Q_{A}v_{l+1}-\beta_{l}u_{l}
	\end{pmatrix}
	&= P_{l}\begin{pmatrix}
	0_{n} \\
	Q_{A}v_{l+1}
	\end{pmatrix} - \beta_{l}P_{l}\begin{pmatrix}
	-e_{l}^{(n)} \\
	u_{l}
	\end{pmatrix} \\
	&= \begin{pmatrix}
	0_{n} \\
	Q_{A}v_{l+1}
	\end{pmatrix} - (u_{l}^{T}Q_{A}v_{l+1})p_{l} + \beta_{l}p_{l} .
	\end{align*}
	By \cref{3.11}, we have
	\begin{equation}\label{3.15}
	Q_{A}^{T}u_{i} = \alpha_{i}v_{i}+\beta_{i}v_{i+1}+\sum_{j=1}^{i}\xi_{ji}v_{j}+g_{i}
	\end{equation}
	for $i=1,2,\dots,l$. By \cref{3.15} and using the column orthogonality of $V_{l}$, we have
	\begin{align*}
	u_{l}^{T}Q_{A}v_{l+1} = v_{l+1}^{T}(Q_{A}^{T}u_{l}) 
	= v_{l+1}^{T}(\alpha_{l}v_{l}+\beta_{l}v_{l+1}+\sum_{j=1}^{l}\xi_{jl}v_{j}+g_{l}) 
	= \beta_{l} + v_{l+1}^{T}g_{l},
	\end{align*}
	which leads to
	\begin{equation*}
	P_{l}\begin{pmatrix}
	\beta_{l}e_{l}^{(n)} \\
	Q_{A}v_{l+1}-\beta_{l}u_{l}
	\end{pmatrix}
	=\begin{pmatrix}
	0_{n} \\
	Q_{A}v_{l+1}
	\end{pmatrix} - (v_{l+1}^{T}g_{l})p_{l}.
	\end{equation*}
	Using the same method as the above, we have
	\begin{align*}
	\begin{split}
	P_{i}\begin{pmatrix}
	0_{n} \\
	Q_{A}v_{l+1}
	\end{pmatrix}
	&= \begin{pmatrix}
	0_{n} \\
	Q_{A}v_{l+1}
	\end{pmatrix} - (u_{i}^{T}Q_{A}v_{l+1})p_{i} \\
	&= \begin{pmatrix}
	0_{n} \\
	Q_{A}v_{l+1}
	\end{pmatrix} - v_{l+1}^{T}(\alpha_{i}v_{i}+\beta_{i}v_{i+1}+\sum_{j=1}^{i}\xi_{ji}v_{j}+g_{i})p_{i} \\
	&= \begin{pmatrix}
	0_{n} \\
	Q_{A}v_{l+1}
	\end{pmatrix} - (v_{l+1}^{T}g_{i})p_{i}
	\end{split}
	\end{align*}
	for $i=1,2,\dots,l-1$. Combining the above two equalities leads to
	\begin{equation*}
		\widehat{P}_{l+1}\begin{pmatrix}
			\beta_{l}e_{l}^{(l)} \\
			\alpha_{l+1}e_{1}^{(s)}
		\end{pmatrix}
		= P_{1}\cdots P_{l-1}\Bigg( \begin{pmatrix}
			0_{n} \\
			Q_{A}v_{l+1}
		\end{pmatrix} -(v_{l+1}^{T}g_{l})p_{l} \Bigg) 
		= \begin{pmatrix}
			0_{n} \\
			Q_{A}v_{l+1}
		\end{pmatrix} + f_{l+1} ,
	\end{equation*}
	with $f_{l+1} = -\sum_{i=1}^{l}(P_{1}\cdots P_{i-1})(v_{l+1}^{T}g_{i})p_{i}$, where $P_0=I_{m+n}$. Note that $\|p_{i}\|=\sqrt{2}$ and $P_{i}$ are Householder matrices. By using the upper bound on $\|g_i\|=\|\tilde{g}_i\|$ and neglecting high order terms of $\tau$ we get $\|f_{l+1}\|=\mathcal{O}(l\kappa(C)\tau)$.
\end{proof}

This result will play an important role in the following analysis. Now we give a backward error bound about the $k$-step bidiagonal reduction of $Q_A$, which is the main result in this subsection.

\begin{theorem}\label{thm4.1}
	For the $k$-step rJBD process, there exist a column orthonormal matrix $\bar{U}_{k}\in \mathbb{R}^{m\times k}$ and a matrix $E_{k} \in \mathbb{R}^{m\times n}$ such that
	\begin{align}
		& (Q_A+E_k)V_{k} = \bar{U}_{k}B_{k} , \label{4.15}  \\
		& (Q_A+E_k)^{T}\bar{U}_{k} = V_kB_{k}^{T}+\beta_{k}v_{k+1}(e_{k}^{(k)})^{T}, \label{4.16}
	\end{align}
	and
	\begin{equation}\label{back_err}
	\lVert E_{k} \lVert = \mathcal{O}(\sqrt{n}k\kappa(C)\tau) .
	\end{equation}
\end{theorem}

Before proving \Cref{thm4.1}, we first give some remarks. Notice that the relations \cref{4.15,4.16} are matrix-form recurrences of the $k$-step (upper) Lanczos bidiagonalization of $\bar{Q}_A=Q_A+E_k$. Thus the subspace $\mathrm{span}(V_k)$ is the Krylov subspace $\mathcal{K}_k(\bar{Q}_{A}^{T}\bar{Q}_{A}, v_1)$; see e.g. \cite[\S 10.4.1]{Golub2013}. Therefore, the singular values of $B_k$ will approximate the singular values of $\bar{Q}_{A}$ instead of that of $Q_A$. Clearly, the accuracy of approximations to $c_i$ by the SVD of $B_k$ is limited by the value of $\kappa(C)\tau$.

The proof of \Cref{thm4.1} depends on the following two lemmas. The first lemma gives a similar relation as \Cref{lem4.1}, where some quantities are constructed only for aiding subsequent proofs.
\begin{lemma}\label{lem4.2}
	For the $k$-step rJBD with $k<n$, there exist vectors $\check{u}_{k+1},\dots,\check{u}_{n}$ and $\check{v}_{k+2},\dots,\check{v}_{n}$ and nonnegative numbers $\check{\alpha}_{k+1},\dots,\check{\alpha}_{n}$ and  $\check{\beta}_{k+1},\dots,\check{\beta}_{n-1}$ (define $\check{\beta}_{n}:=0$ and $\check{v}_{n+1}:=0$), such that: \\
	$\mathrm{(I).}$ $\check{u}_{i}$ and $\check{v}_{i}$ are of unit-norm, and $\check{V}=(V_{k+1}, \check{v}_{k+2},\dots,\check{v}_{n})$ is orthogonal; \\
	$\mathrm{(II).}$ define for $k\leq l \leq n-1$
	\begin{equation*}
		\widetilde{P}_{l+1}=\widehat{P}_{k}\check{P}_{k+1}\cdots \check{P}_{l+1} , \ \
		\check{P}_{i}=I_{m+n}-\check{p}_{i}\check{p}_{i}^{T}  , \ \
		\check{p}_{i}=\begin{pmatrix}
			-e_{i}^{(n)} \\
			\check{u}_{i}
		\end{pmatrix},
	\end{equation*}
	then there exist $f_{k+1}\in \mathbb{R}^{m+n}$ such that
	\begin{equation}\label{3.14}
		\widetilde{P}_{l+1}\begin{pmatrix}
			\check{\beta}_{l}e_{l}^{(l)} \\
			\check{\alpha}_{l+1}e_{1}^{(s)}
		\end{pmatrix} =
		\begin{pmatrix}
			0_{n} \\
			Q_A\check{v}_{l+1}
		\end{pmatrix} + f_{k+1}, \ \  \|f_{k+1}\|=\mathcal{O}(k\kappa(C)\tau),
	\end{equation}
	where $\check{v}_{k+1}:=v_{k+1}$ and $\check{\beta}_{k}:=\beta_{k}$.
\end{lemma}
\begin{proof}
	First we construct vectors $\check{u}_{k+1},\dots,\check{u}_{n}$ and $\check{v}_{k+2},\dots,\check{v}_{n}$. For $i\geq k+1$, vectors $\check{u}_{i}$ and $\check{v}_{i+1}$ are generated as
	\begin{align*}
		& \check{\alpha}_{i}\check{u}_{i} =Q_A\check{v}_{i}-\check{\beta}_{i-1}\check{u}_{i-1}, \\
		& r_i = Q_{A}^{T}\check{u}_{i}-\check{\alpha}_{i}v_{i} \ \ \ 
		\check{\beta}_{i}\check{v}_{i+1} = r_i - \sum\limits_{j=1}^{k+1}(v_{j}^{T}r_i)v_j-\sum\limits_{j=k+2}^{i}(\check{v}_{j}^{T}r_i)\check{v}_j ,
	\end{align*}
	such that $\|\check{u}_{i}\|=\|\check{v}_{i+1}\|=1$, where for $i=k+1$ we let $\check{u}_{k}=u_{k}$, $\check{v}_{k+1}=v_{k+1}$ and $\check{\beta}_{k}=\beta_{k}$. If the procedure terminates at some step, it can be continued by choosing a new starting vector. Note that $\check{v}_{i+1}$ are generated with full reorthogonalization. Thus $\check{\beta}_{n}\check{v}_{n+1}=0_n$ and $\check{V}$ is orthogonal.
	
	For $l\geq k+1$, by using similar calculations as the proof of \cref{lem4.1}, we have
	\begin{align*}
		& \check{P}_l\check{P}_{l+1}\begin{pmatrix}
			\check{\beta}_{l}e_{l}^{(l)} \\
			\check{\alpha}_{l+1}e_{1}^{(s)}
		\end{pmatrix} = \begin{pmatrix}
			0_{n} \\
			Q_{A}\check{v}_{l+1}
		\end{pmatrix} , \\
		& \check{P}_i\begin{pmatrix}
		0_{n} \\
		Q_{A}\check{v}_{l+1}
		\end{pmatrix} =\begin{pmatrix}
		0_{n} \\
		Q_{A}\check{v}_{l+1}
		\end{pmatrix} , \ \ i=k+1,\dots,l-1, \\
		& P_i\begin{pmatrix}
		0_{n} \\
		Q_{A}\check{v}_{l+1}
		\end{pmatrix} =\begin{pmatrix}
		0_{n} \\
		Q_{A}\check{v}_{l+1}
		\end{pmatrix} -(\check{v}_{l+1}^{T}g_i)p_i, \ \ i=1,\dots, k .
	\end{align*}
	Thus we obtain
	\begin{equation*}
		\widehat{P}_{k}\check{P}_{k+1}\cdots\check{P}_{l+1}\begin{pmatrix}
			\beta_{l}e_{l}^{(l)} \\
			\alpha_{l+1}e_{1}^{(s)}
		\end{pmatrix}
		= P_{1}\cdots P_{k}\begin{pmatrix}
			0_{n} \\
			Q_{A}v_{l+1}
		\end{pmatrix} 
		= \begin{pmatrix}
			0_{n} \\
			Q_{A}v_{l+1}
		\end{pmatrix} + f_{k+1} 
	\end{equation*}
	with $f_{k+1} = -\sum_{i=1}^{k}(P_{1}\cdots P_{i-1})(\check{v}_{l+1}^{T}g_{i})p_{i}$ and $\|f_{k+1}\|=\mathcal{O}(k\kappa(C)\tau)$.

	For $l=k$, it can also be verified that \cref{3.14} holds and we omit the similar calculations. 
\end{proof}

\Cref{lem4.1,lem4.2} lead to the following result.
\begin{lemma}\label{lem4.3}
	For the $k$-step rJBD process,
	define the upper bidiagonal matrix as
	\[\check{B}=
	\begin{pmatrix}
		\alpha_{1} & \beta_1 & & & & & \\
		& \alpha_{2} & \ddots & & & & \\
		&  & \ddots & \beta_{k-1} & & & \\
		& & & \alpha_{k} & \beta_k & & \\
		& & &  & \check{\alpha}_{k+1} & \ddots & \\
		& & & & & \ddots & \check{\beta}_{n-1} \\
		& & & & & & \check{\alpha}_{n}
	\end{pmatrix} \in\mathbb{R}^{n\times n}.\]
	If follows that
	\begin{equation}\label{Bn}
		\begin{pmatrix}
			0_{n\times n} \\
			Q_A\check{V}_n
		\end{pmatrix} + F_k =
		\widetilde{P}_n\begin{pmatrix}
			\check{B} \\
			0_{m\times n}
		\end{pmatrix}
	\end{equation}
	with $F_k=(\underbrace{0_{m+n},f_2,\dots,f_{k}}_{k},\underbrace{f_{k+1},\dots,f_{k+1}}_{n-k})\in\mathbb{R}^{(m+n)\times n}$, where $f_{n+1}:=0$ for $k=n$.
\end{lemma}
\begin{proof}
	The proof can be completed by comparing each column of the right- and left-hand sides of \cref{Bn}.
	
	For the first column, we have
	\begin{align*}
		\widetilde{P}_n\begin{pmatrix}
			\alpha_{1} \\ 0_{m+n-1}
		\end{pmatrix}
	&= P_1\Big[P_2\cdots P_k\check{P}_{k+1}\cdots\check{P}_n \begin{pmatrix}
		\alpha_{1} \\ 0_{m+n-1}
	\end{pmatrix}\Big] \\
	&= P_1\begin{pmatrix}
		\alpha_{1} \\ 0_{m+n-1}
	\end{pmatrix}=\begin{pmatrix}
	\alpha_{1} \\ 0_{m+n-1}
	\end{pmatrix} + 
	\alpha_1\begin{pmatrix}
	-e_{1}^{(n)} \\ u_1
	\end{pmatrix} \\
	&= \begin{pmatrix}
		0_n \\ \alpha_1u_1
	\end{pmatrix}
	=\begin{pmatrix}
		0_n	\\ Q_Av_1
	\end{pmatrix}.
	\end{align*}
	For the $(l+1)$-th column with $1\leq l\leq k-1$, by \Cref{lem4.1} we have 
	\begin{align*}
		\widetilde{P}_n\begin{pmatrix}
		\beta_{l}e_{l}^{(l)} \\
		\alpha_{l+1}e_{1}^{(s)}
	\end{pmatrix}
	&= \widehat{P}_{l+1}\Big[P_{l+2}\cdots\check{P}_n
	\begin{pmatrix}
		\beta_{l}e_{l}^{(l)} \\
		\alpha_{l+1}e_{1}^{(s)}
	\end{pmatrix}\Big] \\
	&= \widehat{P}_{l+1}\begin{pmatrix}
		\beta_{l}e_{l}^{(l)} \\
		\alpha_{l+1}e_{1}^{(s)}
	\end{pmatrix} 
	= \begin{pmatrix}
		0_{n} \\
		Q_Av_{l+1}
	\end{pmatrix} + f_{l+1}.
	\end{align*}
	Similar calculations can be carried out for the $(l+1)$-th column with $k\leq l\leq n-1$ by using \Cref{lem4.2}. Therefore, the equality \cref{Bn} holds. 	
\end{proof}

With aids of the above two lemmas, we can now give the proof of \Cref{thm4.1}.
\begin{proof}[Proof of \Cref{thm4.1}]
	By \Cref{lem4.3}, we have
	\[ \begin{pmatrix}
		0_{n\times n} \\
		Q_A\check{V}_n
	\end{pmatrix} + F_k =
	\widetilde{P}_n\begin{pmatrix}
		\check{B} \\
		0_{m\times n}
	\end{pmatrix}
	=\begin{pmatrix}
	\widehat{P}_{11} & \widehat{P}_{12}  \\
	\widehat{P}_{21} & \widehat{P}_{22} \\
	\end{pmatrix}
	\begin{pmatrix}
	\check{B} \\
	0_{m\times n}
	\end{pmatrix} =
	\begin{pmatrix}
	\widehat{P}_{11} \\
	\widehat{P}_{21}
	\end{pmatrix}\check{B} ,\]
	where $\widetilde{P}_n$ is partitioned as
	\[\widetilde{P}_{n} =
	\begin{blockarray}{ccc}
	n & m & \\
	\begin{block}{(cc)c}
	\widehat{P}_{11} & \widehat{P}_{12} & n \\
	\widehat{P}_{21} & \widehat{P}_{22} & m \\
	\end{block}
	\end{blockarray}. \]
	By \cite[Theorem 4.1]{Paige2009}, there exist a column orthonormal matrix 
	$\bar{U}_{n}\in \mathbb{R}^{m\times n}$ and a matrix $M\in \mathbb{R}^{m\times n}$ satisfying $0.5\leq \|M\|\leq 1$ such that
	\[ Q_{A}\check{V}_{n} + \widehat{E}_k = \bar{U}_{n}\check{B}_{n} ,\]
	where
	$\widehat{E}_k =\begin{pmatrix}
	M\widehat{P}_{11}^{T}, & I_{m}
	\end{pmatrix}F_k$. Therefore, we have
	\[ (Q_{A}+E_k)\check{V}_{n} = \bar{U}_{n}\check{B}_{n} \]
	with
	$E_k = \begin{pmatrix}
	M\widehat{P}_{11}^{T}, & I_{m}
	\end{pmatrix}F_{k}\check{V}_{n}^{T}$, which can be also written 
	as 
	\[(Q_A+E_k)^T\bar{U}_n = \check{V}_n\check{B}_{n}^{T} .\]
	By equating the first $k$ columns of the above two equalities, respectively, we obtain \eqref{4.15} and \eqref{4.16}.
	Finally, we have the upper bound
	\[\|E_k\|\leq\Big\|\begin{pmatrix}
	M\widehat{P}_{11}^{T}, & I_{m}
	\end{pmatrix}\Big\|\|F_{k}\|\leq \sqrt{2}\|F_{k}\|=\mathcal{O}(\sqrt{n}k\kappa(C)\tau) ,\]
	where we have used $\|F_{k}\|\leq\|F_k\|_{F}\leq \sqrt{n}\max_{1\leq l\leq k}\|f_{l+1}\|$.
\end{proof}

\subsection{Bidiagonal reduction of $Q_L$}
For the rJBD process, a relation that is similar to \cref{B_Bh} holds, and the process of generating $\widehat{B}_{k}$ is closely related to the Lanczos bidiagonalization of $Q_{L}$. 

\begin{theorem}\label{thm4.3}
	For the $k$-step rJBD process, we have
	\begin{equation}\label{3.18}
		B_{k}^{T}B_{k}+\bar{B}_{k}^{T}\bar{B}_k=I_{k} + H_{k},
	\end{equation}
	where $H_{k}$ is a diagonal matrix, and the $i$-th diagonal of $H_{k}$ is of order $\mathcal{O}(\theta_{i-1}\kappa(C)\tau)$ with $\theta_{i}=\sum_{j=0}^{i-1}(\hat{\beta}_{i}\cdots\hat{\beta}_{i-j})/(\hat{\alpha}_{i}\cdots\hat{\alpha}_{i-j})$ for $i\geq 1$ and $\theta_0=0$.
\end{theorem}

The growth speed of $\theta_i$ is moderate under the assumption that $\hat{\beta}_i/\hat{\alpha}_i$ is a moderate value for $i=1,\dots,k$, which is reasonable since if the rJBD process does not numerically terminate, then both $\hat{\beta}_i$ and $\hat{\alpha}_i$ are moderate values. We use the following example to illustrate the variation of values $\hat{\beta}_k/\hat{\alpha}_k$ and $\theta_k$ with respect to $k$. The property of the test matrix pair $\{\sf illc1850, well1850\}$ is shown in \cref{tab1}.

\begin{figure}[htbp]
	\centering
	\subfloat{\includegraphics[width=0.42\textwidth]{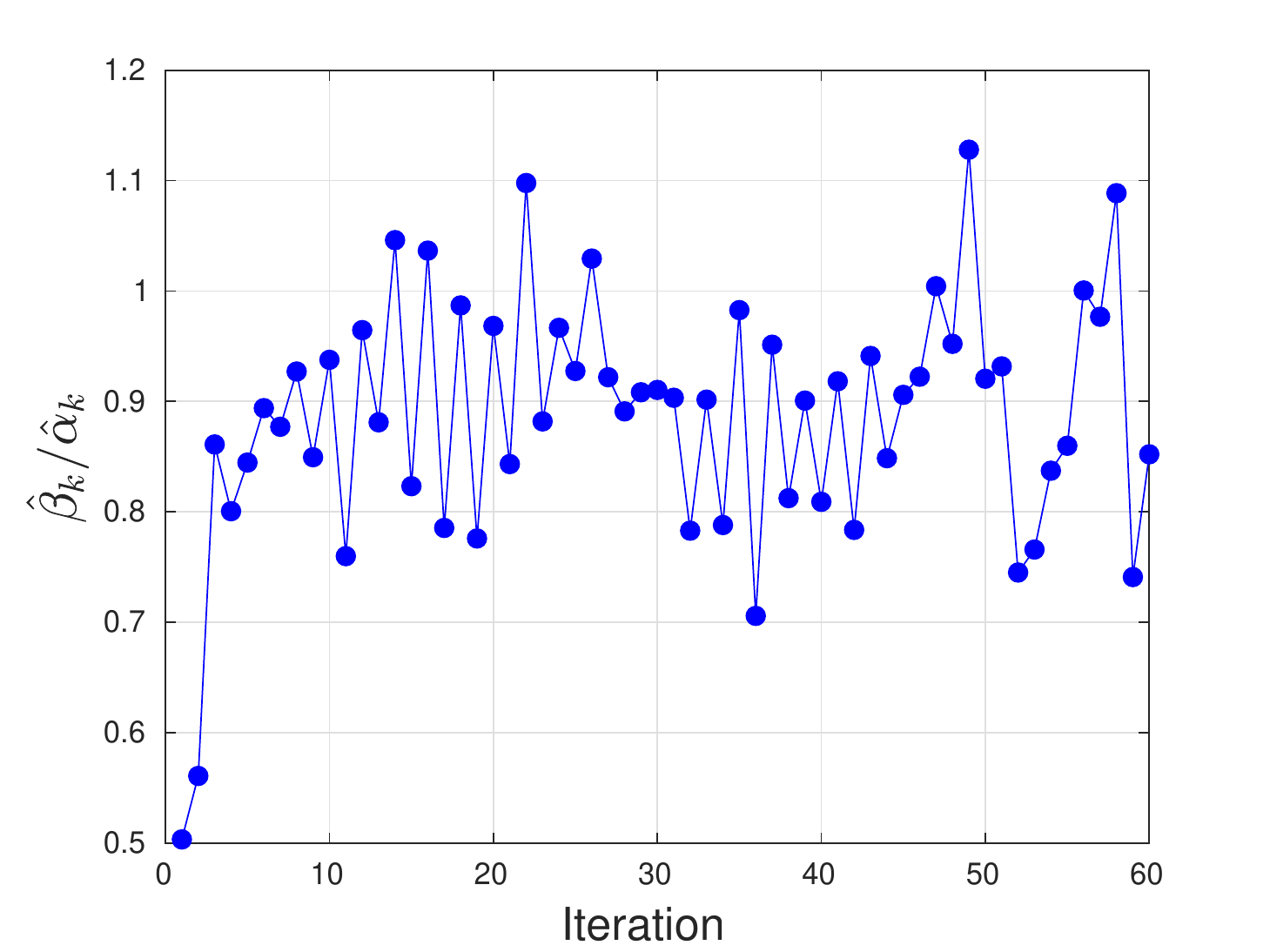}}
	\subfloat{\includegraphics[width=0.42\textwidth]{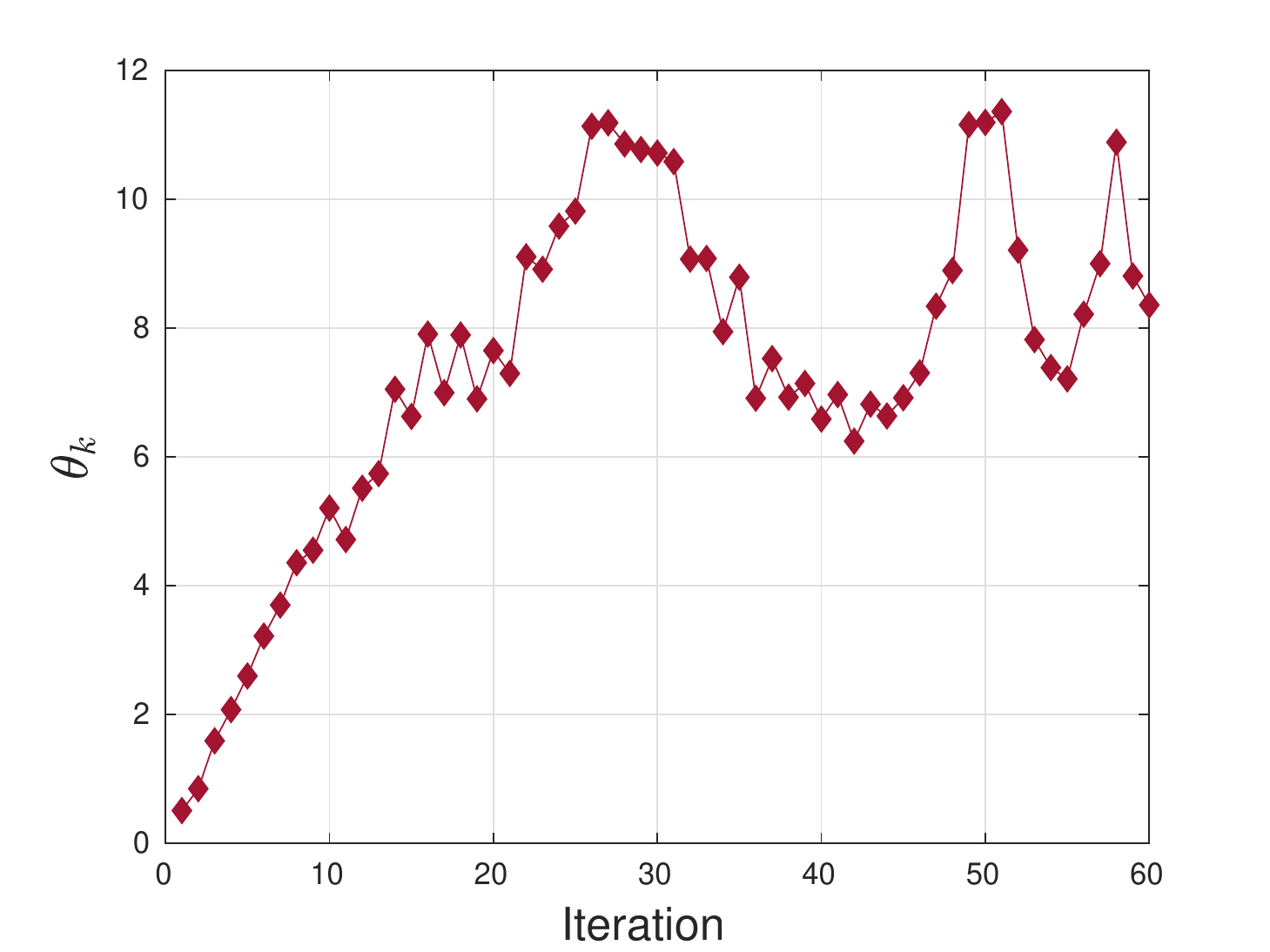}}
	\caption{Variation of values $\hat{\beta}_k/\hat{\alpha}_k$ and $\theta_k$ for rJBD of \{\sf illc1850, well1850\}, $\tau=10^{-12}$.}
	\label{fig0}
\end{figure}

Note that the only difference between  relations \cref{3.18} and \cref{B_Bh} is the perturbation term $H_k$, which comes from the inaccurate inner iteration. From \Cref{thm4.3}, we know that the singular values of $\bar{B}_{k}$ are determined by that of $B_k$ within errors of order $\mathcal{O}(\kappa(C)\tau)$, where we omit the moderate value $\max_{1\leq i\leq k-1}\theta_i$. This ensures that $\bar{B}_{k}$ can also be used to approximate generalized singular values of $\{A,L\}$.

To prove this result, we need the following lemma.
\begin{lemma}\label{lem4.5}
	For any $i\geq 1$, we have
	\begin{equation*}
		\hat{\beta}_iQ_{L}^{T}\hat{u}_i = \hat{\beta}_{i}^{2}\hat{v}_{i+1}+\widehat{V}_id_i+q_i
	\end{equation*}
	with a vector $d_i\in\mathbb{R}^{i}$, and $\|q_i\|=\mathcal{O}(\theta_i\kappa(C)\tau)$.
\end{lemma}
\begin{proof}
	The proof can be completed by mathematical induction. For $i=1$, we obtain from $\hat{\alpha}_1\hat{u}_1=Q_L\hat{v}_1$ that
	\begin{align*}
		\hat{\alpha}_1Q_{L}^{T}\hat{u}_1
		&= (I_n-Q_{A}^{T}Q_A)\hat{v}_1
		 = \hat{v}_1 - Q_{A}^{T}(\alpha_{1}u_1) \\
		&= \hat{v}_1-\alpha_{1}(\beta_{1}v_2+\alpha_{1}v_1+\xi_{11}v_1+g_1) \\
		&= \hat{\alpha}_1\hat{\beta}_1\hat{v}_2+(1-\alpha_{1}^{2}-\alpha_{1}\xi_{11})\hat{v}_1-\alpha_{1}g_1,
	\end{align*}
	then we get
	\[\hat{\beta}_1Q_{L}^{T}\hat{u}_1= \hat{\beta}_{1}^{2}\hat{v}_{2}+\widehat{V}_1d_1-\hat{\beta}_1/\hat{\alpha}_1\cdot \alpha_1g_1
	 = \hat{\beta}_{1}^{2}\hat{v}_{1}+\widehat{V}_1d_1 + q_1\]
	 with $d_1=\hat{\beta}_1/\hat{\alpha}_1\cdot (1-\alpha_{1}^{2}-\alpha_{1}\xi_{11})$ and $q_1=-\hat{\beta}_1/\hat{\alpha}_1\cdot \alpha_1g_1$, and $\|q_1\|=\mathcal{O}(\theta_1\kappa(C)\tau)$ since $\alpha_1=\|\tilde{v}_1(1:m)\|\leq 1$. 
	 
	 Suppose the relation is true for indices up to $i-1\geq 1$. For index $i$, we have
	 \begin{align*}
	 	\hat{\alpha}_iQ_{L}^{T}\hat{u}_i
	 	&= Q_{L}^{T}Q_{L}\hat{v}_i-\hat{\beta}_{i-1}Q_{L}^{T}\hat{u}_{i-1} \\
	 	&= (I_n-Q_{A}^{T}Q_{A})\hat{v}_i-\hat{\beta}_{i-1}Q_{L}^{T}\hat{u}_{i-1} \\
	 	&= \hat{v}_i+(-1)^iQ_{A}^{T}(\alpha_iu_i+\beta_{i-1}u_{i-1})-\hat{\beta}_{i-1}Q_{L}^{T}\hat{u}_{i-1}.
 	 \end{align*}
	By \cref{3.15} we have
	 \[(-1)^i\alpha_{i}Q_{A}^{T}u_{i} = \hat{\alpha}_i\hat{\beta}_{i}\hat{v}_{i+1}+(-1)^i\alpha_{i}(\alpha_{i}v_{i}+\sum_{j=1}^{i}\xi_{ji}v_{j}+g_{i}) \]
	 and
	 \[Q_{A}^{T}u_{i-1}\in\mathrm{span}\{\hat{v}_1,\dots,\hat{v}_i\}+g_{i-1} .\]
	 Combining the above two relations with the induction hypothesis
	 \[\hat{\beta}_{i-1}Q_{L}^{T}\hat{u}_{i-1}= \hat{\beta}_{i-1}^{2}\hat{v}_{i}+\widehat{V}_{i-1}d_{i-1}+q_{i-1}\]
	 and $\|q_{i-1}\|=\mathcal{O}(\theta_{i-1}\kappa(C)\tau)$,
	 we get
	 \begin{equation}\label{mid1}
		\hat{\alpha}_iQ_{L}^{T}\hat{u}_i= \hat{\alpha}_i\hat{\beta}_{i}\hat{v}_{i+1}+\widehat{V}_i\tilde{d}_i
	 +(-1)^{i}(\alpha_{i}g_i+\beta_{i-1}g_{i-1})+q_{i-1}
	 \end{equation}
	 with an $\tilde{d}_i\in\mathbb{R}^i$. By \cref{lem4.1} we have
	 \begin{equation}\label{alp_beta}
		(\alpha_{i}^2+\beta_{i-1}^{2})^{1/2}= 
		\Big\|\begin{pmatrix}
			\beta_{i-1}e_{i}^{(i)} \\
			\alpha_{i}e_{1}^{(s)}
		\end{pmatrix} \Big\|\leq \Big\|\begin{pmatrix}
		0_{n} \\
		Q_Av_{i}
		 \end{pmatrix}\Big\| + \|f_{i}\| 
	   \leq 1+	\mathcal{O}(i\kappa(C)\tau).
	 \end{equation}	
 	Thus we get
 	\[\|\alpha_{i}g_i+\beta_{i-1}g_{i-1}\|\leq \sqrt{2}(\alpha_{i}^2+\beta_{i-1}^{2})^{1/2}
 	\max\{\|g_i\|,\|g_{i-1}\|\}=\mathcal{O}(\kappa(C)\tau) \]
 	by neglecting higher orders of $\tau$. We finally obtain from \cref{mid1}
	\[\hat{\beta}_iQ_{L}^{T}\hat{u}_i = \hat{\beta}_{i}^{2}\hat{v}_{i+1}+\widehat{V}_id_i+q_i \]
	with $q_i= \hat{\beta}_i/\hat{\alpha}_i[(-1)^{i}(\alpha_{i}g_i+\beta_{i-1}g_{i-1})+q_{i-1}]$,
	and
	\[\|q_i\|\leq\hat{\beta}_i/\hat{\alpha}_i\mathcal{O}(\kappa(C)\tau)+\hat{\beta}_i/\hat{\alpha}_iq_{i-1}
	= \mathcal{O}(\theta_i\kappa(C)\tau) \]
	since $\hat{\beta}_i/\hat{\alpha}_i+\hat{\beta}_i/\hat{\alpha}_i\cdot\theta_{i-1}=\theta_i$.
	This completes the proof of the induction step.
\end{proof}

Now we can give the proof of \Cref{thm4.3}.
\begin{proof}[Proof of \Cref{thm4.3}]
	Using the bidiagonal structure of $B_{k}$ and $\bar{B}_{k}$, we know that $H_k$ is symmetric tridiagonal. Note that the subdiagonals of $B_{k}^{T}B_{k}$ and $\bar{B}_{k}^{T}\bar{B}_{k}$ are $\alpha_{i}\beta_i$ and $-\hat{\alpha}_i\hat{\beta}_i$, respectively. Thus, the subdiagonals of $H_k$ are zero and $H_k$ is a diagonal matrix. For the $i$-th diagonal element that is $\alpha_{i}^2+\beta_{i-1}^{2}+\hat{\alpha}_{i}^{2}+
	\hat{\beta}_{i-1}^{2}$, we use the relations
	\begin{align*}
		& \alpha_{i}u_i+\beta_{i-1}u_{i-1}=\tilde{v}_{i}(1:m), \\
		& \hat{\alpha}_{i}\hat{u}_i+\hat{\beta}_{i-1}\hat{u}_{i-1}=\tilde{v}_{i}(m+1:m+p).
	\end{align*}
	Adding the squares of norms of the above two equalities leads to
	\[\alpha_{1}^2+\beta_{i-1}^{2}+\hat{\alpha}_{i}^{2}+
	\hat{\beta}_{i-1}^{2}+2(\alpha_{i}\beta_{i-1}u_{i-1}^{T}u_{i}+\hat{\alpha}_{i}\hat{\beta}_{i-1}\hat{u}_{i-1}^{T}\hat{u}_{i})=1 .\]
	For $i=1$, we have $\alpha_{1}^2+\hat{\alpha}_{i}^{2}=1$ due to $\beta_{0}=\hat{\beta}_{0}=0$. For $i>1$, since
	\begin{align*}
		\alpha_iu_{i-1}^{T}u_i
		&= u_{i-1}^{T}(Q_{A}v_i-\beta_{i-1}u_{i-1})
		 = u_{i-1}^{T}Q_{A}v_i-\beta_{i-1} \\
		&= v_{i}^{T}(\beta_{i-1}v_i+\alpha_{i-1}v_{i-1}+\sum_{j=1}^{i-1}\xi_{ji}v_j+g_{i-1})-\beta_{i-1} \\
		&= v_{i}^{T}g_{i-1},
	\end{align*}
	by \cref{alp_beta} we have
	\[|\alpha_i\beta_{i-1}u_{i-1}^{T}u_i| \leq\beta_{i-1}\|g_{i-1}\|
	 \leq [1+\mathcal{O}(i\kappa(C)\tau)]\|g_{i-1}\|= \mathcal{O}(\kappa(C)\tau) .\]
	By \Cref{lem4.5} we have
	\begin{align*}
		|\hat{\alpha}_{i}\hat{\beta}_{i-1}\hat{u}_{i-1}^{T}\hat{u}_{i}|
		&= |\hat{\beta}_{i-1}\hat{u}_{i-1}^{T}(Q_L\hat{v}_i-\hat{\beta}_{i-1}\hat{u}_{i-1})| \\
		&= |\hat{v}_{i}^{T}(\hat{\beta}_{i-1}Q_{L}^{T}\hat{u}_{i-1})-\hat{\beta}_{i-1}^{2}| \\
		&= |\hat{v}_{i}^{T}(\hat{\beta}_{i-1}^{2}\hat{v}_{i}+\widehat{V}_{i-1}l_{i-1}+
		q_{i-1})-\hat{\beta}_{i-1}^{2}| \\
		&= \mathcal{O}(\theta_{i-1}\kappa(C)\tau)
	\end{align*}
	Therefore we obtain
	\[-2(\alpha_{i}\beta_{i-1}u_{i-1}^{T}u_{i}+\hat{\alpha}_{i}\hat{\beta}_{i-1}\hat{u}_{i-1}^{T}\hat{u}_{i})
	= \mathcal{O}(\theta_{i-1}\kappa(C)\tau) ,\]
	which is the $i$-th diagonal of $H_{k}$.
\end{proof}

Similar to relations \cref{3.10,3.11}, there is a couple of recursive relations describing the reduction process from $Q_L$ to $\widehat{B}_k$.  
\begin{theorem}\label{thm4.4}
	The following relations hold for the $k$-step rJBD process:
	\begin{align}
	& Q_L\widehat{V}_k = \widehat{U}_k\widehat{B}_k, \label{Qb1} \\
	& Q_{L}^{T}\widehat{U}_{k}=\widehat{V}_{k}(\widehat{B}_{k}^{T}+\widehat{D}_{k}) +\hat{\beta}_{k}\hat{v}_{k+1}(e_{k}^{(k)})^{T}+\widehat{G}_{k} , \label{Qb2}
	\end{align}
	where $\widehat{D}_{k}$ is upper triangular, and
	\begin{equation}\label{3.17}
	\| \widehat{G}_{k}\|
	= \mathcal{O}(\|\widehat{B}_{k}^{-1}\|\sqrt{n}\kappa(C)\tau).
	\end{equation}
\end{theorem}
\begin{proof}
	Relation \cref{Qb1} is just \cref{3.12}. Combining \cref{3.10,3.11}, we have
	\begin{align*}
	Q_{A}^{T}Q_{A}V_{k} &=Q_{A}^{T}U_{k}(B_{k}) 
	= [V_{k}(B_{k}^{T}+D_{k})+\beta_{k}v_{k+1}(e_{k}^{(k)})^{T}+G_{k}]B_{k} \\
	&= V_{k}B_{k}^{T}B_{k}+\alpha_{k}\beta_{k}v_{k+1}(e_{k}^{(k)})^{T}+
	V_{k}D_{k}B_{k}+ G_{k}B_{k} .
	\end{align*}
	Premultiply \cref{Qb1} by $Q_{L}^{T}$, we have
	\begin{align*}
	Q_{L}^{T}Q_{L}V_{k}=Q_{L}^{T}\widehat{U}_{k}\widehat{B}_{k}P .
	\end{align*}
	Adding the above two equalities and using \cref{thm4.3}, we obtain
	\begin{align*}
	V_k &= (Q_{A}^{T}Q_{A}+Q_{L}^{T}Q_{L})V_{k} \\
	&=  V_{k}[I_{k}-P\widehat{B}_{k}^{T}\widehat{B}_{k}P+H_{k}]
	+Q_{L}\widehat{U}_{k}\widehat{B}_{k}P+\alpha_{k}\beta_{k}v_{k+1}(e_{k}^{(k)})^{T}+ V_{k}D_{k}B_{k}+
	G_{k}B_{k}.
	\end{align*}
	Using $\alpha_{k}\beta_{k}=\hat{\alpha}_k\hat{\beta}_k$ and $\hat{v}_{k+1}=(-1)^kv_{k+1}$, after some rearrangements we obtain 
	\begin{align*}
	\widehat{V}_{k}\widehat{B}_{k}^{T}\widehat{B}_{k}
	&= Q_{L}^{T}\widehat{U}_{k}\widehat{B}_{k}-\hat{\alpha}_{k}\hat{\beta}_{k}\hat{v}_{k+1}(e_{k}^{(k)})^{T}
	+\widehat{V}_kP(D_kB_k+H_k)P+G_kB_kP ,
	\end{align*}
	Therefore, we have
	\[ Q_{L}^{T}\widehat{U}_k=\widehat{V}_k(\widehat{B}_{k}^{T}+\widehat{D}_k)+
	\hat{\beta}_k\hat{v}_{k+1}(e_{k}^{(k)})^{T}-G_kB_kP\widehat{B}_{k}^{-1},\]
	where $\widehat{D}_k=-P(D_kB_k+H_k)P\widehat{B}_{k}^{-1}$ is upper triangular. Relation \cref{Qb2} is obtained by letting $\widehat{G}_{k}=-G_kB_kP\widehat{B}_{k}^{-1}$. By \cref{thm4.1} we get
	\[\|B_k\|=\|\bar{U}_{k}^{T}(Q_A+E_k)V_{k}\|\leq\|Q_A\|+\|E_k\|\leq 1+\mathcal{O}(\sqrt{n}k\kappa(C)\tau) .\]
	By neglecting high order terms of $\tau$ in $\|G_kB_kP\widehat{B}_{k}^{-1}\|$, we finally obtain the upper bound on $\|\widehat{G}_{k}\|$.
\end{proof}

Since $\widehat{D}_{k}$ is upper triangular, if we write the matrix $\widehat{D}_{k}$ as 
\[\widehat{D}_{k}=\begin{pmatrix}
	\hat{\xi}_{11} &\cdots &\cdots  & \hat{\xi}_{1k} \\
	& \hat{\xi}_{22} &\cdots & \hat{\xi}_{2k} \\
	&  &\ddots &\vdots \\
	& &  & \hat{\xi}_{kk} \\
\end{pmatrix} \in \mathbb{R}^{k \times k} ,\] 
then \Cref{thm4.4} implies for each $i = 1, \dots, k$ that
\[\hat{\beta}_{i}\hat{v}_{i+1} = Q_{L}^{T}\hat{u}_{i}-\hat{\alpha}_{i}\hat{v}_{i}-
\sum_{j=1}^{i}\hat{\xi}_{ji}\hat{v}_{j} - \hat{g}_{i} \]
with $\|\hat{g}_{i}\|= \mathcal{O}(\|\widehat{B}_{k}^{-1}\|\sqrt{n}\kappa(C)\tau)$, which corresponds to the reorthogonalization of $\hat{v}_{i}$ with error term $\hat{g}_{i}$, where $\hat{\xi}_{ji}$ are coefficients appeared in the reorthogonalization. Based on \cref{Qb1,Qb2}, a similar result as \Cref{thm4.1} about the bidiagonal reduction of $Q_L$ can also be obtained.

\section{Convergence and accuracy of the approximate GSVD components}\label{sec5}
The results of \Cref{sec4} can be used to investigate the convergence and accuracy of GSVD components computed by rJBD. First we give a brief review on the JBD based GSVD computation. For the GSVD \cref{2.2} of $\{A,L\}$, let $X=(x_1,\ldots,x_n)$, $P_A=(p_{A,1},\ldots,p_{A,m})$ and $P_L=(p_{L,1},\ldots,p_{L,p})$. Then it can be written in the vector form:
\begin{equation*}
	\left\{
	\begin{aligned}
		& Ax_i=c_i p_{A,i}\\
		& Lx_i=s_i p_{L,i}\\
		& s_i A^Tp_{A,i}=c_iL^{T}p_{L,i}
	\end{aligned}
	\right.
	\qquad	
\end{equation*}
for $i=1,\dots,n$, where the $i$-th largest generalized singular value is $c_{i}/s_{i}$, and the corresponding generalized singular vectors are $x_{i}$, $p_{A,i}$ and $p_{L,i}$, respectively. We also use pair $\{c_i, s_i\}$ to denote a generalized singular value. Note that each $x_i$ satisfies the normalization condition:
\begin{equation}\label{normal}
	x_{i}^{T}(A^TA+L^TL)x_i = 1 .
\end{equation}
In this paper, we only consider approximations to $\{c_i, s_i\}$ and the corresponding right generalized singular vector $x_{i}$. In order to approximate left generalized singular vectors $p_{A,i}$ and $p_{L,i}$, a strategy extracting information from $\mbox{span}(U_{k})$ and $\mbox{span}(\widehat{U}_k)$ is needed since $U_{k}$ and $\widehat{U}_k$ are not column orthonormal. For a possibly worked method, one can refer to \cite[Section 6]{Barlow2013}.

Assume the compact SVD of $B_{k}$ is computed as
\begin{equation}\label{svd_Bk}
	B_{k} = P_{k}\Theta_{k}Y_{k}^{T}, \ \ \Theta_{k}=\mbox{diag}(c_{1}^{(k)}, \dots,
	c_{k}^{(k)}), \ \
	1 \geq c_{1}^{(k)} > \dots > c_{k}^{(k)} \geq 0,
\end{equation}
where $P_{k}=(p_{1}^{(k)}, \dots, p_{k}^{(k)})$ and $Y_{k}=(y_{1}^{(k)}, \dots, y_{k}^{(k)})$ are $k$-by-$k$ orthogonal matrices. Since $c_{i}^{2}+s_{i}^{2}=1$, we only need to compute $c_i$, and the approximate generalized singular values are $\{c_{i}^{(k)}, s_{i}^{(k)}\}$  with $s_{i}^{(k)}=(1-(c_{i}^{(k)})^{2})^{1/2}$. The approximate right generalized singular vectors are $x_{i}^{(k)}=R^{-1}V_{k}y_{i}^{(k)}$ for $i=1,\dots,k$. Recall from \Cref{sec2} that $R$ is invertible under the assumption that $\{A,L\}$ is a regular matrix pair. It is shown in \cite{Zha1996} that the explicit computation of $R^{-1}$ can be avoided to compute $x_{i}^{(k)}$ by solving
\begin{equation}\label{ls_x}
	\begin{pmatrix}
		A \\ L
	\end{pmatrix}x_{i}^{(k)}=QRR^{-1}V_{k}y_{i}^{(k)}=\widetilde{V}_{k}y_{i}^{(k)}
\end{equation}
iteratively. The above approximations can also be obtained by the SVD of $\widehat{B}_{k}$, which is connected with that of $B_k$ by \cref{3.18}. Detailed discussions about the SVD of $B_k$ and $\widehat{B}_k$ when $H_k\neq 0$ can be found in \cite[Section 4]{JiaLi2019}. Here we do not discuss it any more.

Now we investigate the final accuracy of computed GSVD components by the SVD of $B_{k}$. Suppose that the algorithm is stopped at the $k_0$-th step, and the singular values and right singular vectors of $B_{k_0}$ are $c_{i}^{(k_0)}$ and $w_{i}^{(k_0)}$ for $1\leq i\leq k_0$. By \Cref{thm4.1}, $c_{i}^{(k_0)}$ and $V_{k_0}w_{i}^{(k_0)}$ will approximate the SVD components of $\bar{Q}_{A}=Q_A+E_{k_0}$ since $B_{k_0}$ is the Ritz-Galerkin projection of $\bar{Q}_{A}$ on subspaces $\mathrm{span}(\bar{U}_{k_0})$ and $\mathrm{span}(V_{k_0})$. In order to analyze the final accuracy, we use the following assumption.

\begin{assumption}\label{assum2}
Denote the $i$-th largest singular value of $\bar{Q}_{A}$ by $\bar{c}_i$ with corresponding right singular vector $\bar{w}_i$. We assume at the $k_0$-th step that
\begin{equation}
	|\bar{c}_i-c_{i}^{(k_0)}| \ll \tau, \ \ \ 
	\|\bar{w}_i-V_{k_0}y_{i}^{(k_0)}\| \ll \tau .
\end{equation}
\end{assumption}
This assumption can always be satisfied for a sufficiently large $k_0 \leq n$, since $B_{k_0}$ can be used to approximate the SVD components of the $m\times n$ matrix $\bar{Q}_{A}$.

\begin{theorem}\label{thm5.1}
	For the rJBD based GSVD computation by the SVD of $B_{k}$ which stops at $k_0$ such that \Cref{assum2} is satisfied, it follows for any $1\leq i\leq k_0$ that
	\begin{equation}\label{sv_accur}
		|c_i-c_{i}^{(k_0)}| = \mathcal{O}(\sqrt{n}k_0\kappa(C)\tau) .
	\end{equation}
	Suppose the multiplicity of $\{c_i,s_i\}$ is $1$ and \cref{ls_x} is solved exactly. Let $\gamma_i=\min\{c_{i-1}-c_i, c_i-c_{i+1}\}$ for $1<i<k_0$ and $\gamma_1=c_1-c_2$, $\gamma_{k_0}=c_{k_0-1}-c_{k_0}$. For any $1\leq i\leq k_0$, if $\|E_{k_0}\|<\gamma_i$, then 
	\begin{equation}\label{rsv_accur}
		\|x_i-x_{i}^{(k_0)}\| /\|R^{-1}\| = \mathcal{O}\Big(\frac{\sqrt{n}k_0\kappa(C)\tau}{\gamma_{i}-\|E_{k_0}\|}\Big) .
	\end{equation}
\end{theorem}
\begin{proof}
Notice that the SVD of $Q_A$ is $Q_A=P_AC_AW^T$. By the perturbation theory of singular values (see e.g., \cite[Corollary 8.6.2]{Golub2013}), we have
\begin{equation*}
	|c_i-\bar{c}_i| \leq \|Q_{A}-\bar{Q}_{A}\|= \|E_{k_0}\| = \mathcal{O}(\sqrt{n}k_0\kappa(C)\tau) .
\end{equation*}
Therefore, under \Cref{assum2} we have
\[|c_i-c_{i}^{(k_0)}|\leq |c_i-\bar{c}_i|+|\bar{c}_i-c_{i}^{(k_0)}| = \mathcal{O}(\sqrt{n}k_0\kappa(C)\tau) .\]

For $c_i$ that is a singular value of $Q_A$ with multiplicity $1$, by the perturbation theorem of singular vectors \cite[Theorem 1.2.8]{Bjorck1996}, we have the perturbation bound 
\begin{equation*}
	|\sin \theta(w_{i},\bar{w}_i)|\leq \frac{\lVert E_{k_0}\lVert}{\gamma_{i}-\lVert E_{k_0} \lVert},
\end{equation*}
which leads to $\|w_{i}-\bar{w}_i\| = \mathcal{O}\Big(\frac{\sqrt{n}k_0\kappa(C)\tau}{\gamma_{i}-\lVert E_{k_0} \lVert}\Big)$
by neglecting high order terms of $\tau$. Note that $x_{i}^{(k_0)}=R^{-1}V_{k_0}y_{i}^{(k_0)}$. Under \Cref{assum2} we have
\begin{align}
	\|x_i-x_{i}^{(k_0)}\| 
	&\leq \|R^{-1}w_i-R^{-1}\bar{w}_i\| + \|R^{-1}\bar{w}_i-R^{-1}V_{k_0}y_{i}^{(k_0)}\|  \nonumber \\
	&\leq \|R^{-1}\|\|w_{i}-\bar{w}_i\| +\|R^{-1}\|\|\bar{w}_i-V_{k_0}y_{i}^{(k_0)}\|  \nonumber \\
	&= \mathcal{O}\Big(\frac{\|R^{-1}\|\sqrt{n}k_0\kappa(C)\tau}{\gamma_{i}-\lVert E_{k_0} \lVert}\Big),  \nonumber
\end{align}
Dividing both sides by $\|R^{-1}\|$, the upper bound is obtained.
\end{proof}

We remark that the matrix-size/iteration-step dependent constant $\sqrt{n}k_0$ in $\mathcal{O}(\cdot)$ is nonessential, because it is introduced only for the end to estimate an upper bound for $\|E_{k_{0}}\|$. Note that the convergence rate of $c_{i}^{(k)}$ and $w_{i}^{(k)}$ as $k$ increases from $1$ to $k_0$ mainly depends on the convergence rate of approximating the SVD of $B_{k}$. Thus \Cref{thm5.1} implies that the final accuracy of approximate GSVD components is limited by $\kappa(C)\tau$ while the convergence rate is not affected too much. Combining \cref{svd_Bk} with \cref{3.18} we have $\bar{B}_{k}^{T}\bar{B}_{k}-H_k= Y_{k}(I_k-\Theta_{k}^2)Y_{k}^{T}$, which implies that the singular values of $\bar{B}_k$ are determined by that of $B_k$ within errors of order $\mathcal{O}(\kappa(C)\tau)$. Thus if we want to use the SVD of $\bar{B}_k$ to approximate $s_i$, the final accuracy is also limited by the value of $\kappa(C)\tau$; this will be illustrated by a numerical example in \Cref{sec6}.

Note that the normalization condition \cref{normal} is $x_{i}^{T}R^TRx_i=1$. The expression \cref{rsv_accur} can be regarded as another form of relative error. It indicates that the final accuracy of approximate right generalized singular vectors depends not only on the value of $\kappa(C)\tau$ but also on the gap between generalized singular values. For singular values with multiplicity bigger than $1$, the computation of invariant singular subspaces instead of single singular vectors are usually considered. Although in this case the mathematical expression is a bit more complicated, the spirit is similar to the approach of obtaining \cref{rsv_accur}; interested readers can refer to \cite{Wedin1972} or \cite[\S 5.4]{StewartSun1990}.

Finally, we investigate the solution accuracy of \cref{ls_x} for getting the final $x_{i}^{(k_0)}$. Suppose \cref{ls_x} is solved iteratively using stopping criterion \cref{stop} with tolerance $\bar{\tau}$ and the corresponding solution is $\bar{x}_{i}^{(k_0)}$. Using the same approach as that for establishing \cref{3.3}, we have
\begin{equation}\label{sov_accur}
	\dfrac{\|x_{i}^{(k_0)}-\bar{x}_{i}^{(k_0)}\|}{\|x_{i}^{(k_0)}\|}
	\leq \frac{\kappa(C)\bar{\tau}}{1-\kappa(C)\bar{\tau}} \Big(1+\frac{\kappa(C)\lVert \tilde{s}_{i}\lVert}{\lVert C\lVert\lVert x_{i}^{(k_0)}\lVert}\Big)
	= \frac{\kappa(C)\bar{\tau}}{1-\kappa(C)\bar{\tau}} ,
\end{equation}
where the residual $\tilde{s}_{i}=\widetilde{V}_{k}y_{i}^{(k_0)}-Cx_{i}^{(k_0)}=0$ since \cref{ls_x} is consistent. Comparing \cref{rsv_accur,sov_accur}, the relative error $\|x_{i}^{(k_0)}-\bar{x}_{i}^{(k_0)}\|/\|x_{i}^{(k_0)}\|$ need not be much smaller than $\kappa(C)\tau/\gamma_i$. For a well-conditioned $C$, values $\bar{\tau}\in[0.1\tau, 10\tau]$ are often feasible as illustrated by experimental results.

\section{Experimental results}\label{sec6}
We report some experimental results to justify the theoretical results obtained. All numerical experiments are performed in  MATLAB R2019b, where all computations are carried out using double precision with roundoff unit $2^{-53}\approx 1.11\times 10^{-16}$. The codes are available at \url{https://github.com/Machealb/gsvd_iter}.

The tested matrices are mainly taken from the SuiteSparse matrix collection \cite{Davis2011} with the same names. The matrix $L_{1d}\in\mathbb{R}^{(n-1)\times n}$ is a bidiagonal matrix with one more row than columns and values $-1$ and $1$ on the subdiagonal and diagonal parts, respectively. The matrix pair $\{A_1,L_1\}$ is constructed as follows. Set $m=n=p=10000$. Let $C_A=\mathrm{diag}(\{c_i\}_{i=1}^{n})$ with $c_i=(n-i+1)/2n$ and $S_L=\mathrm{diag}(\{s_i\}_{i=1}^{n})$ with $s_i=(1-c_{i}^2)^{1/2}$. Then let $D$ be a diagonal matrix generated by the MATLAB built-in function $D=\texttt{diag(linspace(1,1e5,n))}$. Finally let $A_1=C_AD$ and $L_1=S_LD$. By the construction, the QR factorization of 
$C=\begin{pmatrix}
	A_1 \\ L_1
\end{pmatrix}$ is $\begin{pmatrix}
	A_1 \\ L_1
\end{pmatrix} = \begin{pmatrix}
	C_A \\ S_L
\end{pmatrix}D$ with $Q=\begin{pmatrix}
	C_A \\ S_L
\end{pmatrix}$, and $\kappa(C)=\kappa(D)=10^5$.
The properties of the matrices are described in \Cref{tab1}. Note the scaling factor $0.1$ before $L_{1d}$, which ensures a faster convergence of LSQR for inner iterations. For each matrix pair, we use the random vector $s=\texttt{randn(n,1)}$ with random seed $\texttt{rng(2022)}$ as the starting vector for JBD and rJBD. The inner least squares problem is iteratively solved by the LSQR with stopping criterion \cref{stop}.

\begin{table}[htbp]\label{tab1}
	\footnotesize
	\caption{Properties of the test matrices}
	\begin{center}
		\begin{tabular}{llll}
			\toprule
			Matrix pair & $m,p,n$  & $\kappa(C)$ & Description  \\
			\midrule
			\{\sf illc1850, well1850\}  & $1850, 1850, 712$ & 38.6 & least squares problem    \\
			\{\sf dw2048, rdb2048\} & $2048, 2048, 2048$ & 261.0 & electromagnetics problem  \\
			\{\sf swang1, $0.1L_{1d}$\}  & $3169, 3168, 3169$ & 45.0 & semiconductor device problem  \\
			\{\sf $A_1$, $L_1$\} & $10000, 10000, 10000$  & 100000.0 & self-constructed  \\
			\bottomrule
		\end{tabular}
	\end{center}
\end{table}

\begin{figure}[htbp]
	\centering
	\subfloat[\{\sf illc1850, well1850\}]{\label{fig:1a}\includegraphics[width=0.4\textwidth]{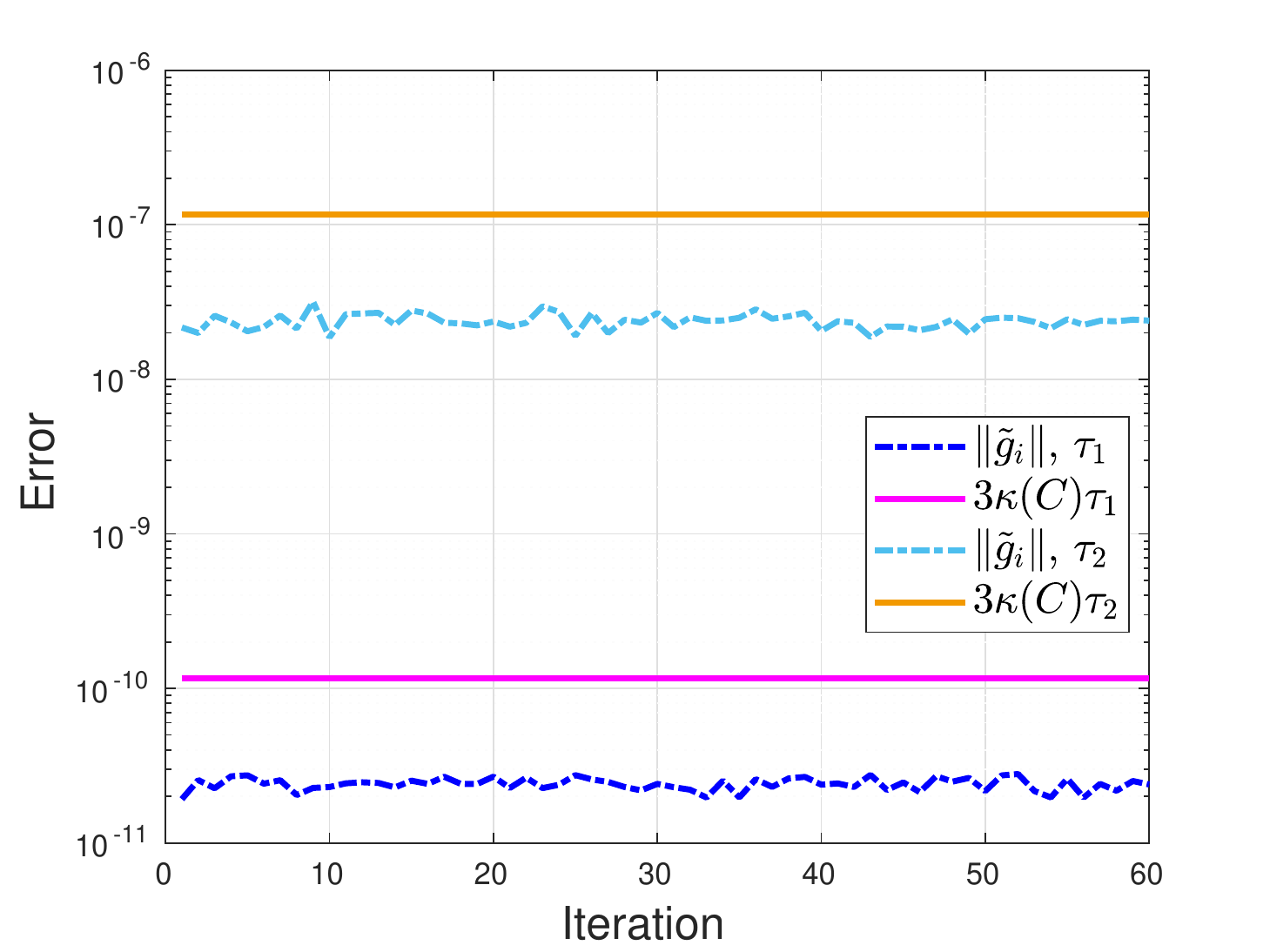}}	
	\subfloat[\{\sf dw2048, rdb2048\}]{\label{fig:1b}\includegraphics[width=0.4\textwidth]{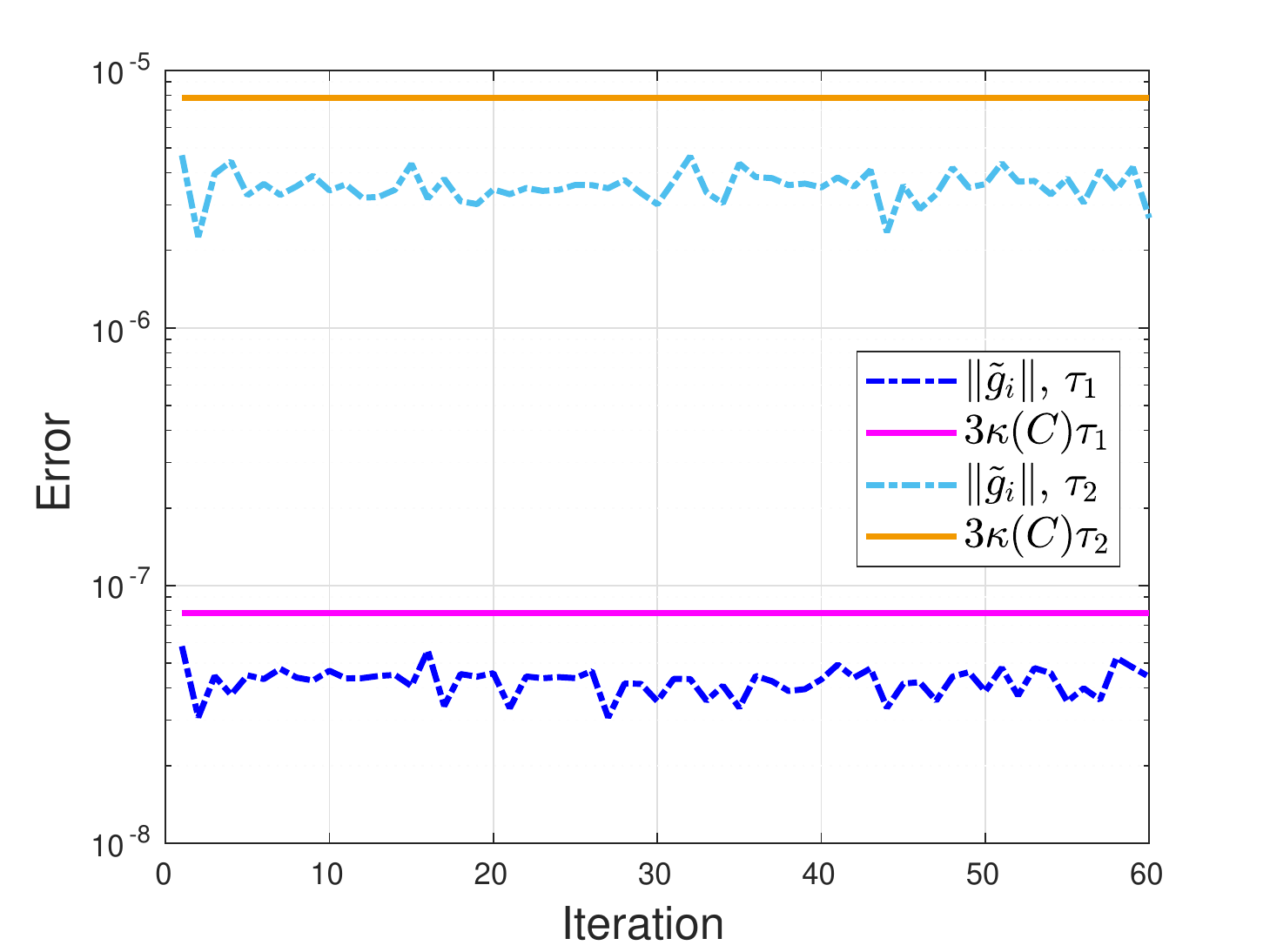}}
	\quad  %
	\subfloat[\{\sf swang1, $0.1L_{1d}$\}]{\label{fig:1c}\includegraphics[width=0.4\textwidth]{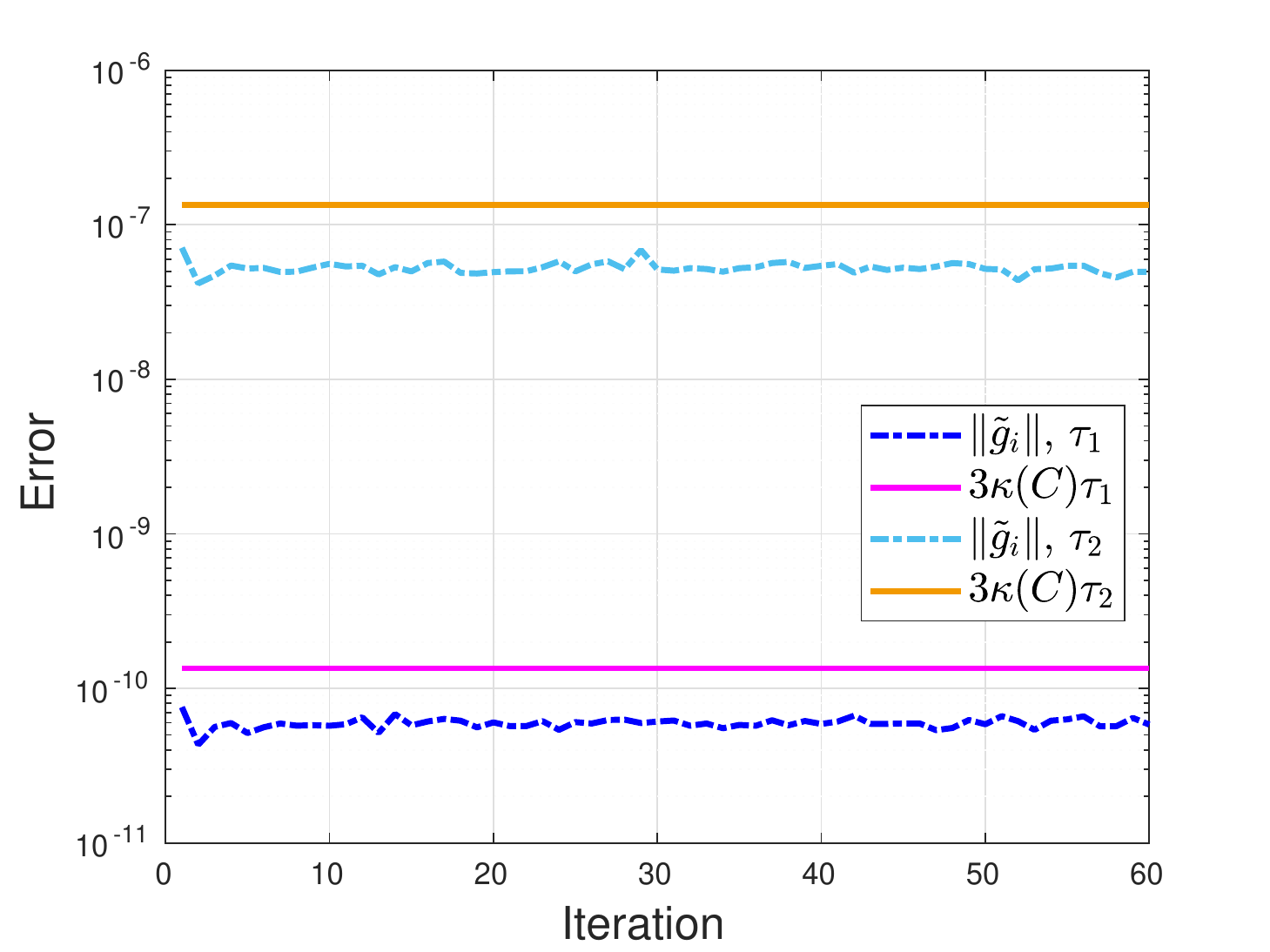}}	
	\subfloat[\{\sf $A_1$, $L_1$\}]{\label{fig:1d}\includegraphics[width=0.4\textwidth]{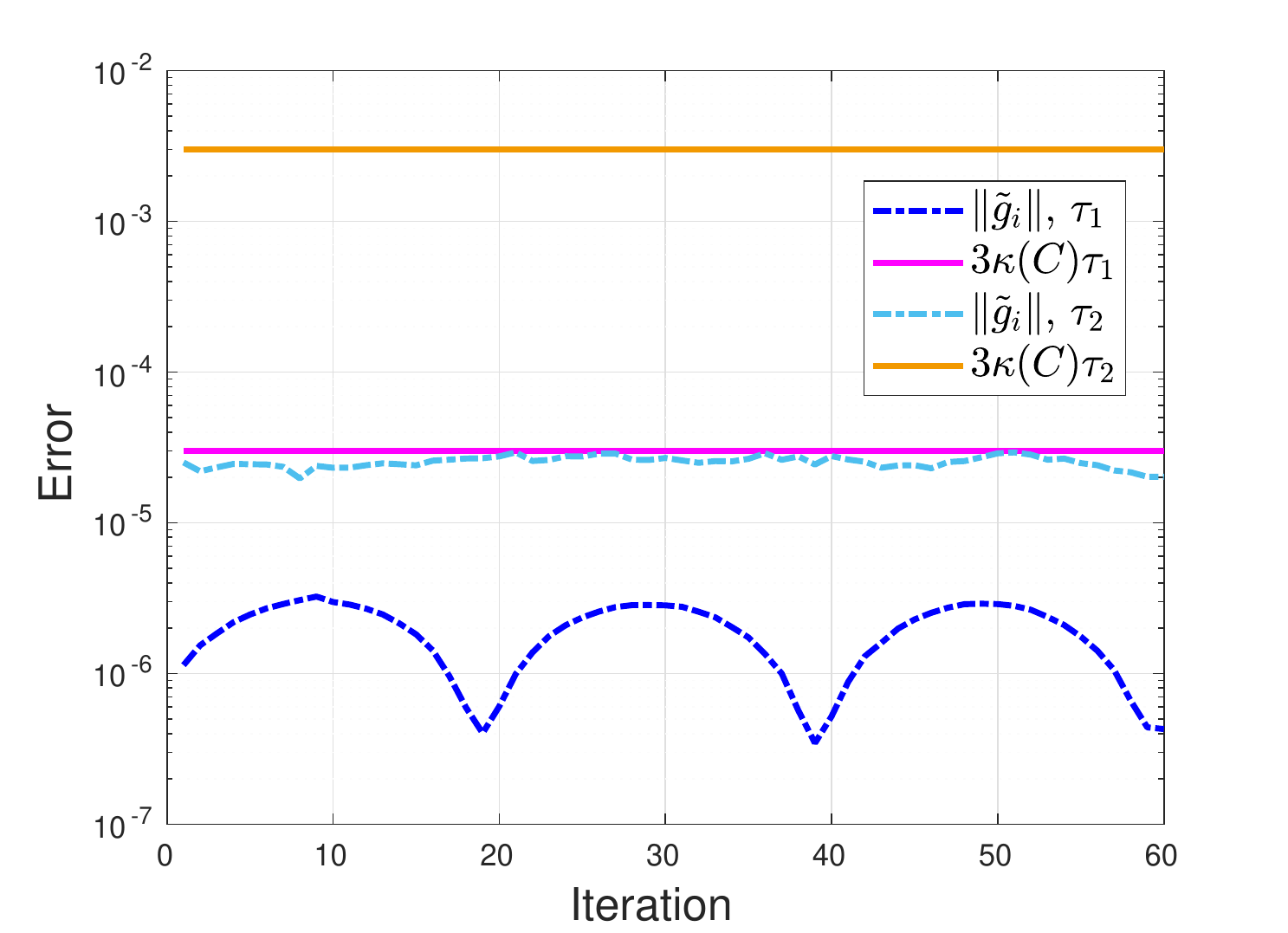}}
	\caption{Computational error $\tilde{g}_i$ and its upper bound for JBD. For \{\sf illc1850,well1850\} and \{\sf swang1,$0.1L_{1d}$\}, $\tau_1=10^{-12}$, $\tau_2=10^{-9}$; for \{\sf dw2048,rdb2048\} and \{\sf $A_1$,$L_1$\}, $\tau_1=10^{-10}$, $\tau_2=10^{-8}$.}
	\label{fig1}
\end{figure}

\begin{figure}[htbp]
	\centering	
	\subfloat[\{\sf illc1850, well1850\}, JBD]{\label{fig:12a}\includegraphics[width=0.4\textwidth]{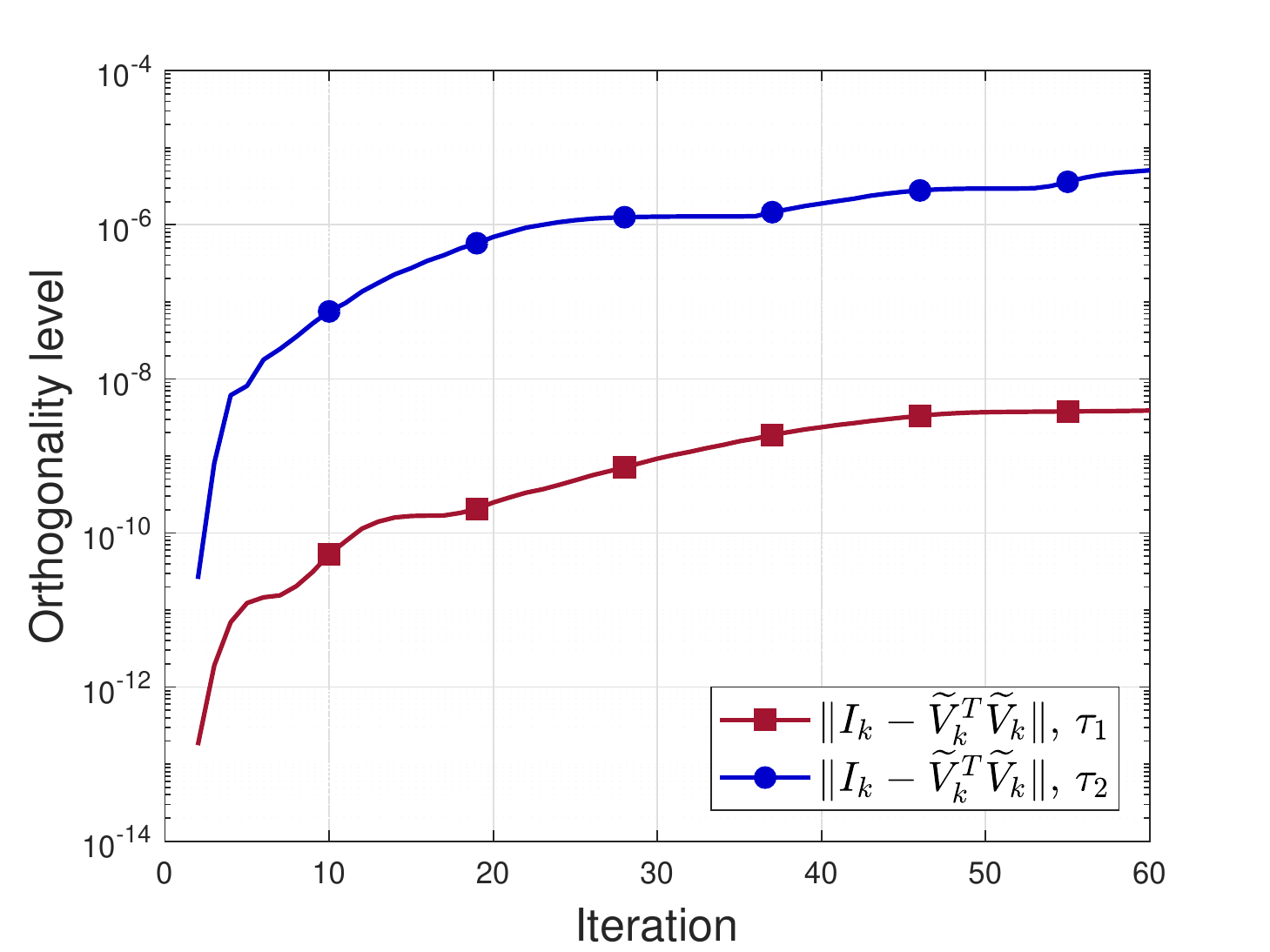}}
	\subfloat[\{\sf dw2048, rdb2048\}, JBD]{\label{fig:12b}\includegraphics[width=0.4\textwidth]{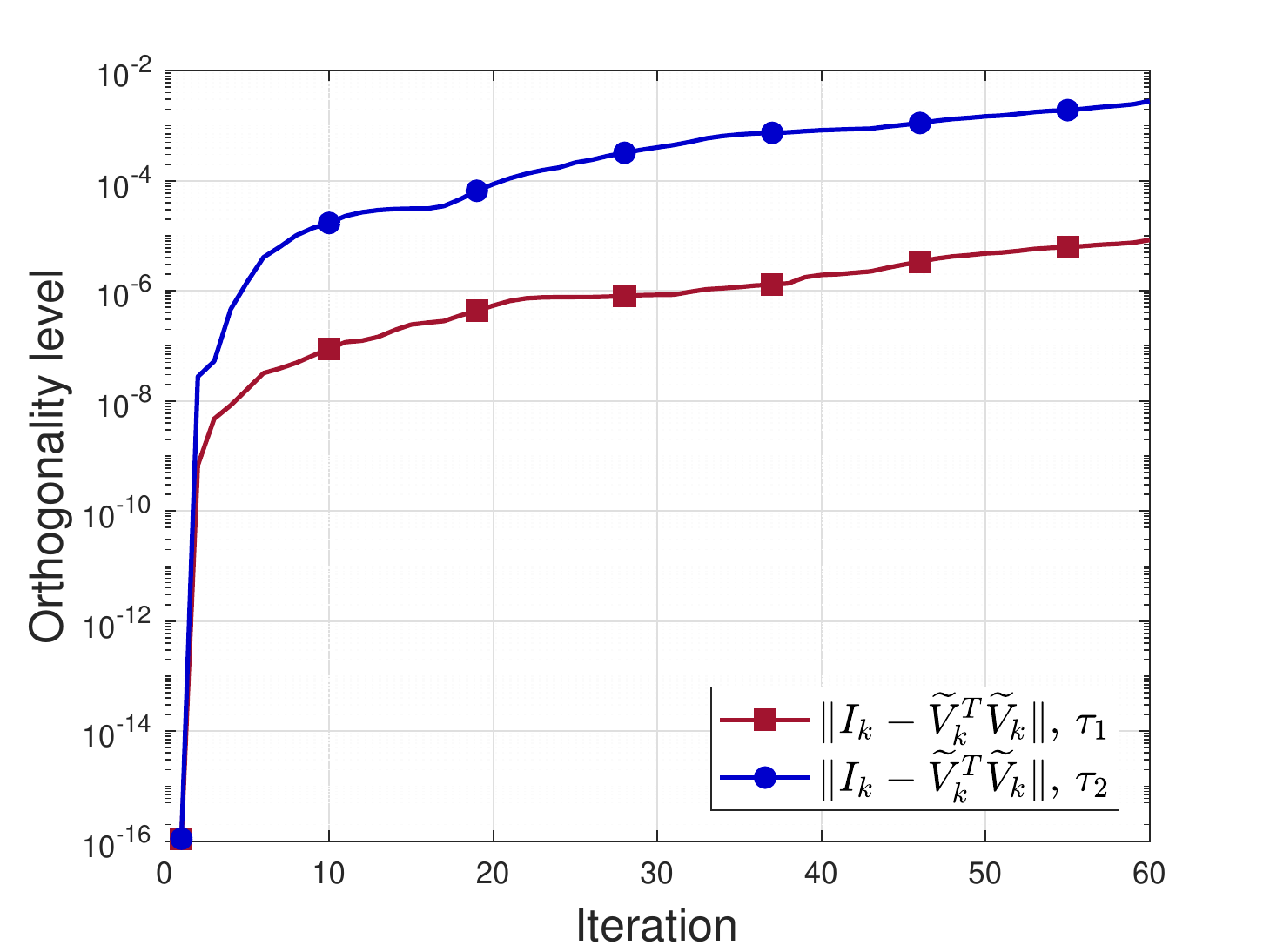}}
	\quad  %
	\subfloat[\{\sf illc1850, well1850\}, rJBD, $\tau=\tau_1$]{\label{fig:12c}\includegraphics[width=0.41\textwidth]{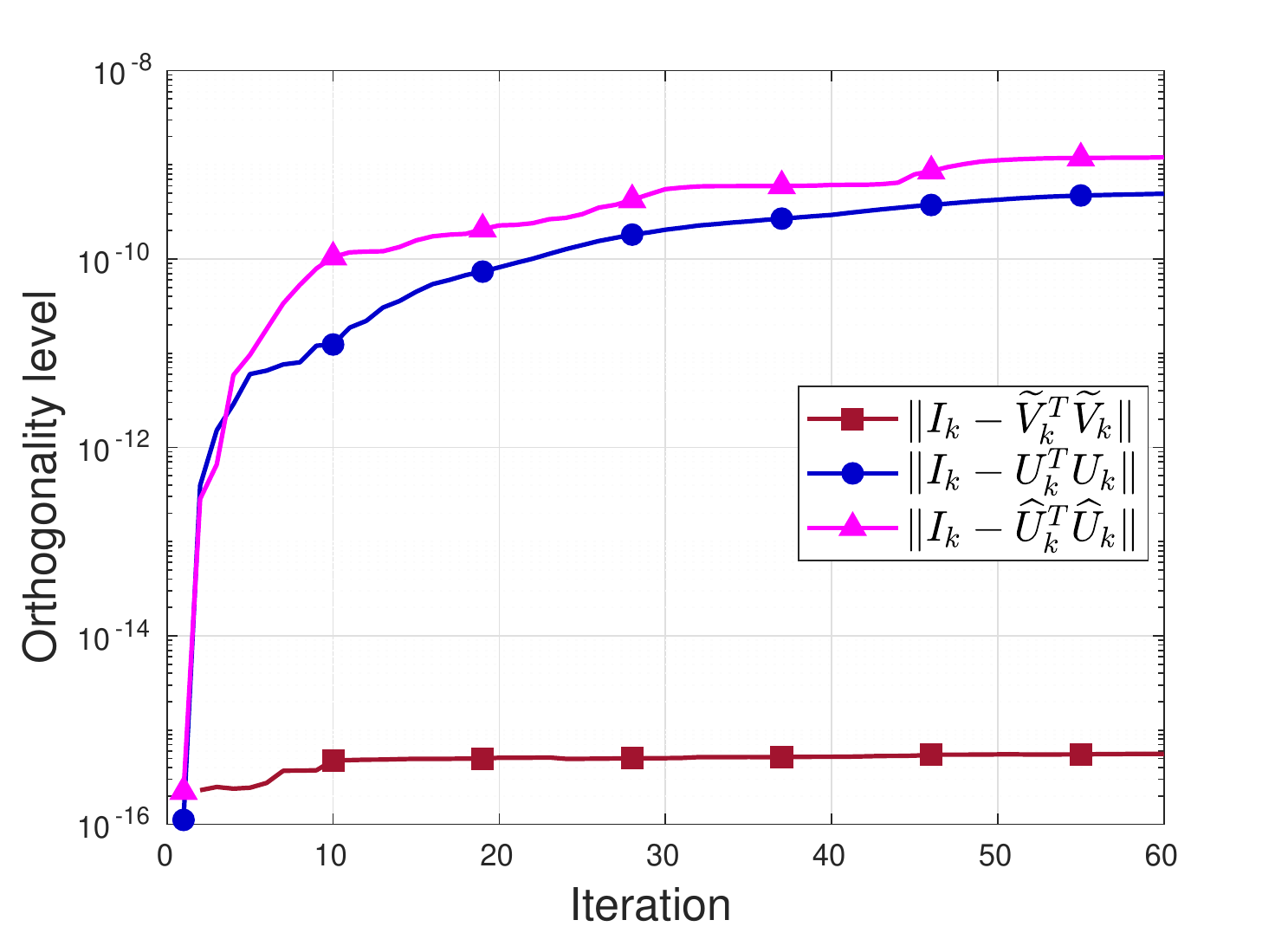}}
	\subfloat[\{\sf dw2048, rdb2048\}, rJBD, $\tau=\tau_1$]{\label{fig:12d}\includegraphics[width=0.41\textwidth]{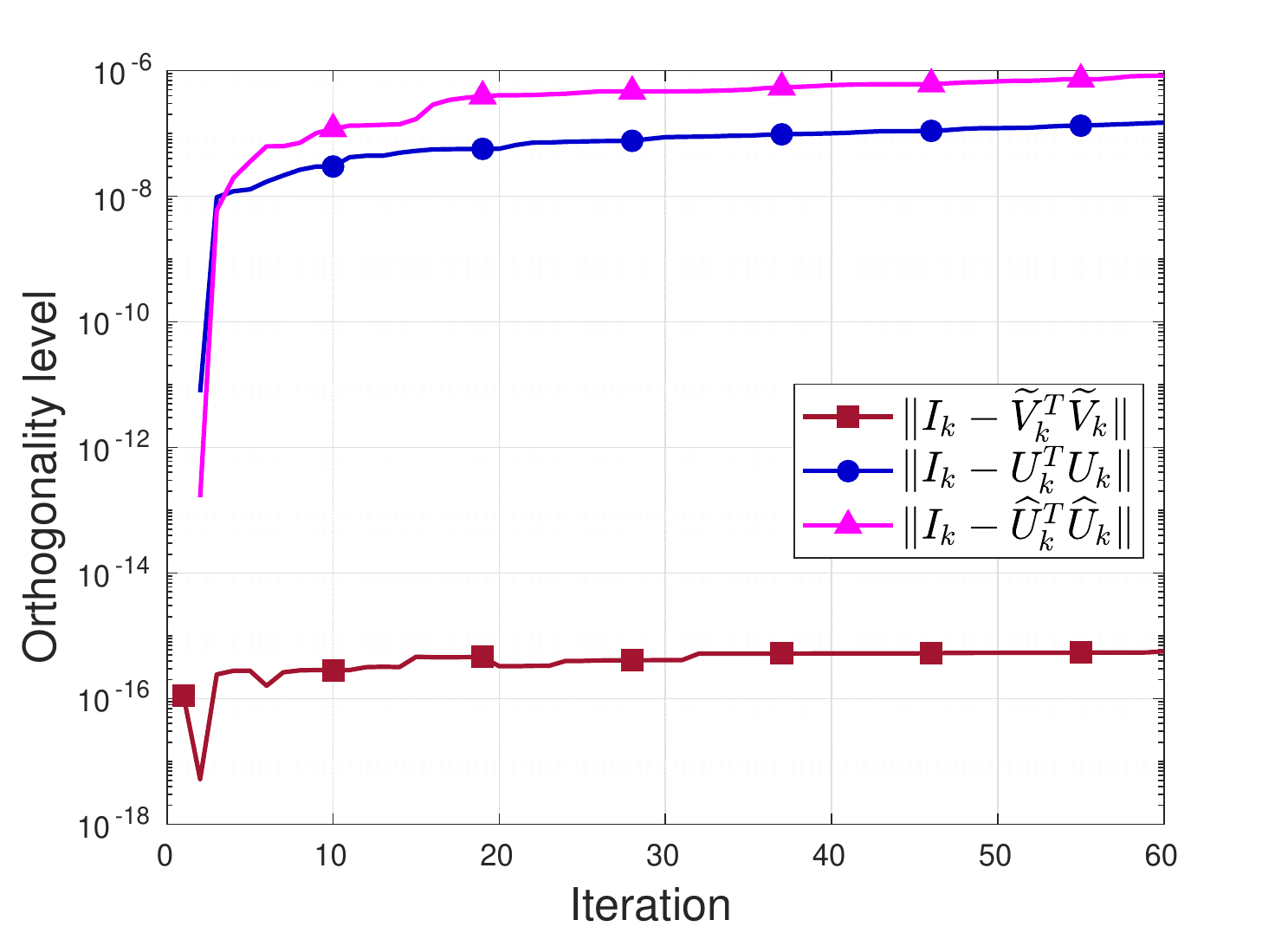}}
	\caption{Loss of orthogonality of $\tilde{v}_i$ for JBD and orthogonality level of $\tilde{v}_i$, $u_i$ and $\hat{u}_i$ for rJBD.}
	\label{fig1.2}
\end{figure}

\paragraph*{Experiments for JBD and rJBD}
We first use some numerical experiments to confirm the theoretical results about JBD and rJBD process. \Cref{fig1} shows the numerical behavior of JBD where inner iterations are computed using LSQR with stopping tolerance $\tau$. The computational error $\tilde{g}_i$ and its upper bound \cref{bound_g} of the four test matrix pairs are drawn. For each test matrix pair and each $\tau=\tau_1$ or $\tau=\tau_2$, we find that $\|\tilde{g}_i\|$ varies slightly and $3\kappa(C)\tau$ is an upper bound, which confirms \cref{thm3.1}. Note that for \{$A_1$, $L_1$\} the upper bound is more over-estimated than others, this is because the upper bound in \cref{3.4} is more likely to be over-estimated when the condition number is very large. Besides, in the experiments we have found that for \cref{ls} it takes extremely many iterations (even large than $n$) of LSQR to achieve the desired accuracy described by \cref{stop}. In this case, a proper preconditioner or a scaling factor transforming $\{A,L\}$ to $\{A,\gamma L\}$ can be very useful to accelerate convergence. \Cref{fig1.2} depicts the orthogonality level of $\tilde{v}_i$, $u_i$ and $\hat{u}_i$ measured by $\|I_{k}-\widetilde{V}_{k}^{T}\widetilde{V}_k\|$ and so on when inner iterations are computed inexactly. We can find that a large $\tau$ leads to a more quickly loss of orthogonality of $\tilde{v}_i$. This phenomenon has already been observed in \cite{Zha1996}, the reason of which is revealed by \cref{thm3.2}. For rJBD that applies full reorthogonalization to $\tilde{v}_i$, the orthogonality level of $\tilde{v}_i$ is kept around a value close to the roundoff unit, but the orthogonality of $u_i$ and $\hat{u}_i$ still loses gradually since reorthogonalization is not applied to them.

\begin{figure}[htbp]
	\centering
	\subfloat[\{\sf illc1850, well1850\}, $\tau=10^{-10}$]{\label{fig:2a}\includegraphics[width=0.4\textwidth]{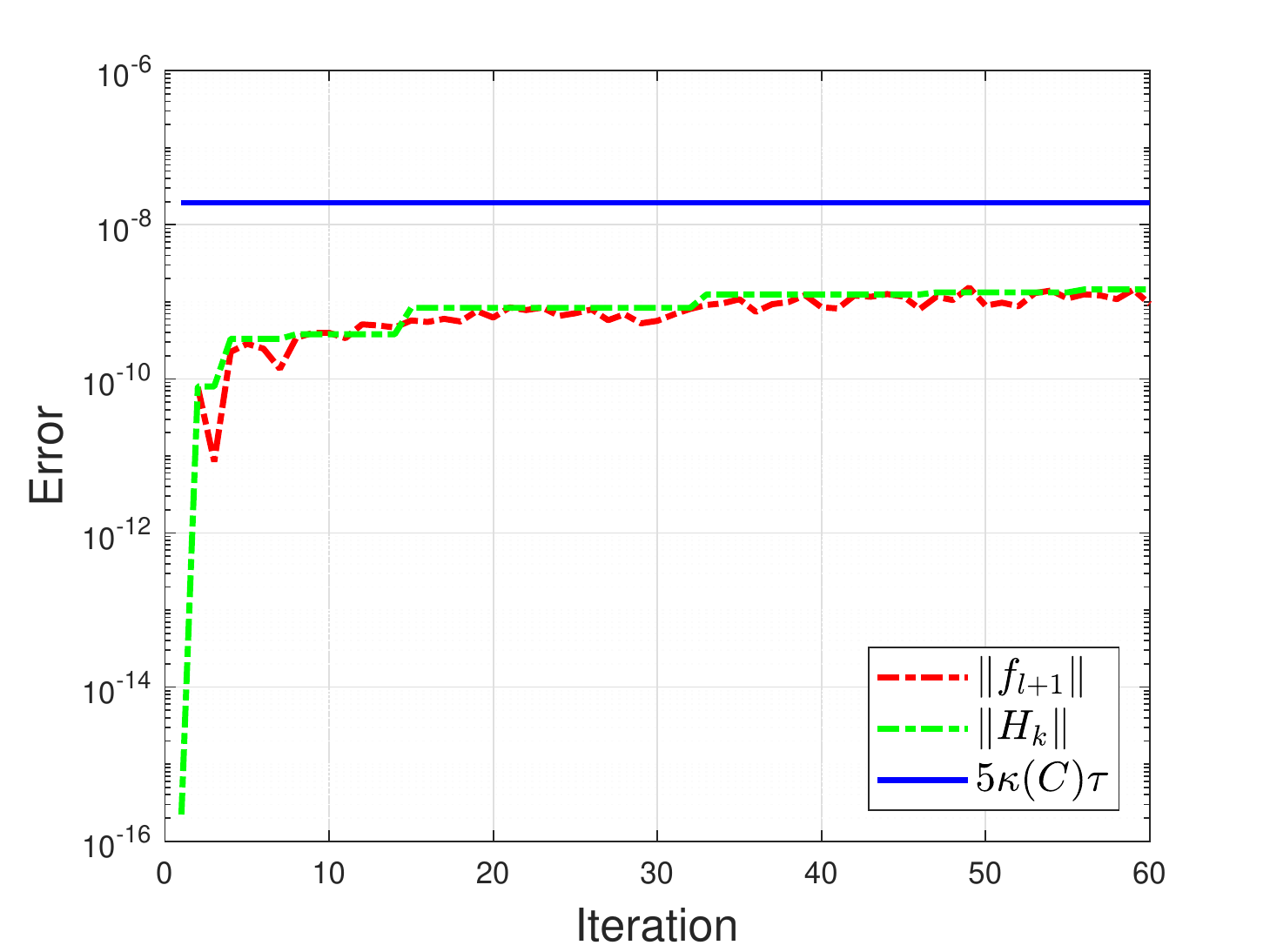}}
	\subfloat[\{\sf dw2048, rdb2048\}, $\tau=10^{-9}$]{\label{fig:2b}\includegraphics[width=0.4\textwidth]{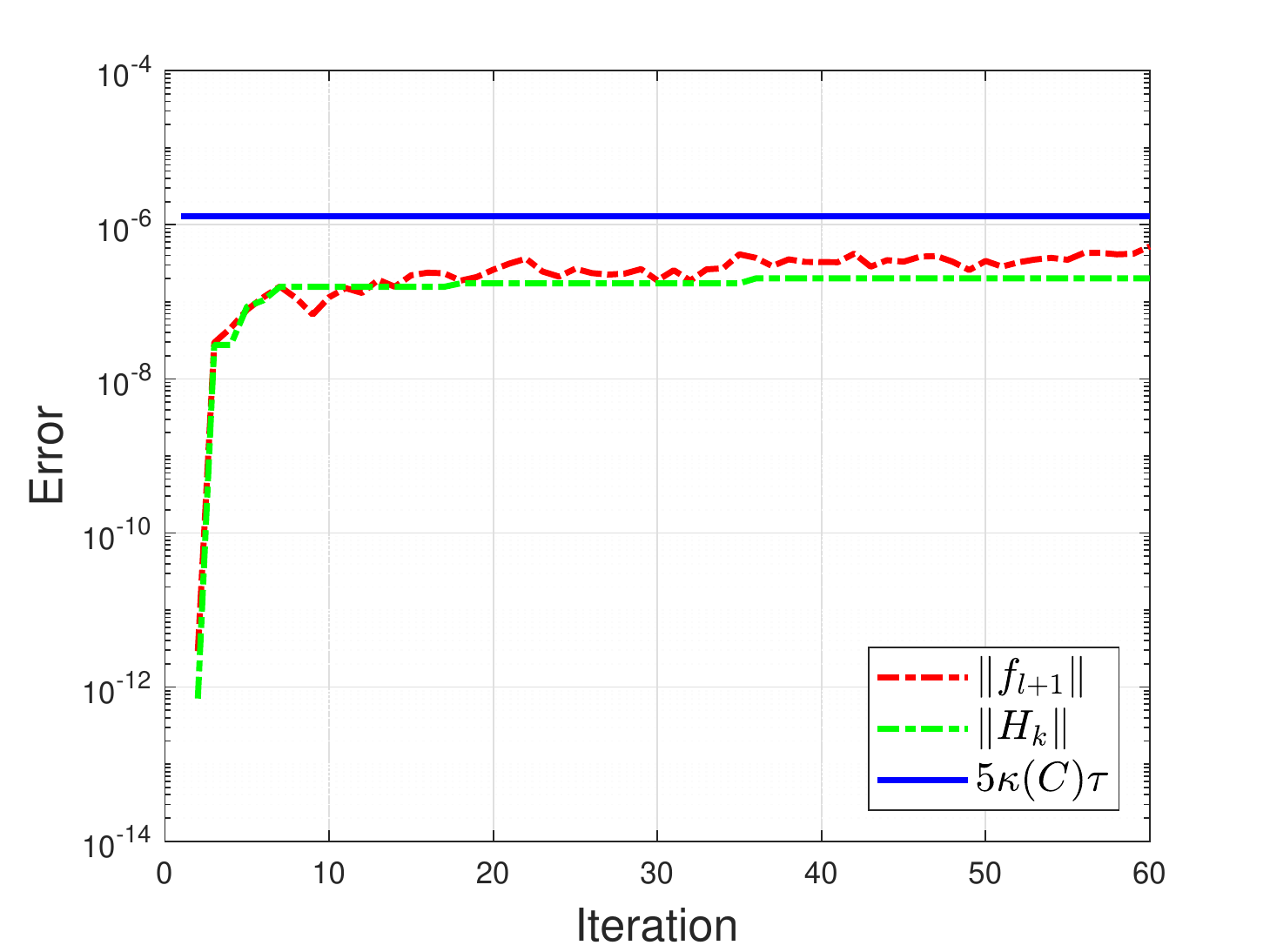}}
	\quad  %
	\subfloat[\{\sf swang1, $0.1L_1$\}, $\tau=10^{-8}$]{\label{fig:2c}\includegraphics[width=0.4\textwidth]{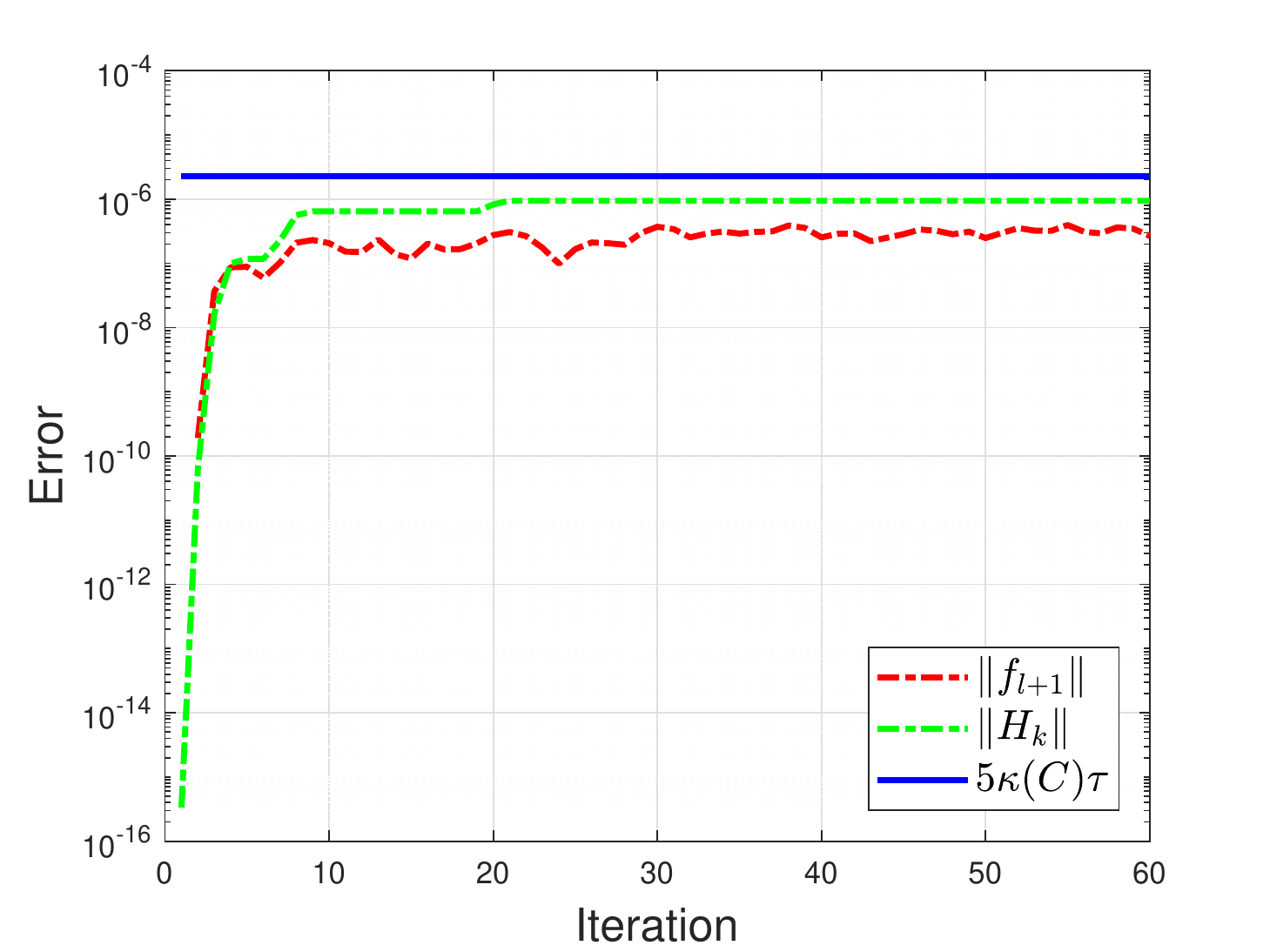}}	
	\subfloat[\{\sf $A_1$, $B_1$\}, $\tau=10^{-7}$]{\label{fig:2d}\includegraphics[width=0.41\textwidth]{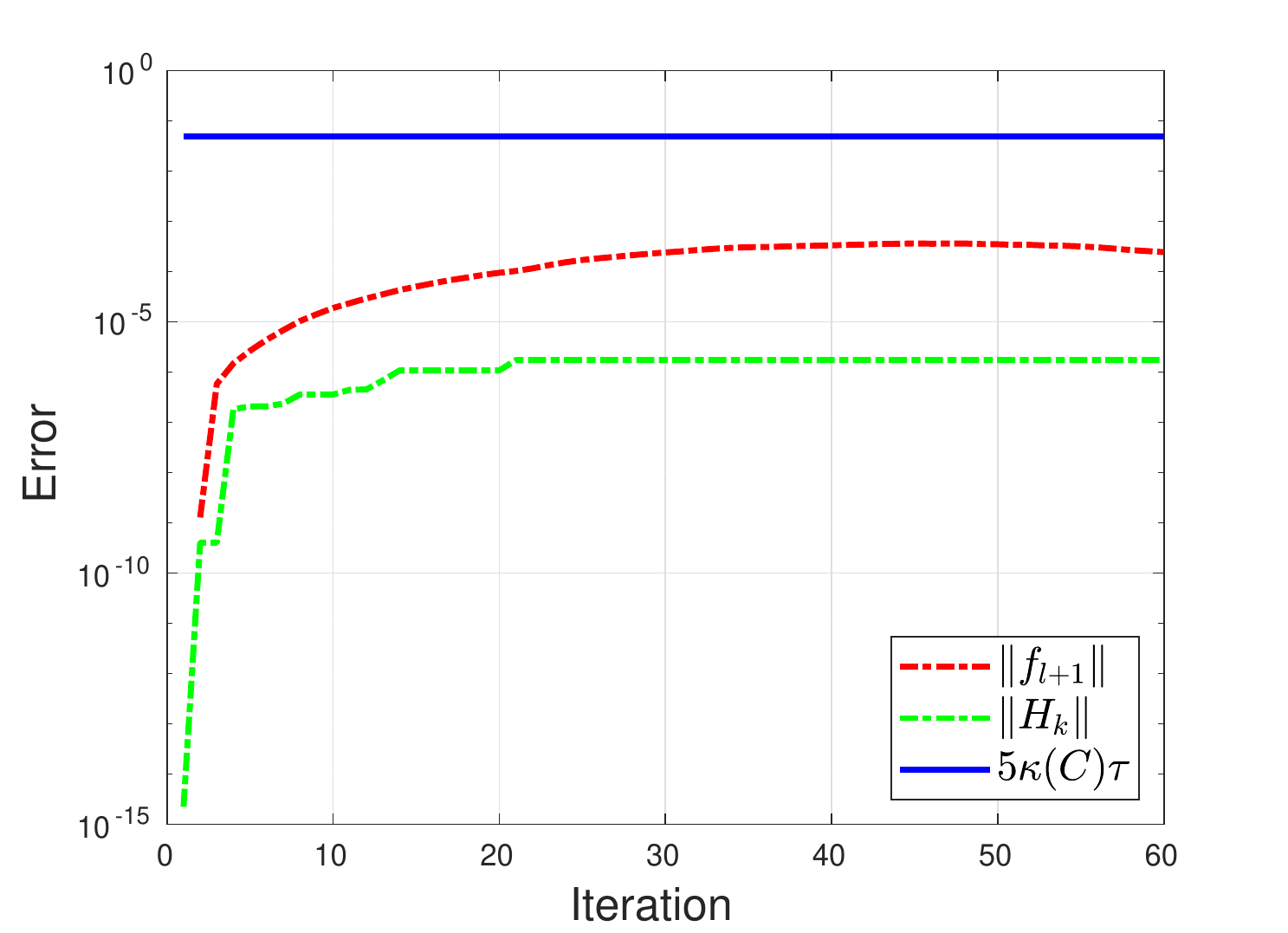}}
	\caption{The norm of error terms $f_{l+1}$, $H_k$ and their upper bounds for rJBD.}
	\label{fig2}
\end{figure}

\begin{figure}[htbp]
	\centering
	\subfloat[\{\sf illc1850, well1850\}, $\tau=10^{-10}$]{\label{fig:3a}\includegraphics[width=0.4\textwidth]{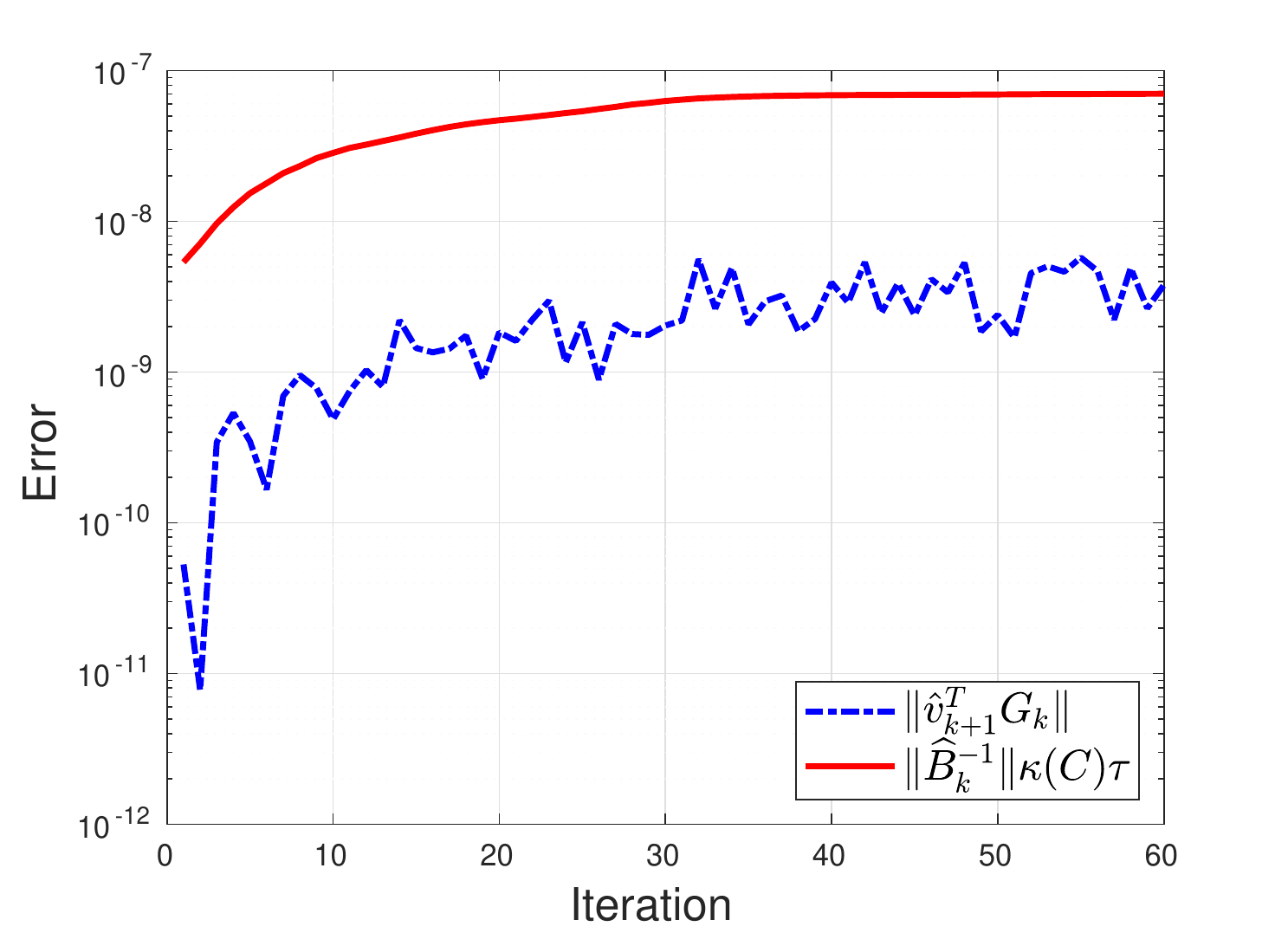}}	
	\subfloat[\{\sf dw2048, rdb2048\}, $\tau=10^{-9}$]{\label{fig:3b}\includegraphics[width=0.4\textwidth]{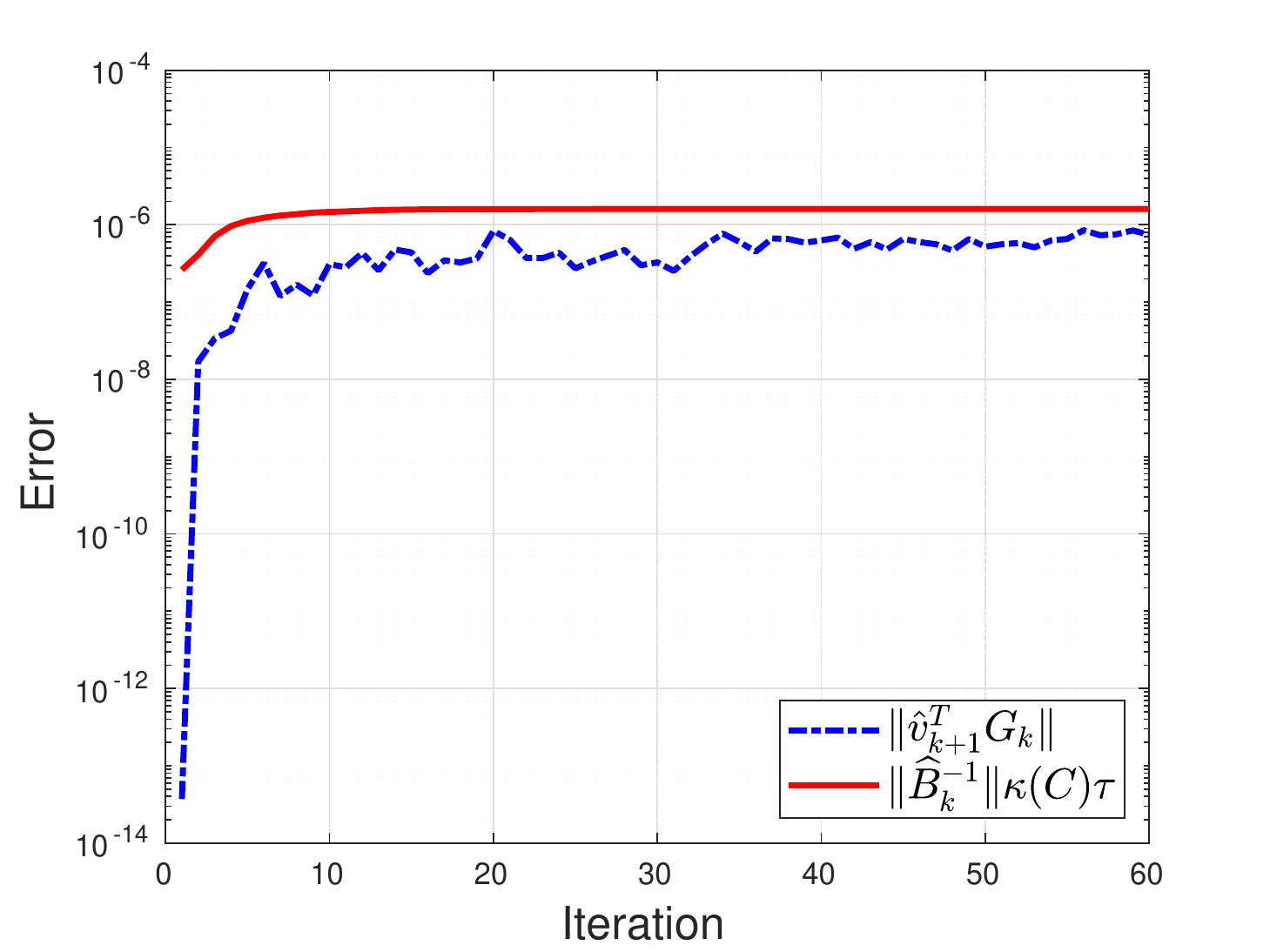}}
	\quad  %
	\subfloat[\{\sf swang1, $0.1L_1$\}, $\tau=10^{-8}$]{\label{fig:3c}\includegraphics[width=0.4\textwidth]{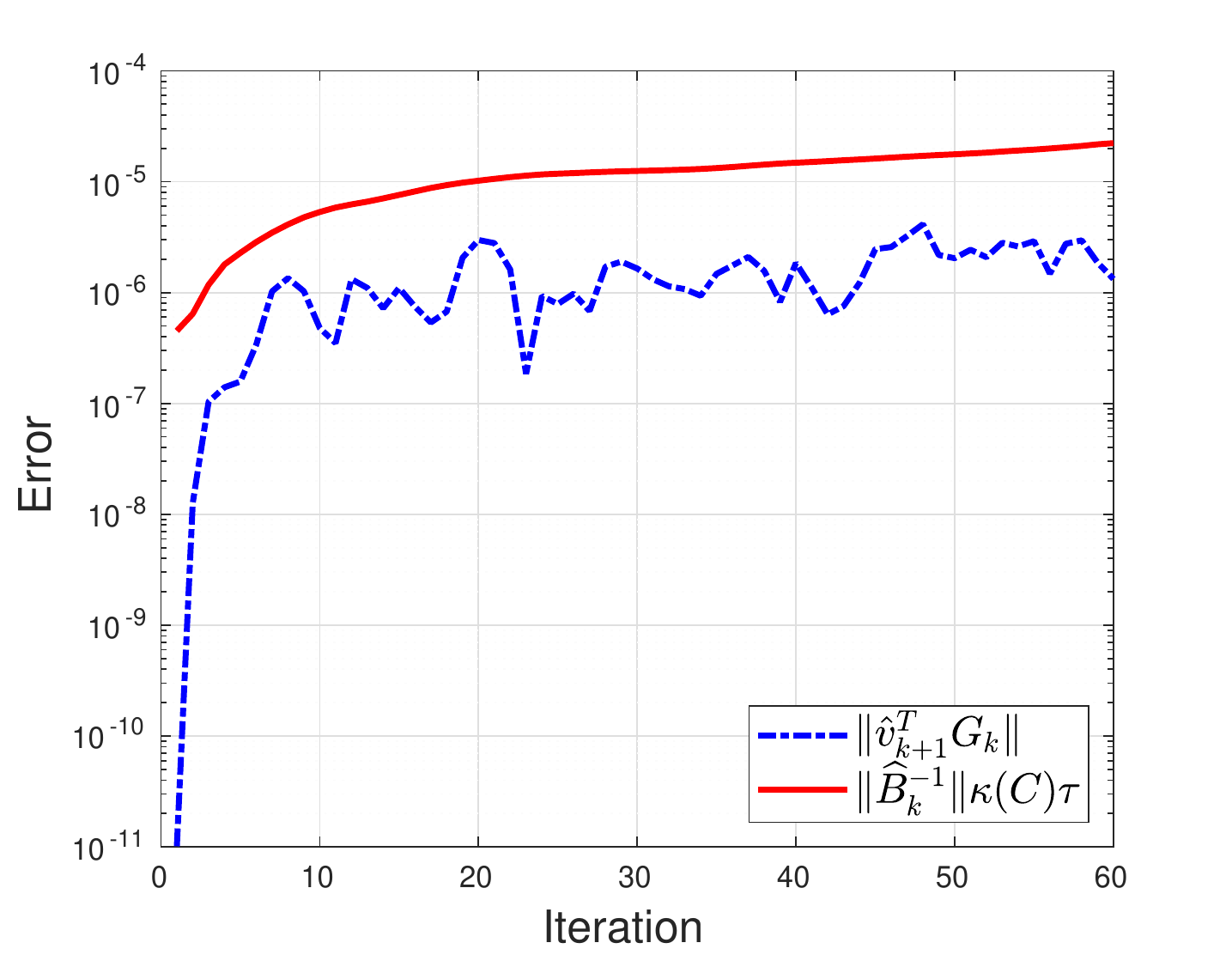}}	
	\subfloat[\{\sf $A_1$, $B_1$\}, $\tau=10^{-7}$]{\label{fig:3d}\includegraphics[width=0.41\textwidth]{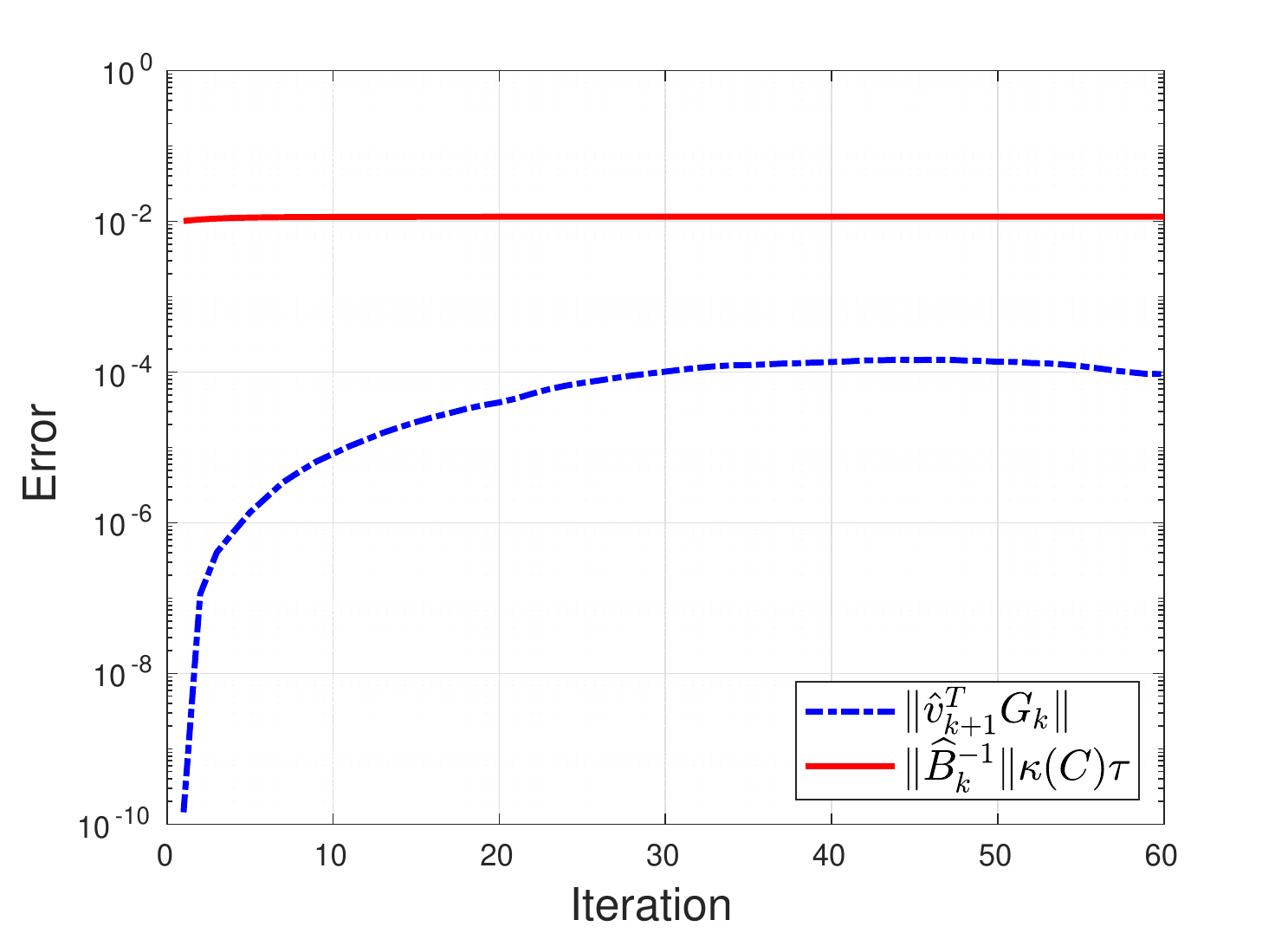}}
	\caption{The norm of error terms $\tilde{v}_{k+1}^{T}\widehat{G}_k$ and its upper bound for rJBD.}
	\label{fig3}
\end{figure}

\Cref{fig2} and \cref{fig3} illustrate our error analysis results of rJBD. For the four test examples, both the norms of $f_{l+1}$ and $H_k$ grow slightly with respect to iteration number $l$ or $k$, and they can be controlled by $\kappa(C)\tau$ times a moderate constant for the rJBD with not too many iterations. This confirms \Cref{lem4.1} and \Cref{thm4.3}. The relation \cref{3.18} indicates that the singular values of $\widehat{B}_k$ are determined by that of $B_k$ with a perturbation of order $\|H_k\|=\mathcal{O}(\kappa(C)\tau)$. Therefore, together with \cref{thm4.1}, we can expect that the absolute errors of approximate generalized singular values computed by the SVD of $B_k$ or $\widehat{B}_k$ are both of order $\mathcal{O}(\kappa(C)\tau)$ and the convergence behavior of the two are very similar as long as the value of $\kappa(C)\tau$ is small to a certain extent, as will be shown in \Cref{fig4}. Since $\widehat{G}_k$ can not be computed explicitly, we depict $\|\tilde{v}_{k+1}^{T}\widehat{G}_k\|$ with $\tilde{v}_{k+1}^{T}\widehat{G}_k=\tilde{v}_{k+1}^{T}Q_{L}^{T}\widehat{U}_{k}-\hat{\beta}_{k}e_{k}^{(k)}$ by \cref{Qb2}. We observe that $\|\tilde{v}_{k+1}^{T}\widehat{G}_k\|$ grows slightly with $\|\widehat{B}_{k}^{-1}\kappa(C)\tau\|$ an upper bound, which confirms \Cref{thm4.4}. For \{$A_1$, $L_1$\} the upper bounds are more over-estimated than others, the reason of which is the same as the case of $\tilde{g}_i$.

\paragraph*{Experiments for GSVD computations}
Then we illustrate the numerical behavior of the rJBD method for partial GSVD computations. The matrix pair $\{A_2,L_2\}$ used in the experiments is constructed as follows. Set $m=n=p=800$. Let $C_A=\mathrm{diag}(\{c_i\}_{i=1}^{n})$ with $c(1)=0.99, c(2)=0.98, c(3)=0.97$, $c(4:n-3)=\texttt{linspace(0.96,0.04,n-6)}$ and $c(n-2)=0.03, c(n-1)=0.02, c(n)=0.01$, and $S_B=\mathrm{diag}(\{s_i\}_{i=1}^{n})$ with $s_i=(1-c_{i}^2)^{1/2}$. Then let $W$ be an orthogonal matrix by letting $W=\texttt{gallery(`orthog',n,2)}$, and $D=\texttt{diag(linspace(1,10,n))}$. Finally let $A_2=C_AW^TD$ and $L_2=S_LW^TD$. By the construction, the generalized singular values of $\{A_2,L_2\}$ are $c_i/s_i$ with right generalized singular vectors be the $i$-th column of $D^{-1}W$, and $\kappa(C)=10$.

\begin{figure}[htbp]
	\centering
	\subfloat[JBD]{\label{fig:4a}\includegraphics[width=0.4\textwidth]{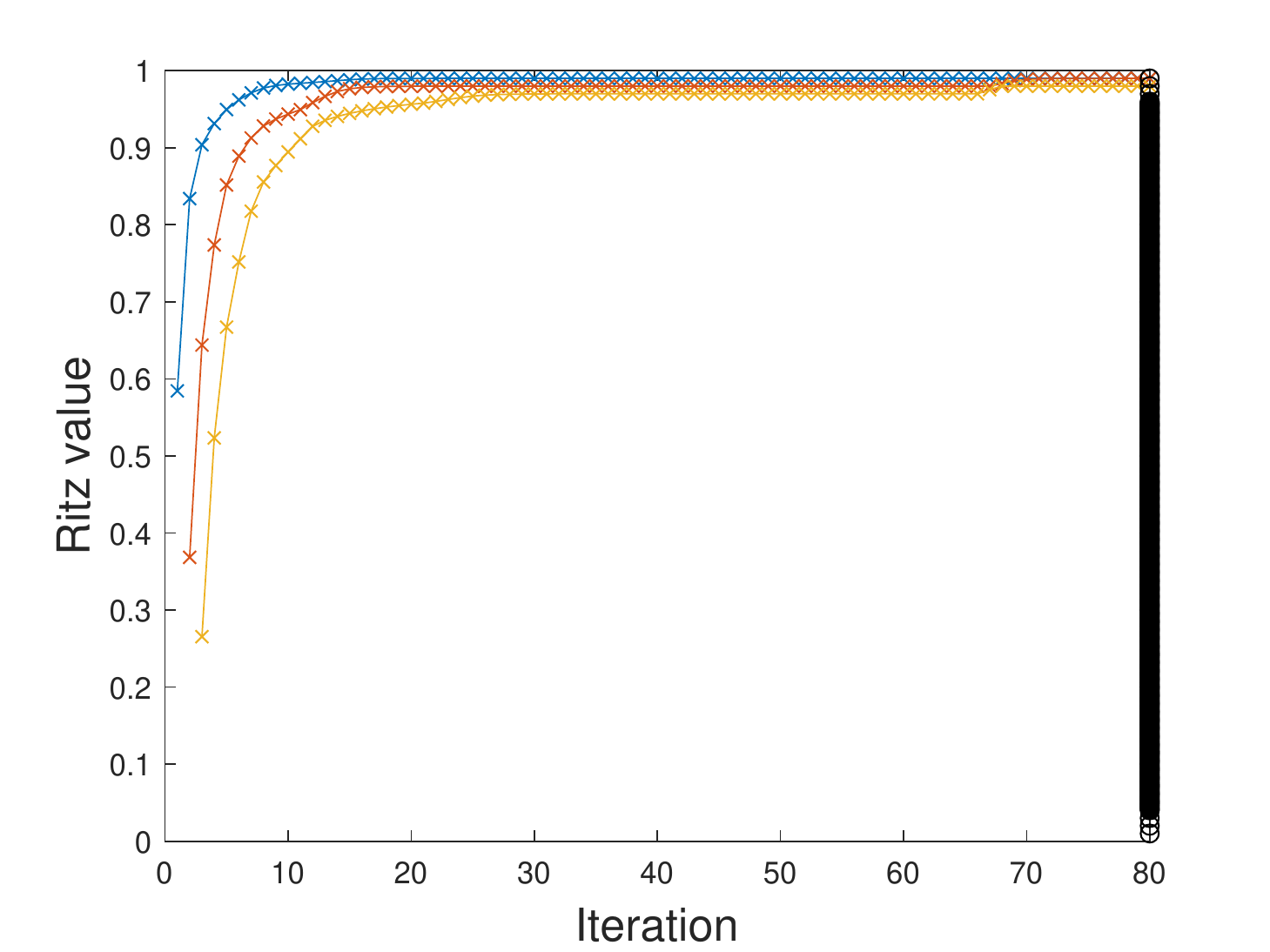}}	\subfloat[rJBD]{\label{fig:4b}\includegraphics[width=0.4\textwidth]{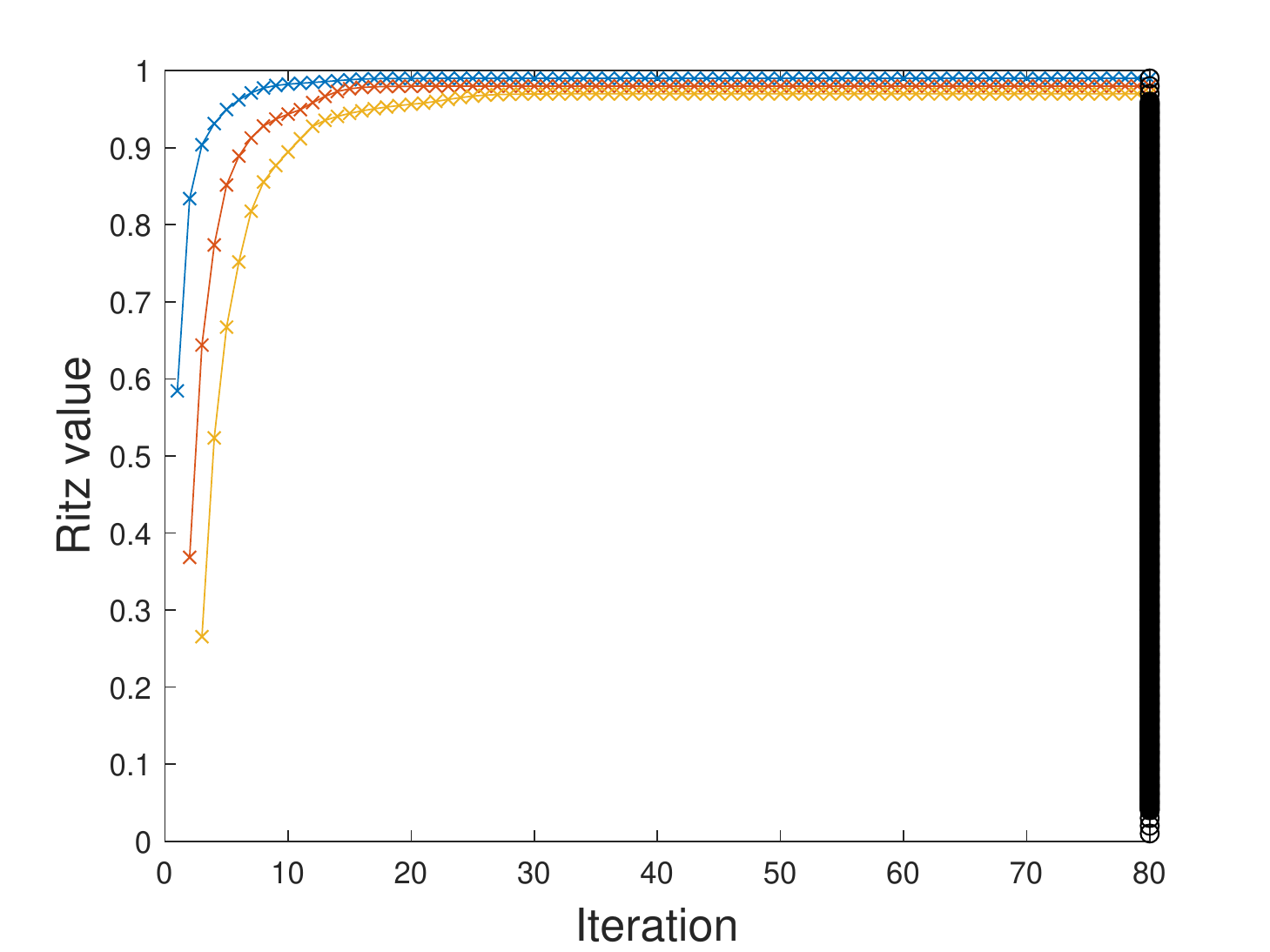}}
	\quad  %
	\subfloat[JBD]{\label{fig:4c}\includegraphics[width=0.4\textwidth]{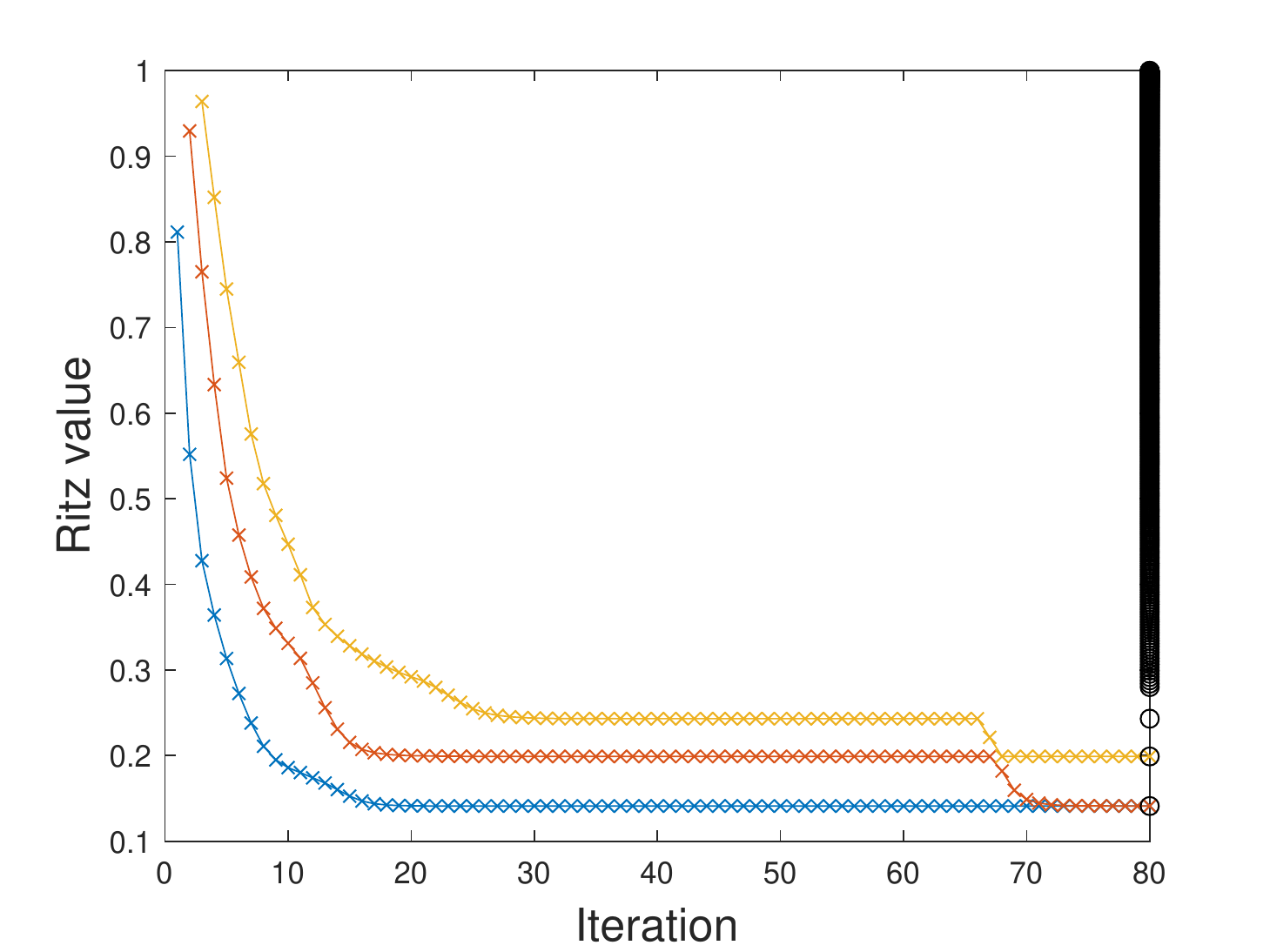}}	\subfloat[rJBD]{\label{fig:4d}\includegraphics[width=0.4\textwidth]{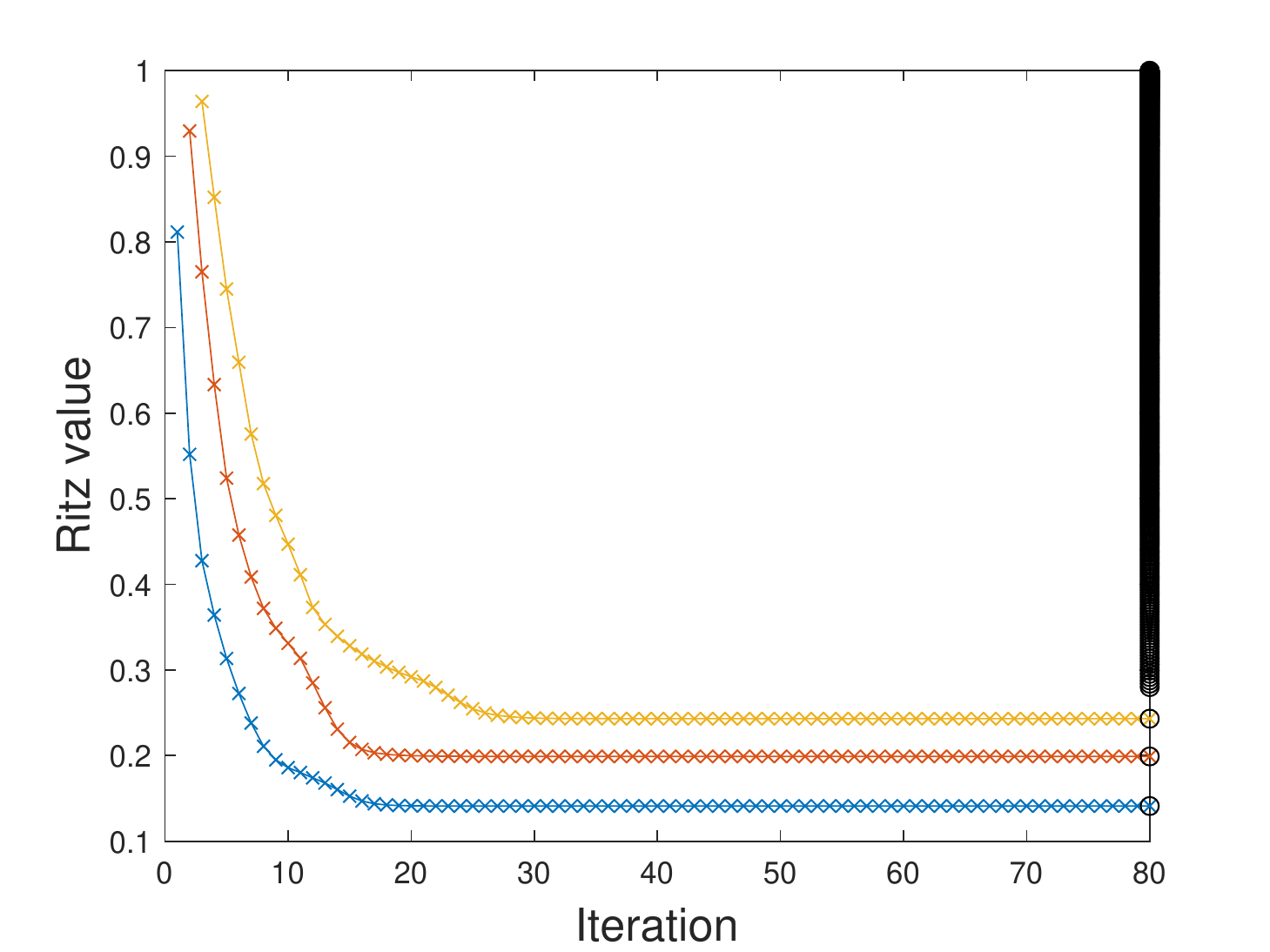}}
	\caption{Convergence of Ritz values. The top two use the SVD of $B_{k}$ to approximate $c_1, c_2$ and $c_3$; the bottom two use the SVD of $\widehat{B}_{k}$ to approximate $s_1, s_2$ and $s_3$. $\tau=10^{-10}$.}
	\label{fig4}
\end{figure}

\begin{figure}[htbp]
	\centering
	\subfloat[$\tau_1=10^{-10}$, $\tau_2=10^{-12}$]{\label{fig:5a}\includegraphics[width=0.41\textwidth]{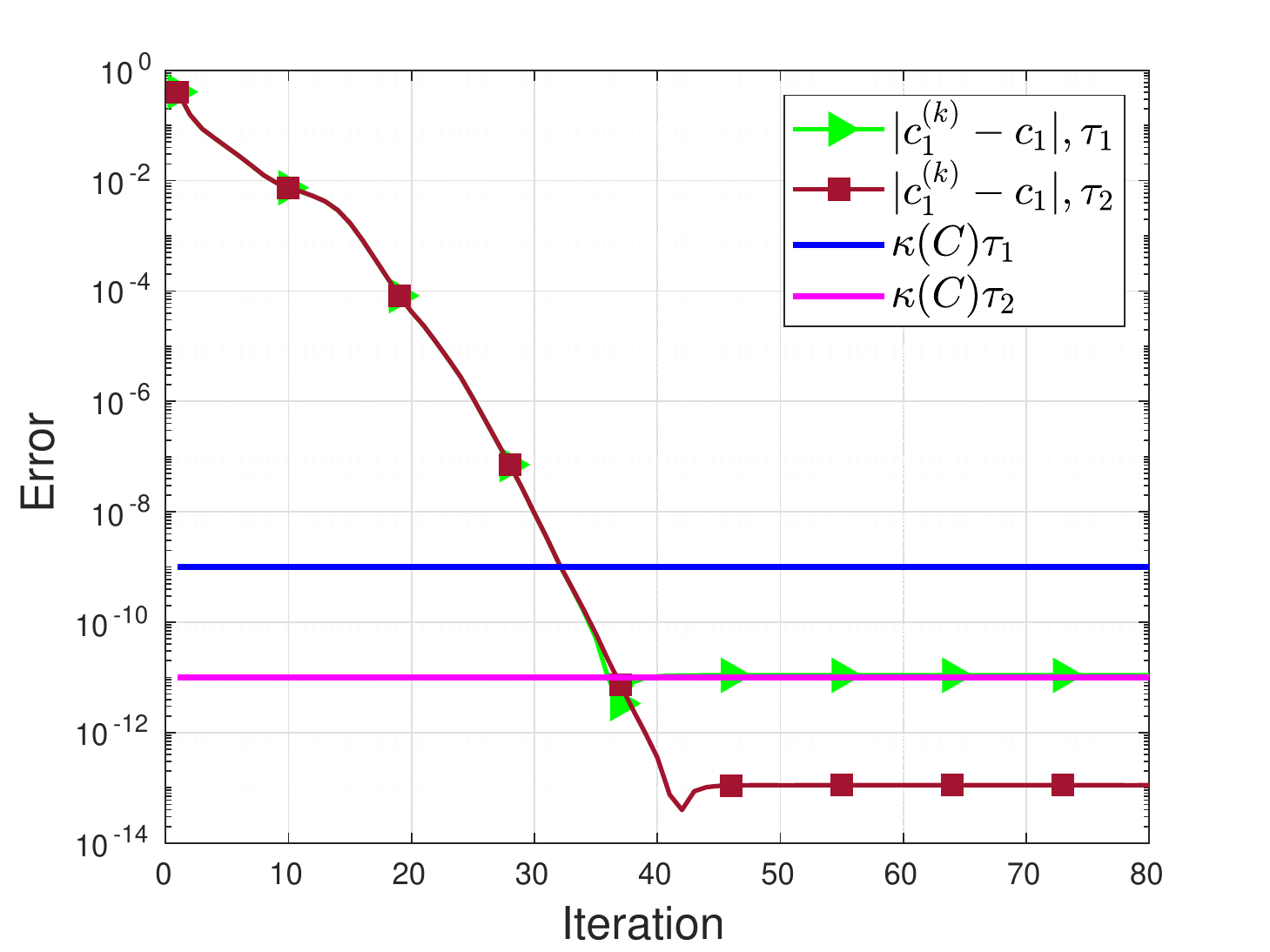}}
	\subfloat[$\tau_1=10^{-10}$, $\tau_2=10^{-12}$]{\label{fig:5b}\includegraphics[width=0.41\textwidth]{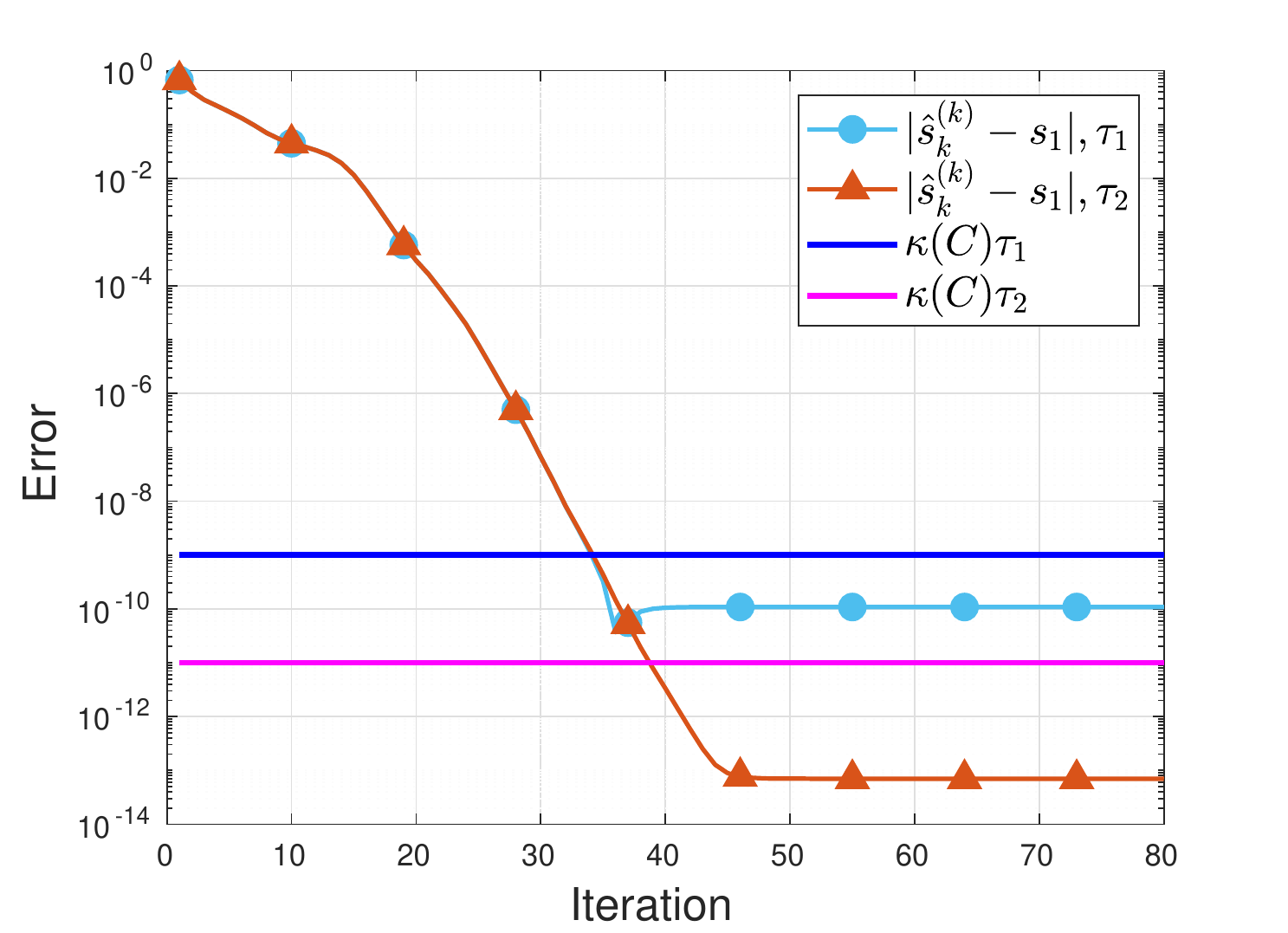}}
	\caption{The approximate accuracy of $c_{1}^{(k)}$ by the SVD of $B_k$ and $\hat{s}_{k}^{(k)}$ by the SVD of $\widehat{B}_k$.}
	\label{fig5}
\end{figure}

\Cref{fig4} depicts the convergence of Ritz values, which are the first three largest singular values of $B_k$ or the first three smallest singular values of $\widehat{B}_k$. Both of them are used to compute the first three largest generalized singular values of $\{A_2, B_2\}$ by approximating $c_1, c_2, c_3$ or $s_1, s_2, s_3$. The right vertical line indicates the values of $c_{i}$ or $s_i$ for $i=1,\dots,n$, and the left and right panels exhibit the convergence behaviors of JBD and rJBD for partial GSVD computations, respectively. For JBD we observe from \cref{fig:4a} that the second largest Ritz value suddenly jumps up at some iteration to continue converging to $c_1$, which is the so called ``ghost phenomenon" due to the loss of orthogonality of Lanczos vectors. This phenomenon can be observed more clearly in \cref{fig:4c}, which depicts the same convergence behavior as \cref{fig:4a}, since they approximate the same generalized singular values. For the rJBD method, the convergence of Ritz values becomes regular, which is in accordance with that in exact arithmetic that a simple generalized singular value is approximated by Ritz values without ghosts. This property can be explained by \Cref{thm4.1}. The convergence behaviors of approximations to the smallest generalized singular values are similar and we do not show them any more.

\Cref{fig5} shows the final accuracy of $c_{1}^{(k)}$ for approximating $c_1$ and of $\hat{s}_{k}^{(k)}$ for approximating $s_1$, where the stopping tolerance for inner iterations are $10^{-10}$ and $10^{-12}$. By \Cref{thm4.1}, $c_{1}^{(k)}$ will converge to $c_1$ with absolute error $|c_{1}^{(k)}-c_1|$ of order $\mathcal{O}(\kappa(C)\tau)$, which can be clearly observed in \cref{fig:5a}. \Cref{fig:5b} shows the convergence and absolute error $|\hat{s}_{k}^{(k)}-s_1|$ of $\hat{s}_{k}^{(k)}$ by the SVD of $\widehat{B}_k$, which is very similar to that of $c_{1}^{(k)}$ but with a slight difference, due to the relation \cref{3.18}.

\begin{figure}[htp]
	\centering
	\subfloat[$\tau_1=10^{-10}$]
	{\label{fig:6a}\includegraphics[width=0.41\textwidth]{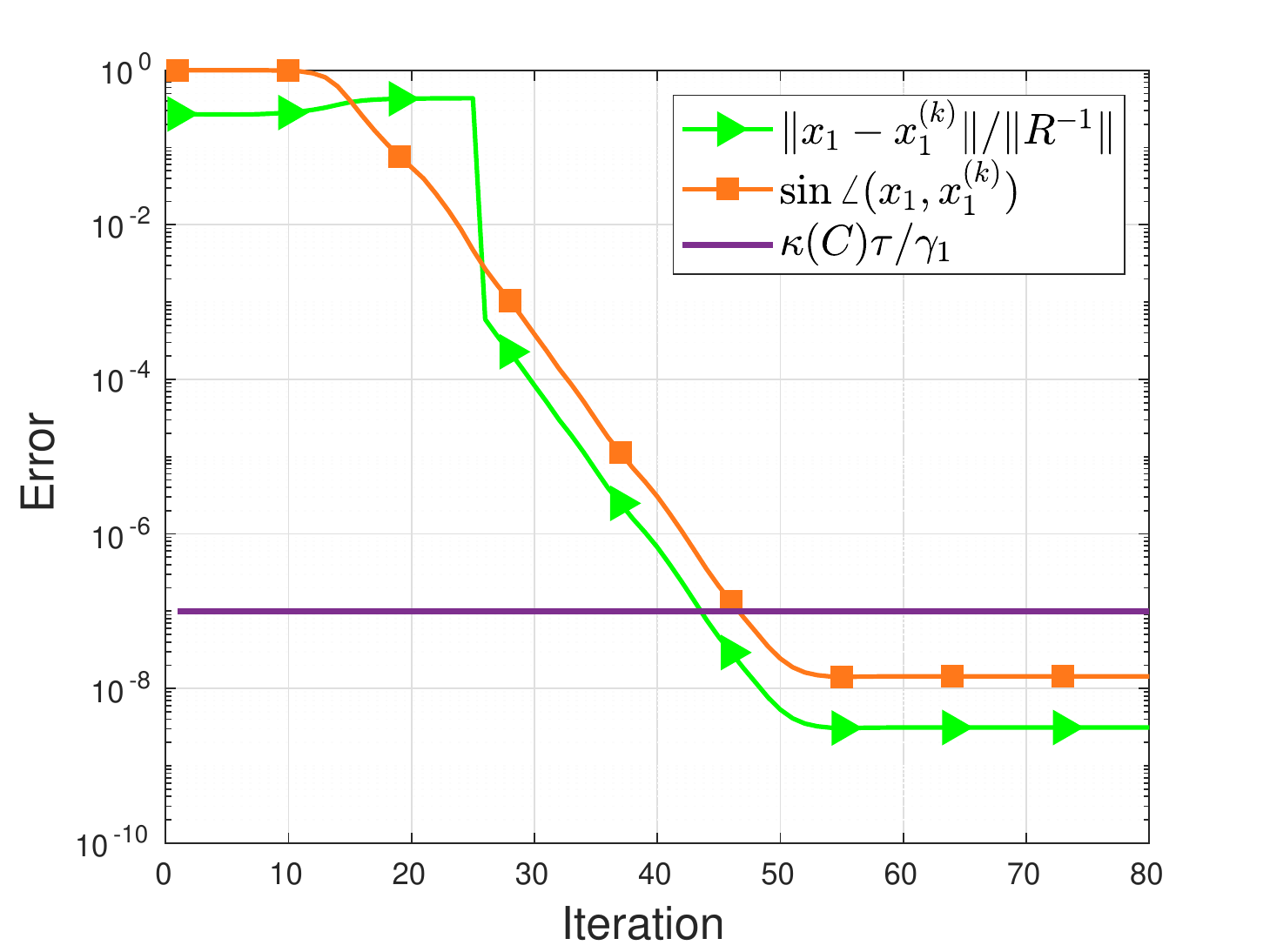}}	
	\subfloat[$\tau_1=10^{-10}$]
	{\label{fig:6b}\includegraphics[width=0.42\textwidth]{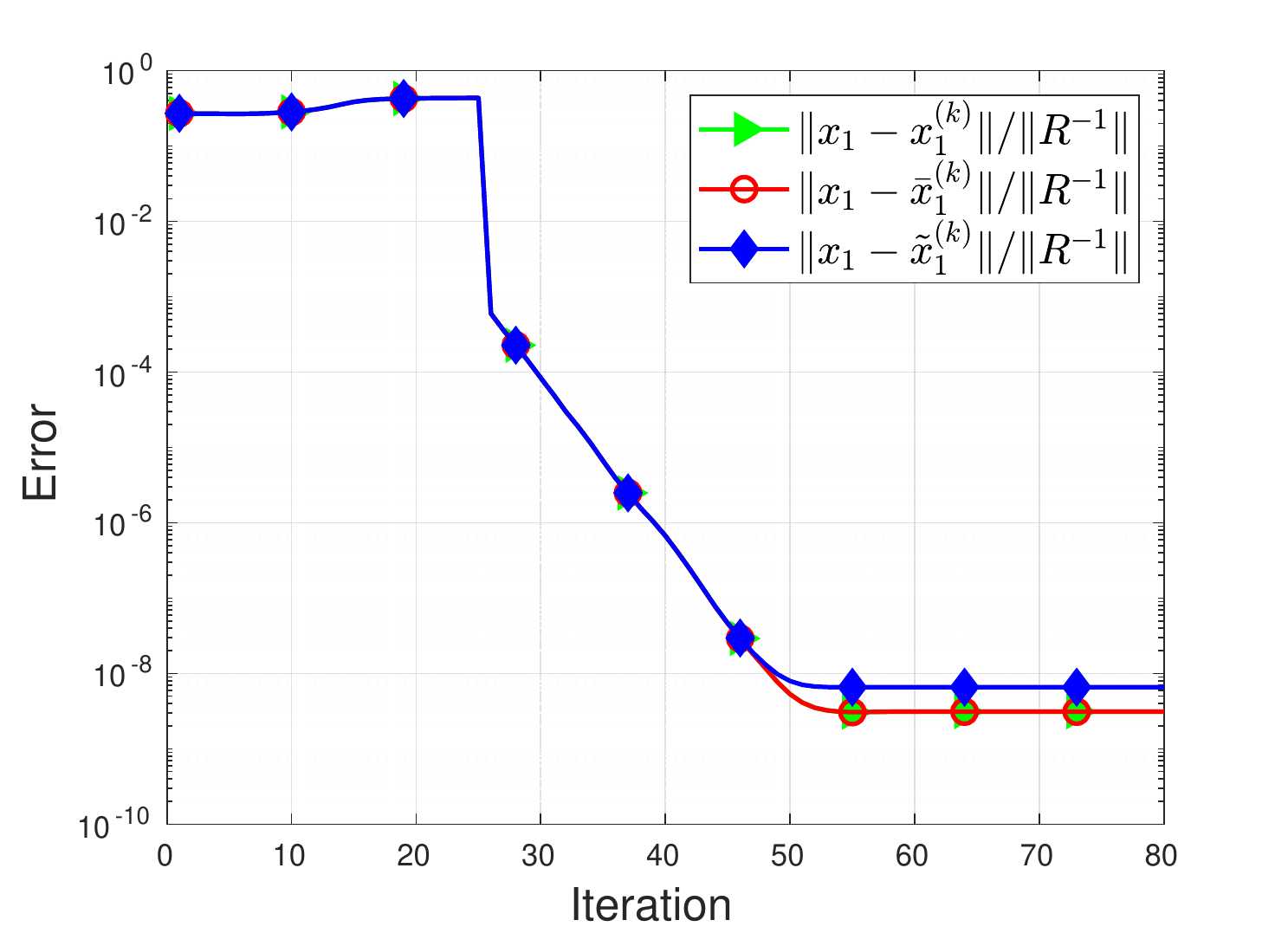}}
	\quad
	\subfloat[$\tau_1=10^{-12}$]
	{\label{fig:6c}\includegraphics[width=0.41\textwidth]{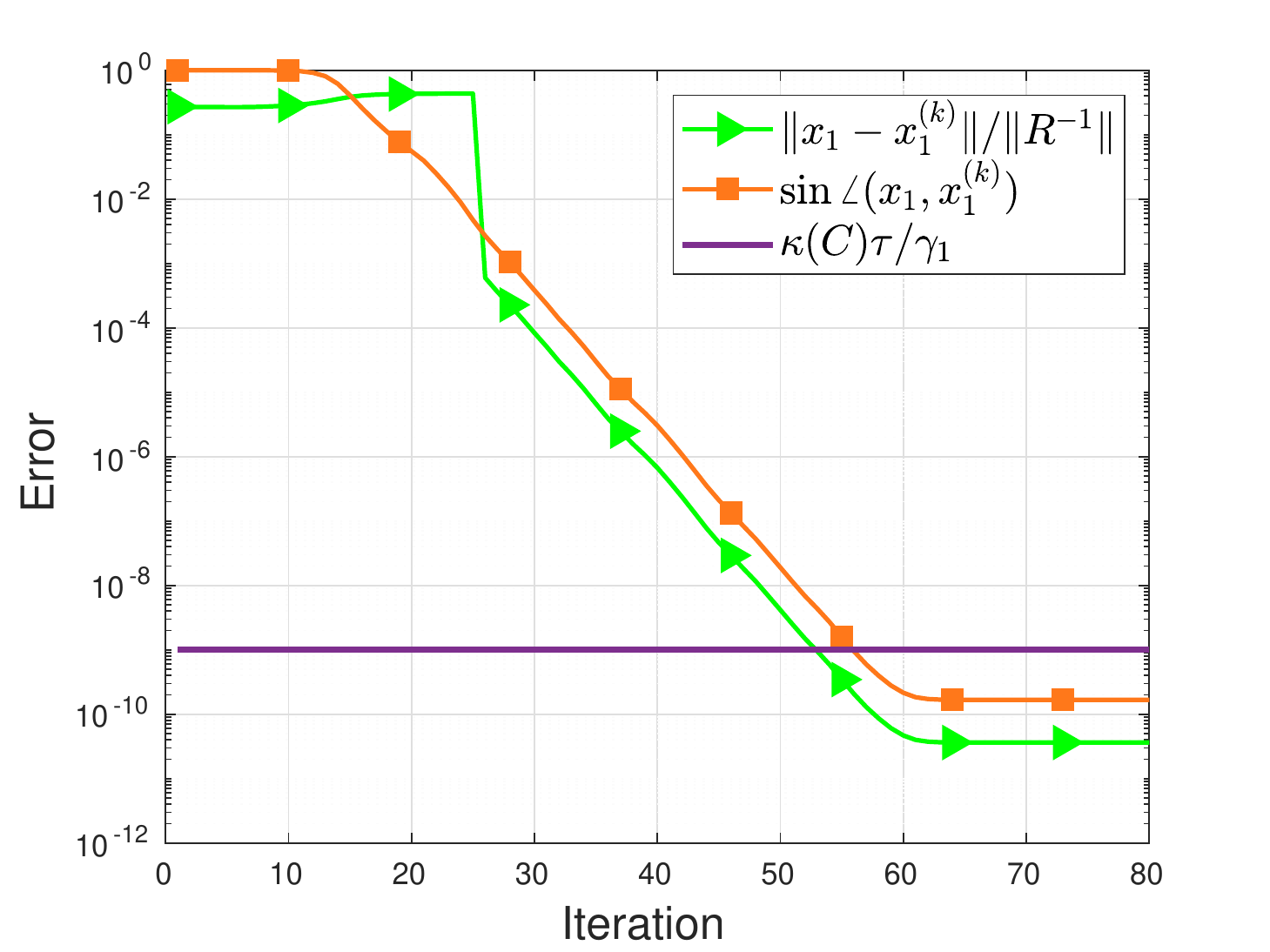}}	
	\subfloat[$\tau_1=10^{-12}$]
	{\label{fig:6d}\includegraphics[width=0.43\textwidth]{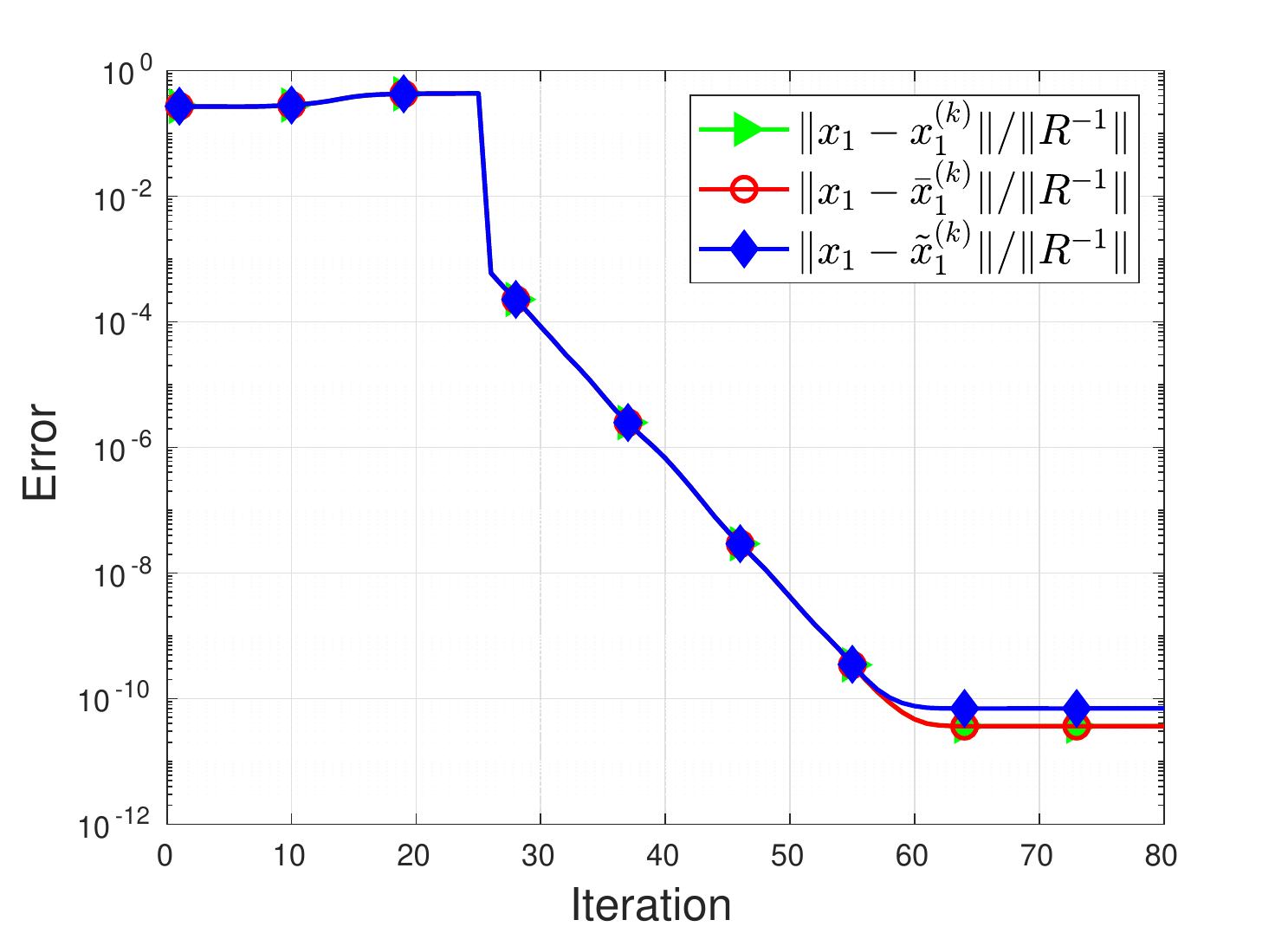}}
	\caption{The accuracy of approximations to $x_1$, where $\bar{x}_{1}^{(k)}$ and $\tilde{x}_{1}^{(k)}$ are computed from \cref{ls_x} by LSQR with stopping tolerance $\bar{\tau}_1=\tau$ and $\bar{\tau}_2=100\tau$.}
	\label{fig6}
\end{figure}

Finally we show the convergence and final accuracy of approximate generalized singular vectors. The stopping tolerance for inner iterations are $10^{-10}$ and $10^{-12}$, and the approximations to the right generalized singular vector $x_1$ corresponding to $c_1/s_1$ are obtained by the SVD of $B_k$, where $x_{1}^{(k)}$ is computed by explicitly using QR factorization of $C$ to solve \cref{ls_x} while $\bar{x}_{1}^{(k)}$ and $\tilde{x}_{1}^{(k)}$ are computed by solving \cref{ls_x} using LSQR with stopping tolerances $\bar{\tau}_1=\tau$ and $\bar{\tau}_2=100\tau$. The approximation errors are also measured using
\[\sin\angle(x_{1}, x_{1}^{(k)}), \ \  \sin\angle(x_{1}, \bar{x}_{1}^{(k)}), \ \
\sin\angle(x_{1}, \tilde{x}_{1}^{(k)}).\]
Since $\gamma_1\gg k\kappa(C)\tau$ for $k=1,\dots,80$, in this case \cref{rsv_accur} becomes $\|x_1-x_{1}^{(k)}\| /\|R^{-1}\| = \mathcal{O}(\sqrt{n}k\kappa(C)\tau/\gamma_1)$, we use $\kappa(C)\tau/\gamma_1$ as an upper bound on final accuracy of approximate vectors.

From \Cref{fig6} we can find that $x_{1}^{(k)}$ can eventually approximate $x_1$ with relative errors bounded by $\kappa(C)\tau/\gamma_1$, and the convergence rate is not affected too much by different values of $\tau$. The computed $\bar{x}_{1}^{(k)}$ with $\bar{\tau}_1=\tau$ has the same accuracy as $x_{1}^{(k)}$, while the accuracy of $\tilde{x}_{1}^{(k)}$ computed with $\bar{\tau}_1=100\tau$ is slightly worse. Although we do not show it here, using $\bar{\tau}=10\tau$ to solve \cref{ls_x} can also get a vector with the same accuracy as $x_{1}^{(k)}$. All values of $\bar{\tau}\in[0.1\tau, 10\tau]$ are feasible for computing $x_1$ .

\section{Conclusion and outlook}\label{sec7}
For the joint bidiagonalization of a matrix pair $\{A,L\}$, we have studied the influence of inaccuracy of inner iterations on the behavior of the algorithm. For a commonly used stopping criterion with tolerance $\tau$ to describe solution accuracy of inner least squares problems, we have shown that the orthogonality of Lanczos vectors will be lost where the loss rate depends on $\tau$ and the condition number of $C=(A^T,L^T)^T$. A reorthogonalized JBD process called rJBD is proposed to keep orthogonality of $\widetilde{V}_k$, and an error analysis has been carried out to build up connections between the rJBD process and Lanczos bidiagonalizations of $Q_A$ and $Q_L$, where a backward error bound about the bidiagonal reduction of $Q_A$ is established. The results of error analysis are used to investigate the convergence and accuracy of the computed GSVD components of $\{A,L\}$ by rJBD, which shows that the approximate generalized singular values can only reach an accuracy of order $\mathcal{O}(\kappa(C)\tau)$ and the accuracy of approximate right generalized singular vectors depends not only on the value of $\kappa(C)\tau$ but also on the gap between generalized singular values, while the convergence rate is not affected very much. Some numerical experiments are made to confirm the theoretical results.

For practical JBD based GSVD computations, our results can provide a guideline for choosing a proper computing accuracy of inner iterations in order to obtain approximate GSVD components with a desired accuracy. Besides, there are still some issues need to be considered to make the JBD method becoming a practical GSVD algorithm. For example, an efficient procedure is needed to extract information from $\mbox{span}(U_k)$ and $\mbox{span}(\widehat{U}_k)$ generated by rJBD to compute left generalized singular vectors. Another issue is how to accelerate convergence of the inner iteration. A proper preconditioner may be very useful for iteratively solving inner least squares problems. Numerical experiments show that an appropriate scaling factor transforming $\{A,L\}$ to $\{A,\gamma L\}$ has a positive effect on both the number of iterations needed by LSQR for inner iterations and the number of outer Lanczos iterations, thus scaling strategies are worth to be investigated. These issues constitute the subject of future research.

\section*{Acknowledgments}
I thank Dr. Long Wang, Prof. Guangming Tan and Prof. Weile Jia for their consistent support during this research. I am grateful to the anonymous referees and Editor Prof. Arvind Saibaba for their detailed reading of the manuscript and providing insightful comments that helped to improve the paper.

\bibliographystyle{siamplain}
\bibliography{references}

\end{document}